\documentclass[a5paper,12pt,reqno]{amsart}

\usepackage{amsmath,amssymb,amsthm}
\usepackage{latexsym}
\usepackage{color}
\usepackage{graphicx}
\usepackage{mathrsfs}
\usepackage{enumerate}
\usepackage{mathtools}
\usepackage[abbrev]{amsrefs}
\usepackage[T1]{fontenc}
\usepackage[colorlinks,
            linkcolor=blue,      
            anchorcolor=blue,  
            citecolor=red,       
            ]{hyperref}
\usepackage{hyperref}
\usepackage{amsthm}
\mathtoolsset{showonlyrefs}

\makeatletter

\@addtoreset{equation}{section}
\makeatother
\allowdisplaybreaks[4]

\theoremstyle{plain}
\newtheorem{thm}{Theorem}[section]
\newtheorem{lemm}[thm]{Lemma}

\newtheorem{cor}[thm]{Corollary}

\theoremstyle{definition}
\newtheorem{df}[thm]{Definition}
\newtheorem{rem}[thm]{Remark}

\newcommand{\supp}{{\rm supp}}

\newcommand{\nablah}{\nabla_{{\rm h}}}
\newcommand{\Deltah}{\Delta_{{\rm h}}}
\newcommand{\Deltahh}{\Delta^{{\rm h}}}

\newcommand{\xh}{x_{\rm h}}

\newcommand{\xih}{\xi_{\rm h}}

\newcommand{\p}{\partial}
\newcommand{\R}{\mathbb{R}}
\newcommand{\f}{\frac}
\newcommand{\Grad}{ \nabla }

\newcommand{\alphah}{\alpha_{\rm h}}
\newcommand{\alphav}{\alpha_3}
\newcommand{\Lh}{L_{x_{\rm h}}}
\newcommand{\Lv}{L_{x_3}}
\newcommand{\Gh}{G_{\rm h}}

\newcommand{\uh}{u_{\rm h}}
\newcommand{\vh}{v_{\rm h}}

\renewcommand{\div}{\operatorname{div}}

\renewcommand{\leq}{\leqslant}
\renewcommand{\geq}{\geqslant}

\begin{document}
\title[Anisotropic NS in a $3$D half-space]
{Large time behavior for solutions to the anisotropic Navier--Stokes equations in a $3$D half-space} 

\author[M.Fujii]{Mikihiro Fujii}
\author[Y.Li]{Yang Li}
\address[M.Fujii]{Graduate School of Science, Nagoya City University, Nagoya, 467-8501, Japan}
\email[M.Fujii]{fujii.mikihiro@nsc.nagoya-cu.ac.jp}
\address[Y.Li]{School of Mathematical Sciences and Center of Pure Mathematics, Anhui University, Hefei, 230601, People's Republic of China}
\email[Y.Li]{lynjum@163.com}
\date{\today}
\keywords{$3$D anisotropic Navier--Stokes equations, large time behavior, decay estimates, half-space}
\subjclass[2010]{35Q30, 35B40, 35Q35}
\begin{abstract}
	We consider the large time behavior of the solution to the anisotropic Navier--Stokes equations in a $3$D half-space. 
    Investigating the precise anisotropic nature of linearized solutions, 
    we obtain the optimal decay estimates for the nonlinear global solutions in anisotropic Lebesgue norms. In particular, we reveal the enhanced dissipation mechanism for the third component of velocity field. 
    We notice that, in contrast to the whole space case, some difficulties arises on the $L^1(\mathbb{R}^3_+)$-estimates of the solution due to the nonlocal operators appearing in the linear solution formula.
    To overcome this, we introduce suitable Besov type spaces and employ the Littlewood--Paley analysis on the tangential space.
\end{abstract}
\maketitle

\tableofcontents

\section{Introduction}
\subsection{Background and motivations}
In this paper, we are interested in the time evolution of anisotropic incompressible Navier--Stokes equations in the half-space, where the effect of vertical viscosity is negligible compared with horizontal viscosity. Such a system of equations is often used in modelling geophysical fluids. For more physical backgrounds and applications on this model, we refer to the classical monograph of Pedlosky \cite{Ped}. In Eulerian coordinates, the governing equations take the form 
\begin{align}\label{eq1}
\begin{cases}
    \partial_t u -  \Delta_{\rm h} u + (u \cdot \nabla)u + \nabla p = 0, & \qquad t>0,x\in \mathbb{R}^3_{+},\\
    \div u = 0, & \qquad t\geq0, x \in \mathbb{R}^3_{+},\\
    u=0, & \qquad t \geq 0, x \in  \p \mathbb{R}^3_{+},\\
    u(0,x) = u_0(x), & \qquad x \in \mathbb{R}^3_{+}.
\end{cases}  
\end{align}
Here, $u:=(u_1,u_2,u_3)$ and $p$ are the velocity field and scalar pressure respectively. $\Deltah:= \p_1^2+\p_2^2$ means the Laplacian operator acting only on the horizontal variables.

Let us briefly recall some known results on the anisotropic Navier--Stokes equations. In their pioneering work \cite{CDGG01}, Chemin et al. proved the existence of local solution for large initial data and the global existence for small initial data in $H^{0,s}(\R^3)$ with $s>1/2$. Iftimie \cite{If02} proved the uniqueness of solutions in $H^{0,s}(\R^3)$ with $s>1/2$. In another excellent work, Paicu \cite{Pai05} proved the global well-posedness in the $L^2$-based anisotropic Besov space $\mathcal{B}^{0,\f{1}{2}}(\R^3)$. Chemin and Zhang \cite{CZ07} and Zhang and Fang \cite{ZF08} further established the global well-posedness in the general $L^p$-based anisotropic Besov space ${B}^{-1+\f{2}{p},\f{1}{2}}_p(\R^3)$ with $2\leq p <\infty$. We refer to \cites{LPZ20,LiuZh20,PZ11,Z09} and the references therein for more results on the global well-posedness for anisotropic Navier--Stokes equations in $\R^3$. With periodic boundary conditions, Paicu \cite{Pai0502} proved the global existence of solutions under additional regularity conditions on certain components of the initial data. Iftimie and Planas \cite{IfPl06} considered the anisotopic Navier--Stokes equations on a half-space supplemented with the boundary condition $u_3|_{\p \R^3_+}=0$ and proved that solutions of the anisotropic Navier--Stokes equations are limits of the Navier--Stokes equations with Navier boundary conditions. 
Paicu and Raugel \cite{PaiRo09} studied the global and local existence and uniqueness of
solutions on a bounded cylindrical domain.

In a recent work, Ji et al. \cite{JWY21} proved the $L^2$-decay estimates of solutions and its first order derivatives in $\R^3$; Xu and Zhang \cite{XZ22} refined the previous result and furthermore proved the enhanced dissipation mechanism for the third component of velocity field\footnote{Roughly speaking, it is proved that the horizontal components $(u_1,u_2)$ decay like the 2D heat kernel, while the third component $u_3$ decays like the 3D heat kernel.}. Later, Fujii \cite{F21} proved the decay estimates in the general $L^p$-norms for any $1\leq p \leq \infty$ and furthermore obtained the asymptotic expansions of solutions; see also \cite{L22} for the extension to the anisotropic incompressible magnetohydrodynamic equations. Very recently, Ji et al. \cite{JTW23} obtained the global stability and asymptotic behavior of solutions in $\mathbb{T}^2 \times \R$, with the velocity field converging to a non-trivial steady state; see also Dong et al. \cite{DWXZ21} for 2D anisotropic Navier--Stokes equations with horizontal dissipation.

In contrast to the whole space case, where there are many previous studies on \eqref{eq1}, there seems no previous studies on the mathematical analysis of the anisotropic Navier--Stokes equations in the physical space with boundary. \emph{It is an interesting problem to study how the interaction between the anisotropic structure of \eqref{eq1} and the boundary condition}. 
In the context of $3$D half-space, we shall give an affirmative answer to the problem above. It should be emphasized that the mechanism of enhanced dissipation for the third component of velocity field will be verified. We also stress that the known techniques previously used for anisotropic Navier--Stokes equations in the whole space cannot be directly adapted to our setting; specific challenges will be presented later. Finally, we notice that there are fruitful decay results on the Navier--Stokes equations and other fluid systems in the half-space, see \cites{BM88,FM01,GMS99,Han10,Han11,Han16,HH12,K89}. However, all of these results considered the full dissipation and thus the methods cannot be directly applied for our anisotropic equations.

\subsection{The main results}
Before giving the main results of this paper, we shall introduce several notations. For any $1\leq p,q \leq \infty$, $\|\cdot  \|_{L^p}$ denotes the norm of $L^p(\R^3_{+})$; $\| \cdot \|_{\Lv^q\Lh^p(\R^3_{+})     }$ denotes the norm of anisotropic Lebesgue space $\Lv^q\Lh^p(\R^3_{+})    $, which is defined by $ L^q(\R_{+};L^p(\R^2_{\xh}) ) $. For any $v\in \R^3$, we denote by $v_{\rm h}:=(v_1,v_2)$ the horizontal components. We denote by $(-\Deltah)^{-\f{s}{2}}$ or $|\nablah|^s$ the pseudo differential oparator with symbol $|\xih|^s$. Let us define the following Chemin--Lerner type space:
\begin{align*}
  \mathcal{L}^{p,q}_{\sigma}(\mathbb{R}^3_+)
    :=   &
    \widetilde{L^q_{x_3}}(\mathbb{R}_{x_3>0};\dot{B}_{p,\sigma}^0(\mathbb{R}^2)) \,\,\,\,\text{  with the norm  }   \\
 \| f \|_{ \mathcal{L}^{p,q}_{\sigma}(\mathbb{R}^3_+) } :=    & \left\{\begin{aligned}
& \left( \sum_{j \in \mathbb{Z}} \| \Delta_j^{\rm h} f \|_{  L^q_{x_3}(\mathbb{R}_{x_3>0};L^p(\mathbb{R}^2_{\xh}  ))     }^{\sigma}           \right)^{ \frac{1}{\sigma} } \,\,\,\, \text{   if   } 1\leq \sigma< \infty ,\\
&  \sup_{j \in \mathbb{Z} }  \| \Delta_j^{\rm h} f \|_{  L^q_{x_3}(\mathbb{R}_{x_3>0};L^p(\mathbb{R}^2_{\xh}  ))     } \,   \,\,\,\,\,\,\,\,\,\,\,\,\,\,\,\,\, \text{   if   }  \sigma= \infty ,   
\end{aligned}\right.
\end{align*}
where $\{\Delta_j^{\rm h}\}_{j \in \mathbb{Z}}$ are the dyadic operators for homogeneous Littlewood--Paley decomposition on $\mathbb{R}^2_{x_{\rm h}}$.
Moreover, we use the abbreviation $\mathcal{L}^{p}_{\sigma}(\mathbb{R}^3_+):= \mathcal{L}^{p,p}_{\sigma}(\mathbb{R}^3_+)$.
We set 
\begin{align}\nonumber 
&
X^s (\R^3_+) = \left\{ f \in {H^s}(\mathbb{R}^3_+) \cap \mathcal{L}_{\infty}^{1}(\mathbb{R}^3_+)\cap \mathcal{L}_{\infty}^{1,\infty}(\mathbb{R}^3_+)\ ;\ \partial_3 f \in \mathcal{L}_{\infty}^{1}(\mathbb{R}^3_+)\cap \mathcal{L}_{\infty}^{1,\infty}(\mathbb{R}^3_+)  \right\}
\end{align}
with the norm 
\begin{align}
\| f \|_{X^s(\mathbb{R}^3_+)} :=
\| f \|_{H^s(\mathbb{R}^3_+)} + \sum_{\alpha_3=0}^1 \| \p_3^{\alphav} f \|_{\mathcal{L}_{\infty}^{1}\cap \mathcal{L}_{\infty}^{1,\infty} (\mathbb{R}^3_+)  }  .
\end{align}

We are ready to give the first result of this paper. 
\begin{thm}\label{main-theorem}
Let $s\geq 5$ be an integer. Then there exists a positive constant $\delta$ such that any $u_0 \in X^s(\R^3_{+})$ with $\div  u_0=0,  u_{0} |_{\p \R^3_+}=0$ and $\| u_0 \|_{X^s(\R^3_{+})}  \leq  \delta $ yields a unique global solution $u \in C((0,\infty); X^s(\R^3_{+}) )$ to the problem \eqref{eq1}.

Moreover, we have the optimal decay estimates. There exists a generic positive constant $C$ such that for any $\alpha=(\alphah,\alpha_3) \in (\mathbb{N}\cup \{ 0 \} )^2 \times (\mathbb{N}\cup \{ 0 \} ) $ with $|\alpha| \leq 1 $, the following estimates hold: 
\begin{itemize}
\item{ Optimal decay estimates for the horizontal components:
\begin{align} \label{decay-h-1}  
 \| \nabla^{\alpha} \uh(t)   \| _{\Lv^q\Lh^p (\R^3_+) }  \leq     C   (1+t)^{ -(1-\f{1}{p})  -\f{|\alphah|}{2}  } \| u_0 \|_{X^s(\R^3_+)   }   
\end{align}
for all $t>0$, $2\leq  p,q \leq \infty$; 
\begin{align} \label{decay-h-2}  
 \| \nabla^{\alpha} \uh(t)   \| _{\Lv^q\Lh^p (\R^3_+) }  \leq     C    t^{ -(1-\f{1}{p})  -\f{|\alphah|}{2}         } \| u_0 \|_{X^s(\R^3_+)   }     
\end{align}  
for all $t \geq 1$
and the following cases 
\begin{align}\label{pq-range-h}
\begin{cases}  
    \text{Case I:} &    1\leq  p,q < 2 \quad \text{with} \quad  (1-\f{1}{p}) + \f{|\alphah|}{2}    >0 ,     \\
      \text{Case II:} &  1\leq q<2,
       2\leq p\leq \infty ,   \\
    \text{Case III:} &  1\leq p<2,
       2\leq q\leq \infty  \quad \text{with} \quad  (1-\f{1}{p}) + \f{|\alphah|}{2}    >0 . 
\end{cases}  
\end{align}   
}
\item{ Optimal decay estimates for the third component: 
\begin{align} \label{decay-3-1}  
 \| \nablah^{\alphah}  u_3(t) \|_{\Lv^q\Lh^p (\R^3_+)}  \leq      C  (1+t)^{ - (1-\f{1}{p})-\f{1}{2}(1-\f{1}{q})   -\f{|\alphah|}{2}         }  \| u_0 \|_{X^s(\R^3_+)   }       
\end{align} 
for all $t>0$, $2\leq  p,q \leq \infty$;
\begin{align} \label{decay-3-2}  
 \| \nablah^{\alphah}  u_3(t) \|_{\Lv^q\Lh^p (\R^3_+)}  \leq      C  t^{ - (1-\f{1}{p})-\f{1}{2}(1-\f{1}{q})   -\f{|\alphah|}{2}         }  \| u_0 \|_{X^s(\R^3_+)   }     
\end{align} 
for all $t>0$ and the following cases 
\begin{align}\label{pq-range-3}
\begin{cases}  
    \text{Case I:} &    1\leq  p,q <2
   \quad \text{with} \quad  (1-\f{1}{p})+\f{1}{2}(1-\f{1}{q})   +\f{|\alphah|}{2}  >0 ,     \\
      \text{Case II:} &  1\leq q<2,
       2\leq p\leq \infty ,   \\
    \text{Case III:} &  1\leq p<2,
       2\leq q\leq \infty  .
\end{cases}  
\end{align}  
}
\end{itemize}
\end{thm}

\begin{rem}
As we can see from \eqref{decay-h-1}-\eqref{decay-h-2} that the horizontal components of velocity field decay like the 2D heat kernel as time goes to infinity. On the contrary, \eqref{decay-3-1}-\eqref{decay-3-2} reveal that the vertical component decays like the 3D heat kernel. Therefore, we have verified the enhanced dissipation mechanism of anisotropic Navier--Stokes equations in the half-space case.  
\end{rem}

Let us mention the characteristic of Theorem \ref{main-theorem}.
In general, the linear analysis is a key ingredient to obtain the optimal decay estimate of global small solutions.
To analyze the linear system, we make use of the technique of Ukai \cite{U87} to provide the explicit solution formula.
This solution formula involves some singular integral operators such as the Riesz transforms for the tangential space, which prevent us from analyzing the time-evolution of $L^1(\mathbb{R}^3_+)$-norm and $L^{\infty}(\mathbb{R}^3_+)$-norm of the linear solution.
To overcome this difficulty, we focus on the fact that the Riesz transforms are bounded in homogeneous Besov spaces and employ the Besov type spaces to obtain the estimate of solutions in a $L^1$-type space 
\begin{align}
    \|u(t)\|_{\mathcal{L}^{1}_{\infty}(\mathbb{R}^3_+)}
    <\infty, \qquad t \geq 0.
\end{align}
This is the reason why we introduced the Chemin--Lerner type space $\mathcal{L}^{p,q}_{\sigma}(\mathbb{R}^3_+)$ above.

On the other hand, Theorem \ref{main-theorem} does not include the uniform $L^1$-estimate 
\begin{align}
    \sup_{t\geq 0}\| u(t) \|_{L^1(\mathbb{R}^3_+)} < \infty.
\end{align}
In the following theorem, we show that the above $L^1$-estimate holds if we impose a slightly stronger assumption on the initial data.

\begin{thm}\label{main-theorem-2}
Let $s\geq 5$ be an integer and let $u \in C((0,\infty); X^s(\R^3_{+}) )$ be the solution to the problem \eqref{eq1} with $u_0 \in X^s(\R^3_{+})$ satisfying $\div  u_0=0, u_{0} |_{\p \R^3_+}=0$ and $\| u_0 \|_{X^s(\R^3_{+})}  \leq  \delta $, where $\delta$ is the constant appearing in Theorem \ref{main-theorem}.
Assume further that $u_0 \in \mathcal{L}^1_1(\R^3_+)$. Then 
there exists a generic positive constant $C$ such that
\begin{align} \label{decay-u3-L1L1}  
    \begin{split}
    &
    \|  \uh(t) \|_{L^1 (\R^3_+)} 
    \leq    
    C  
    \| u_0 \|_{X^s\cap \mathcal{L}^1_1 (\R^3_+)   } ,\\
    &\|  u_3(t) \|_{\mathcal{L}^1_1 (\R^3_+)} 
    \leq    
    C  
    \| u_0 \|_{X^s\cap \mathcal{L}^1_1 (\R^3_+)   } 
    \end{split}
\end{align}
for all $t>0$.
\end{thm}
\begin{rem}
It follows from the continuous embedding $\mathcal{L}_1^1(\mathbb{R}^3_+) \hookrightarrow L^1(\mathbb{R}^3_+)$ that Theorem \ref{main-theorem-2} implies
\begin{align}
    \| u(t) \|_{L^1(\R^3_+) }
    \leq
    C  
    \| u_0 \|_{X^s\cap \mathcal{L}^1_1 (\R^3_+)}.
\end{align}
On the contrary, it is not clear whether or not the $\| u(t) \|_{\mathcal{L}_1^1(\mathbb{R}^3_+)}$ is also bounded in time. 
The result is reasonable in the sense that the vertical component of velocity field enjoys the enhanced dissipation mechanism.
\end{rem}

The rest of this paper is devoted to the proof of Theorems \ref{main-theorem}-\ref{main-theorem-2} and is arranged as follows. In Section \ref{line-anal}, we derive the solution formula for nonlinear equations and present the asymptotic analysis for the anisotropic Stokes semigroup. 
In Section \ref{decay-geq-2}, we analyze the decay estimates for all the Duhamel terms in case of $2\leq p,q \leq \infty$. 
In Section \ref{main-thm-1}, we complete the proof of Theorem \ref{main-theorem} in case of $2\leq p,q \leq \infty$ via the bootstrapping argument. 
In Sections \ref{decay-less-2}-\ref{main-thm-2}, we use the duality argument to handle the decay estimates in case of $1\leq p,q <2$ or $1\leq p<2,
       2\leq q\leq \infty$ or $1\leq q<2,
       2\leq p \leq \infty$. 
       In Section \ref{proof-thm-1.3}, we show the uniform boundedness of velocity field in the endpoint space $L^1 (\R^3_+)$. 
For the convenience of the reader, we collect some basic tools from the Littlewood--Paley theory in the Appendix.

\section{Linear analysis}\label{line-anal}
We consider the following linearized equations on the half-space $\R^3_+$: 
\begin{align}\label{eq:lin-1}
    \begin{cases}
    \partial_t u - \Deltah u + \nabla p = \div f, \qquad & t > 0, x \in \mathbb{R}^3_+,\\
    \div u = 0, \qquad & t \geq 0, x \in \mathbb{R}^3_+,\\
    u = 0,\qquad & t \geq 0, x \in \partial\mathbb{R}^3_+,\\
    u(0,x) = u_0(x),\qquad & x \in \mathbb{R}^3_+,
    \end{cases}
\end{align}  
where $f=(f_{k,\ell}(t,x))_{1 \leq k,\ell \leq 3}$ is a given external force satisfying $f_{k,\ell}=f_{\ell,k}$.   
\subsection{Solution formula}
The aim of this subsection is to derive an explicit formula for the solution to \eqref{eq:lin-1} by using \emph{only} Fourier transform on $\mathbb{R}^2$. 

Let $f^{\rm w}=(f^{\rm w}_{k,\ell}(t,x))_{1 \leq k,\ell \leq 3}$ be the extension of $f$ defined as
\begin{align}
    f^{\rm w}_{k,\ell}(t,x)
    :=
    \begin{cases}
        f_{k,\ell}(t,x) \qquad & (x_3  \geq 0),  \\
        (-1)^{\delta_{\ell,3} }f_{k,\ell}(t,\xh,-x_3) \qquad & (x_3<0),
    \end{cases}
\end{align}
for $k,\ell=1,2,3$.
Here, we note that $f_{k,3}^{\rm w} \neq f_{3,k}^{\rm w}$ in general.
We consider the linear equations \eqref{eq:lin-1} on the whole space $\mathbb{R}^3$: 
\begin{align}
    \begin{cases}
    \partial_t u^{\rm w} - \Deltah u^{\rm w} + \nabla p^{\rm w} = \div f^{\rm w}, \qquad & t > 0, x \in \mathbb{R}^3,\\
    \div u^{\rm w} = 0, \qquad & t > 0, x \in \mathbb{R}^3,\\
    u^{\rm w}(0,x) = 0, \qquad & x \in \mathbb{R}^3.
    \end{cases}
\end{align}
Then, $v:= u - u^{\rm w}$ and $q:= p - p^{\rm w}$ should solve the following equations on the half-space $\R^3_+$:
\begin{align}
    \begin{cases}
    \partial_t v - \Deltah v + \nabla q = 0, \qquad & t > 0, x \in \mathbb{R}^3_+,\\
    \div v = 0, \qquad & t \geq 0, x \in \mathbb{R}^3_+,\\
    v= - u^{\rm w} ,\qquad & t \geq 0, x \in \partial\mathbb{R}^3_+,\\
    v(0,x) = u_0(x),\qquad & x \in \mathbb{R}^3_+.
    \end{cases}
\end{align}

Next, we aim to derive an explicit formula to the equations above by adapting the pioneering work of Ukai \cite{U87}. Applying $\div$ to the first equation of above, we see that $\Delta q = 0$, which implies $(|\xih|-\partial_{x_3})(|\xih|+\partial_{x_3})\mathscr{F}_{\mathbb R^2}[q(\cdot,x_3)]=0$.
By this and the fact that $\mathscr{F}_{\mathbb R^2}[q(\cdot,x_3)]$ is an element of  $\mathscr{S}'(\mathbb{R}^2)$, it holds $(|\xih|+\partial_{x_3})\mathscr{F}_{\mathbb R^2}[q(\cdot,x_3)]=0$. Let 
\begin{align}
    w_3
    :=(|\nablah|+\partial_{x_3})v_3
    =|\nablah|v_3 - \nablah \cdot \vh.  
\end{align}
Then, $w_3$ should solve $\partial_t w_3 - \Deltah w_3 = 0$ with the initial data $(|\nablah|+\partial_{x_3})u_{0,3}$, and thus $w_3(t)=e^{t\Deltah}(|\nablah|+\partial_{x_3})u_{0,3}$. 
Thus, $v_3$ obeys 
\begin{align}
    \begin{cases}
        (  |\nablah|  +\partial_{x_3})v_3 = e^{t\Deltah}(|\nablah|+\partial_{x_3})  u_{0,3} , \qquad &x_3 >0,\\
        v_3(t,\xh,0)= -u^{\rm w}_3(t,\xh,0)
        ,\ & x_3=0.
    \end{cases}
\end{align}
Solving this equation, we have
\begin{align}             
    v_3(t,x) ={}&  
     -e^{-x_3|\nablah|}u^{\rm w}_3(t,\xh,0)
    +
    U e ^{t \Delta_{\rm h}} (u_{0,3}- S \cdot u_{0,\rm h} )(x) .   \label{v3-1} 
\end{align}
Here $S=(S_{x_1},S_{x_2})=\nablah(-\Deltah)^{-\frac{1}{2}}$ are the Riesz transforms on $\R^2$, and $U$ is defined by
\begin{align}
    {Uf} (x)=
    |\nablah| \int_0^{x_3}  e^{-(x_3-y_3)|\nablah| }  {f} (\xh, y_3) dy_3. 
\end{align}
Let $w_{\rm h}:=v_{\rm h} + Sv_3$. 
Since $\nablah q + S \partial_{x_3} q =S(|\nablah|+\partial_{x_3})q =0$, we see that $\partial_t w_{\rm h} - \Deltah w_{\rm h} = 0$ with the initial data $u_{0,\rm h} + Su_{0,3}$, which yields
$v_{\rm h}(t) + S v_3(t) = w_{\rm h}(t) = e ^{ t \Delta_{\rm h} }  (u_{0,\rm h}  + S u_{0,3}  ) $ and
\begin{align}             
    v_{\rm h}(t,x)={}&    
    e ^{ t \Delta_{\rm h} }  (u_{0,\rm h}  + S u_{0,3}  ) (x)
    -S v_3(t,x).     
    \label{vh-1}
\end{align}
Hence, we obtain 
\begin{align}             
    &
    u_3(t,x) ={}  
    u_3^{\rm w}(t,x)
    -
     e^{-x_3|\nablah|}u^{\rm w}_3(t,\xh,0)
    + 
    U e ^{t \Delta_{\rm h}} (u_{0,3}- S \cdot u_{0,\rm h} )(x) ,   \label{u3-1}  \\   
    &
    \begin{aligned}
    u_{\rm h}(t,x)={}&
    u_{\rm h}^{\rm w}(t,x)
    +   
    e ^{ t \Delta_{\rm h} }  (u_{0,\rm h}  + S u_{0,3}  ) (x)\\
    &
    +
    S
     e^{-x_3|\nablah|}u^{\rm w}_3(t,\xh,0)
    -
    SU e ^{t \Delta_{\rm h}} (u_{0,3}- S \cdot u_{0,\rm h} )(x).
    \end{aligned}     
    \label{uh-1}
\end{align}

In what follows, we consider the formula for 
\begin{align}
    u^{\rm w}(t) = \int_0^t e^{(t-\tau)\Deltah} \mathbb{P}\div f^{\rm w}(\tau) d\tau.
\end{align}
By the Fourier transform on $\mathbb{R}^3$, we have
\begin{align}
    \mathscr{F}_{\mathbb{R}^3}[u^{\rm w}_3](t,\xi)
    ={}&
    \int_0^t 
    e^{-(t - \tau)|\xih|^2}
    \sum_{\ell,m=1}^3
    \left( \delta_{3,\ell} - \frac{\xi_{3}\xi_{\ell}}{|\xi|^2} \right)
    i\xi_m 
    \mathscr{F}_{\mathbb{R}^3}[f^{\rm w}_{\ell,m}](\tau,\xi) d\tau\\
    ={}&
    \int_0^t 
    e^{-(t - \tau)|\xih|^2}
    \sum_{m=1}^3
    i\xi_m 
    \mathscr{F}_{\mathbb{R}^3}[f^{\rm w}_{3,m}](\tau,\xi) d\tau\\
    &-
    \int_0^t 
    e^{-(t - \tau)|\xih|^2}
    \sum_{\ell,m=1}^3
    \frac{i\xi_{3}\xi_{\ell}\xi_m}{|\xi|^2}
    \mathscr{F}_{\mathbb{R}^3}[f^{\rm w}_{\ell,m}](\tau,\xi) d\tau\\
    ={}&
    \int_0^t 
    e^{-(t - \tau)|\xih|^2}
    \sum_{m=1}^2
    i\xi_m 
    \mathscr{F}_{\mathbb{R}^3}[f^{\rm w}_{3,m}](\tau,\xi) d\tau\\
    &
    +
    \int_0^t 
    e^{-(t - \tau)|\xih|^2}
    i\xi_3 
    \mathscr{F}_{\mathbb{R}^3}[f^{\rm w}_{3,3}](\tau,\xi) d\tau\\
    &
    -
    \frac{i\xi_3}{|\xih|^2 + \xi_3^2}
    \int_0^t 
    e^{-(t - \tau)|\xih|^2}
    \sum_{\ell,m=1}^2
    {\xi_{\ell}\xi_{m}}
    \mathscr{F}_{\mathbb{R}^3}[f^{\rm w}_{\ell,m}](\tau,\xi) d\tau\\
    &  
    -
    \frac{\xi_3^2}{|\xih|^2 + \xi_3^2}
    \int_0^t 
    e^{-(t - \tau)|\xih|^2}
    \sum_{\ell=1}^2
    {i\xi_{\ell}}
    \mathscr{F}_{\mathbb{R}^3}[   f^{\rm w}_{\ell,3}   
     +f^{\rm w}_{3,\ell}](\tau,\xi) d\tau\\
    &
    -
    \frac{i\xi_3^3}{|\xih|^2 + \xi_3^2}
    \int_0^t 
    e^{-(t - \tau)|\xih|^2}
    \mathscr{F}_{\mathbb{R}^3}[f^{\rm w}_{3,3}](\tau,\xi) d\tau.
\end{align}
Here, since there holds
\begin{align}
    &
    \int_0^t 
    e^{-(t - \tau)|\xih|^2}
    \sum_{m=1}^2
    i\xi_m 
    \mathscr{F}_{\mathbb{R}^3}[f^{\rm w}_{3,m}](\tau,\xi) d\tau\\
    &\quad
    -
    \frac{\xi_3^2}{|\xih|^2 + \xi_3^2}
    \int_0^t 
    e^{-(t - \tau)|\xih|^2}
    \sum_{\ell=1}^2
    {i\xi_{\ell}}
    \mathscr{F}_{\mathbb{R}^3}[f^{\rm w}_{3,\ell}](\tau,\xi) d\tau\\
    &\qquad
    =
    \frac{|\xih|^2}{|\xih|^2 + \xi_3^2}
    \int_0^t 
    e^{-(t - \tau)|\xih|^2}
    \sum_{\ell=1}^2
    {i\xi_{\ell}}
    \mathscr{F}_{\mathbb{R}^3}[f^{\rm w}_{3,\ell}](\tau,\xi) d\tau
\end{align}
and 
\begin{align}
    &
    \int_0^t 
    e^{-(t - \tau)|\xih|^2}
    i\xi_3 
    \mathscr{F}_{\mathbb{R}^3}[f^{\rm w}_{3,3}](\tau,\xi) d\tau
    -
    \frac{i\xi_3^3}{|\xih|^2 + \xi_3^2}
    \int_0^t 
    e^{-(t - \tau)|\xih|^2}
    \mathscr{F}_{\mathbb{R}^3}[f^{\rm w}_{3,3}](\tau,\xi) d\tau\\
    &
    \quad 
    =
    \frac{|\xih|^2}{|\xih|^2 + \xi_3^2}
    \int_0^t 
    e^{-(t - \tau)|\xih|^2}
    i\xi_3 
    \mathscr{F}_{\mathbb{R}^3}[f^{\rm w}_{3,3}](\tau,\xi) d\tau,
\end{align}
we have
\begin{align}\label{Fu_3^w}
    \mathscr{F}_{\mathbb{R}^3}[u^{\rm w}_3](t,\xi)
    ={}&
    -
    \frac{i\xi_3}{|\xih|^2 + \xi_3^2}
    \int_0^t 
    e^{-(t - \tau)|\xih|^2}
    \sum_{\ell,m=1}^2
    {\xi_{\ell}\xi_{m}}
    \mathscr{F}_{\mathbb{R}^3}[f^{\rm w}_{\ell,m}](\tau,\xi) d\tau\\
    &
    + 
    \frac{i\xi_3|\xih|^2}{|\xih|^2 + \xi_3^2}
    \int_0^t 
    e^{-(t - \tau)|\xih|^2}
    \mathscr{F}_{\mathbb{R}^3}[{f^{\rm w}_{3,3}}](\tau,\xi) d\tau      \\
    &
    +
    \frac{|\xih|^2}{|\xih|^2 + \xi_3^2}
    \int_0^t 
    e^{-(t - \tau)|\xih|^2}
    \sum_{\ell=1}^{ 2 }   
    {i\xi_{\ell}}
    \mathscr{F}_{\mathbb{R}^3}[{f^{\rm w}_{3,\ell}}](\tau,\xi) d\tau   \\
    &
    -  
    \int_0^t 
    e^{-(t - \tau)|\xih|^2}
    \sum_{\ell=1}^{ 2 }   
    {i\xi_{\ell}}
    \mathscr{F}_{\mathbb{R}^3}[ {f^{\rm w}_{\ell,3} }](\tau,\xi) d\tau   \\
    &
    + 
     \frac{|\xih|^2}{|\xih|^2 + \xi_3^2}
    \int_0^t 
    e^{-(t - \tau)|\xih|^2}
    \sum_{\ell=1}^{ 2 }   
    {i\xi_{\ell}}
    \mathscr{F}_{\mathbb{R}^3}[{f^{\rm w}_{\ell,3}}](\tau,\xi) d\tau .  
\end{align}
For $k=1,2$, we see that 
\begin{align}
    \mathscr{F}_{\mathbb{R}^3}[u^{\rm w}_k](t,\xi)
    ={}&
    \int_0^t 
    e^{-(t - \tau)|\xih|^2}
    \sum_{\ell,m=1}^3
    \left( \delta_{k,\ell} - \frac{\xi_{k}\xi_{\ell}}{|\xi|^2} \right)
    i\xi_m 
    \mathscr{F}_{\mathbb{R}^3}[f^{\rm w}_{\ell,m}](\tau,\xi) d\tau\\
    ={}&
    \int_0^t 
    e^{-(t - \tau)|\xih|^2}
    \sum_{m=1}^3
    i\xi_m 
    \mathscr{F}_{\mathbb{R}^3}[f^{\rm w}_{k,m}](\tau,\xi) d\tau\\
    &
    -
    \int_0^t 
    e^{-(t - \tau)|\xih|^2}
    \sum_{\ell,m=1}^3
    \frac{i\xi_{k}\xi_{\ell}\xi_m}{|\xi|^2}   
    \mathscr{F}_{\mathbb{R}^3}[f^{\rm w}_{\ell,m}](\tau,\xi) d\tau\\
    ={}&
    \int_0^t 
    e^{-(t - \tau)|\xih|^2}
    \sum_{m=1}^3
    i\xi_m 
    \mathscr{F}_{\mathbb{R}^3}[f^{\rm w}_{k,m}](\tau,\xi) d\tau\\
    &
    -
    \frac{1}{|\xih|^2+\xi_3^2}
    \int_0^t 
    e^{-(t - \tau)|\xih|^2}
    \sum_{\ell,m=1}^2
    {i\xi_{k}\xi_{\ell}\xi_m}   
    \mathscr{F}_{\mathbb{R}^3}[f^{\rm w}_{\ell,m}](\tau,\xi) d\tau\\
    &
    -
    \frac{i\xi_3}{|\xih|^2+\xi_3^2}
    \int_0^t 
    e^{-(t - \tau)|\xih|^2}
    \sum_{\ell=1}^2
    {\xi_{k}\xi_{\ell}}   
    \mathscr{F}_{\mathbb{R}^3}[f^{\rm w}_{3,\ell}+f^{\rm w}_{\ell,3}](\tau,\xi) d\tau\\
    &
    -
    \int_0^t 
    e^{-(t - \tau)|\xih|^2}
    {i\xi_{k}}   
    \mathscr{F}_{\mathbb{R}^3}[f^{\rm w}_{3,3}](\tau,\xi) d\tau\\
    & 
    +
    \frac{|\xih|^2}{|\xih|^2+\xi_3^2}
    \int_0^t 
    e^{-(t - \tau)|\xih|^2}
    {i\xi_{k}}   
    \mathscr{F}_{\mathbb{R}^3}[f^{\rm w}_{3,3}](\tau,\xi) d\tau\\
\end{align}
Using the basic facts that\footnote{Our definition of Fourier transform and inverse Fourier transform is given by
\begin{align}
    \mathscr{F}_{\mathbb{R}^d} [ f ](\xi) = \int_{\mathbb{R}^d} e^{-ix\cdot \xi}f(x) dx,\qquad
    \mathscr{F}_{\mathbb{R}^d}^{-1} [ f ](x) = \frac{1}{(2\pi)^d} \int_{\mathbb{R}^d} e^{ix\cdot \xi}f(\xi) d\xi.
\end{align}
}
\begin{align}
    \mathscr{F}_{\mathbb{R}}^{-1}
    \left[ \frac{1}{|\xih|^2 + \xi_3^2} \right](x_3) 
    &= 
    \frac{1}{2|\xih|}e^{-|\xih||x_3|},\\  
    \mathscr{F}_{\mathbb{R}}^{-1}
    \left[ \frac{i\xi_3}{|\xih|^2 + \xi_3^2} \right](x_3) 
    &= 
    -
    \frac{1}{2}
    \operatorname{sgn}(x_3)
    e^{-|\xih||x_3|},
\end{align}
we obtain 
\begin{align}
    \mathscr{F}_{\mathbb{R}^2}[u^{\rm w}_3](t,  \xih,x_3    )
    ={}&
    { -}
    \int_{\mathbb{R}}
    \frac{ |\xih|}{2}\operatorname{sgn}(x_3-y_3)
    e^{-|\xih||x_3-y_3|}\\
    &\qquad 
    \times
    \int_0^t 
    e^{-(t - \tau)|\xih|^2}
    \sum_{\ell,m=1}^2
    {i}\xi_{\ell}
    { \frac{i\xi_m}{|\xih|} }
    \mathscr{F}_{\mathbb{R}^2}[f^{\rm w}_{\ell,m}](\tau,\xih,y_3) d\tau
    dy_3\\
    & 
    -
    \int_{\mathbb{R}}
    \frac{|\xih|^{ 2   }  }{2}
    e^{-|\xih||x_3-y_3|}\operatorname{sgn}(x_3-y_3)\\
    &\qquad
    \times
    \int_0^t 
    e^{-(t - \tau)|\xih|^2}
    \mathscr{F}_{\mathbb{R}^2}[{f^{\rm w}_{3,3}}](\tau,\xih,y_3) d\tau
    dy_3\\
    &
    +
    \int_{\mathbb{R}}
    \frac{|\xih|}{2}
    e^{-|\xih||x_3-y_3|}
    \int_0^t 
    e^{-(t - \tau)|\xih|^2}
    \sum_{\ell=1}^{ 2 }      
    {i\xi_{\ell}}
    \mathscr{F}_{\mathbb{R}^2}[{f^{\rm w}_{3,\ell}}](\tau,\xih,y_3) d\tau
    dy_3 \\
    & 
    - 
    \int_0^t 
    e^{-(t - \tau)|\xih|^2}
    \sum_{\ell=1}^{ 2}      
    {i\xi_{\ell}}
    \mathscr{F}_{\mathbb{R}^2}[{f^{\rm w}_{\ell,3}}](\tau,\xih,x_3) d\tau  \\
    &  
      + 
     \int_{\mathbb{R}}
    \frac{|\xih|}{2}
    e^{-|\xih||x_3-y_3|}
    \int_0^t 
    e^{-(t - \tau)|\xih|^2}
    \sum_{\ell=1}^{ 2}      
    {i\xi_{\ell}}
    \mathscr{F}_{\mathbb{R}^2}[{f^{\rm w}_{\ell,3}}](\tau,\xih,y_3) d\tau
    d y_3.    
\end{align}
By the change of the variables and the definition of the extension of the external force, we have the following representation on $\R^3_+$
\begin{align}
    u^{\rm w}_3(t)
    ={}&
    -
    V^{(+)}
    \int_0^t 
    e^{(t - \tau)\Deltah}
    \sum_{\ell,m=1}^2
    {\partial_{x_\ell}{S_{x_m}}}
    f_{\ell,m}(\tau) d\tau\\
    &
    -
    V^{(-)}
    \int_0^t 
    e^{(t - \tau)\Deltah}
    |\nablah|
    {f_{3,3}}(\tau) d\tau   \label{uw3} \\
    & 
    +
      (W-1)    
    \int_0^t 
    e^{(t - \tau)\Deltah}   
   \sum_{\ell=1}^2  
    {\partial_{x_\ell}}
    {f_{3,\ell}}(\tau) d\tau,
\end{align}
where and in what follows we set 
\begin{align}
    V^{(\pm)}f(x)
    :={}&
    \frac{1}{2}
    \int_0^{\infty}
    {|\nablah|}
    \left\{
    \operatorname{sgn}(x_3-y_3)
    e^{-|\nablah||x_3-y_3|}
    \pm
    e^{-|\nablah|(x_3+y_3)}
    \right\}
    f(\xh,y_3)dy_3,\\
    W^{(\pm)}f(x)
    :={}&
    \frac{1}{2}
    \int_0^{\infty}
    { |\nablah|}
    \left(
    e^{-|\nablah||x_3-y_3|}
    \pm
    e^{-|\nablah|(x_3+y_3)}
    \right)
    f(\xh,y_3)dy_3, \\
       Vf(x)
    :={}&
    \left(V^{(+)}+V^{(-)}\right)f(x)
    =
    \int_0^{\infty}
    {|\nablah|}
    \operatorname{sgn}(x_3-y_3)
    e^{-|\nablah||x_3-y_3|}
    f(\xh,y_3)dy_3, \\
       Wf(x)
    :={}&
    \left(W^{(+)}+W^{(-)}\right)f(x)
    = \int_0^{\infty}
    { |\nablah|}
    e^{-|\nablah||x_3-y_3|}
    f(\xh,y_3)  d y_3 .
\end{align}  
For $k=1,2$, we see the following representation on $\R^3_+$ that 
\begin{align*}
    u^{\rm w}_k(t)
    ={}&
    \int_0^t 
    e^{(t - \tau)\Deltah}
    \sum_{m=1}^3
    \partial_{x_m} 
    f_{k,m}(\tau) d\tau\\
    &
    +
    W^{(+)}
    \int_0^t 
    e^{(t - \tau)\Deltah}
    \sum_{\ell,m=1}^2
    {\partial_{x_k}{S_{x_\ell}}{S_{x_m}}}   
    f_{\ell,m}(\tau) d\tau\\
    &
    -
    V
    \int_0^t 
    e^{(t - \tau)\Deltah}
    \sum_{\ell=1}^2
    {\partial_{x_k}{S_{x_\ell}}}
    f_{3,\ell}(\tau) d\tau\\
    &
    -
    \left(1-W^{(-)}\right)
    \int_0^t 
    e^{(t - \tau)\Deltah}
    {\partial_{x_k}}   
    f_{3,3}(\tau) d\tau.
\end{align*}

Next we come to the boundary data appearing in \eqref{u3-1}-\eqref{uh-1}. It follows from \eqref{uw3} that 
\begin{align}\label{bd}  
    e^{-x_3|\nablah|}u^{\rm w}_3(t,\xh,0)
    ={}&
    T
    \int_0^t
    e^{(t-\tau)\Deltah}\left(
    \sum_{\ell=1}^2  
    \partial_{x_{\ell}}f_{3,\ell}
    +
    |\nablah|
    f_{3,3}
    \right)(\tau)d\tau,\\
    &  \quad 
   -   
    e^{  -x_3 |\nablah|}       
    \int_0^t
    e^{(t-\tau)\Deltah}
    \sum_{\ell=1}^2  
    \partial_{x_{\ell}}f_{3,\ell}
    (\tau,\xh,0) d\tau,          
 \end{align} 
with 
 \begin{align}
    T f(x)
    :={}&
    \int_0^{\infty}
    |\nablah|
    e^{-(x_3+y_3)|\nablah|}
    f(\xh,y_3)  dy_3.
\end{align}
Collecting all the material obtained so far, we arrive at
\begin{lemm}\label{lemm:sol-form}
The solution $u$ of \eqref{eq:lin-1} has the following formula:
\begin{align}             
    u_3(t,x) ={}&  
    U e ^{t \Delta_{\rm h}} (u_{0,3}- S \cdot u_{0,\rm h} )(x)\\
    &
    -
    V^{(+)}
    \int_0^t 
    e^{(t - \tau)\Deltah}
    \sum_{\ell,m=1}^2
    {\partial_{x_\ell}{S_{x_m}}}
    f_{\ell,m}(\tau) d\tau\\
    &
    -
     \left(V^{(-)}   +   T   \right)
    \int_0^t 
    e^{(t - \tau)\Deltah}
    |\nablah|
    f_{3,3}(\tau) d\tau\\
    & 
    +
     \left(  W -1-T  \right)   
    \int_0^t 
    e^{(t - \tau)\Deltah}
   \sum_{\ell=1}^{ 2 }  
    {\partial_{x_\ell}}
     f_{3,\ell}      (\tau)   d\tau, \\
      & 
   +      
    e^{  -x_3 |\nablah|}       
    \int_0^t
    e^{(t-\tau)\Deltah}
    \sum_{\ell=1}^2  
    \partial_{x_{\ell}}f_{3,\ell}
    (\tau,\xh,0) d\tau,      
\end{align}
and for $k=1,2$,
\begin{align}
    u_{k}(t,x)={}&
    e ^{ t \Delta_{\rm h} }  (u_{0,k}  + S_{x_k} u_{0,3}  ) (x)-
    S_{x_k}  U e ^{t \Delta_{\rm h}} (u_{0,3}- S \cdot u_{0,\rm h} )(x)\\
    &+
    \int_0^t 
    e^{(t - \tau)\Deltah}
    \sum_{m=1}^3
    \partial_{x_m} 
    f_{k,m}(\tau) d\tau\\
    &
    +
    W^{(+)}
    \int_0^t 
    e^{(t - \tau)\Deltah}
    \sum_{\ell,m=1}^2
    {\partial_{x_k}{S_{x_\ell}S_{x_m}}}   
    f_{\ell,m}(\tau) d\tau\\
    &
    -
    S_{x_k}(  V  -  T     )   
    \int_0^t 
    e^{(t - \tau)\Deltah}
    \sum_{\ell=1}^2
    {{\partial_{x_\ell}}}
    f_{3,\ell}(\tau) d\tau\\
    &
    -  
  \left(1    -  T   -W^{(-)}\right)
    \int_0^t 
    e^{(t - \tau)\Deltah}
    {\partial_{x_k}}   
    f_{3,3}(\tau) d\tau \\
     &  
   - S_{x_k}  
    e^{  -x_3 |\nablah|}       
    \int_0^t
    e^{(t-\tau)\Deltah}
    \sum_{\ell=1}^2  
    \partial_{x_{\ell}}f_{3,\ell}
    (\tau,\xh,0) d\tau.        
\end{align}
\end{lemm}

Substituting $-u \otimes u$ into $f$ and invoking the boundary condition of velocity, we obtain an explicit solution formula to the nonlinear equations. 
\begin{cor}\label{lemm:sol-form-u}
The solution $u$ of \eqref{eq1} has the following formula:
\begin{align}             
    u_3(t,x) ={}&  
    U e ^{t \Delta_{\rm h}} (u_{0,3}- S \cdot u_{0,\rm h} )(x)\\
    &
    +
    V^{(+)}
    \int_0^t 
    e^{(t - \tau)\Deltah}
    \sum_{\ell,m=1}^2
    {\partial_{x_\ell}{S_{x_m}}}
    (u_{\ell}u_m)(\tau) d\tau\\
    &{}
    + 
   \left(V^{(-)} +  T \right)
    \int_0^t 
    e^{(t - \tau)\Deltah}
    |\nablah|
    (u_3(\tau)^2) d\tau\\
    &
    -
    \left( 
    W -1-T  \right)     
    \int_0^t 
    e^{(t - \tau)\Deltah}
   \div_{\rm h}
    (\uh u_3)(\tau)     d\tau    \\
    =: {}&  
    U e ^{t \Delta_{\rm h}} (u_{0,3}- S \cdot u_{0,\rm h} )(x) + \sum_{m=1}^{3} \mathcal{D}_m^{\rm v}[u](t) 
\end{align}
and
\begin{align}
    u_{\rm h}(t,x)={}&
    e ^{ t \Delta_{\rm h} }  (u_{0,\rm h}  + S u_{0,3}  ) (x)-
    SU e ^{t \Delta_{\rm h}} (u_{0,3}- S \cdot u_{0,\rm h} )(x)\\
    &-
    \int_0^t 
    e^{(t - \tau)\Deltah}
    \partial_{x_3} 
    (\uh u_{3})(\tau) d\tau\\
    &-
    \int_0^t 
    e^{(t - \tau)\Deltah}
    \div_{\rm h}
    (\uh \otimes \uh)(\tau) d\tau\\
    &
    -
    W^{(+)}
    \int_0^t 
    e^{(t - \tau)\Deltah}
    \nablah
    \sum_{\ell,m=1}^2
    {S_{x_\ell}S_{x_m}}   
    (u_{\ell}u_{m})(\tau) d\tau\\
    &
    +
     S(V - T)
    \int_0^t 
    e^{(t - \tau)\Deltah}
    \div_{\rm h}
    (\uh u_3)(\tau) d\tau\\
    &     
    +\left(1  - T  -W^{(-)}\right)
    \int_0^t 
    e^{(t - \tau)\Deltah}
    \nablah   
    (u_3(\tau)^2) d\tau \\
  =: {}&
    e ^{ t \Delta_{\rm h} }  (u_{0,\rm h}  + S u_{0,3}  ) (x)-
    SU e ^{t \Delta_{\rm h}} (u_{0,3}- S \cdot u_{0,\rm h} )(x)   +\sum_{m=1}^{5} \mathcal{D}_m^{\rm h}[u](t) .  
\end{align}
\end{cor}

\subsection{Decay estimates for the linearized equations} 
In this subsection, we focus on the anisotropic Stokes equations by dropping the convective term:   
\begin{align}\label{st1}
\begin{cases}
    \partial_t u -  \Delta_{\rm h} u + \nabla p = 0, & \qquad t>0,x\in \mathbb{R}^3_{+},\\
    \div  u = 0, & \qquad t\geq0, x \in\mathbb{R}^3_{+},\\
    u=0, & \qquad t \geq 0, x \in  \p \mathbb{R}^3_{+} ,\\
    u(0,x) = u_0(x), & \qquad x \in \mathbb{R}^3_{+}.
\end{cases} 
\end{align}
By Lemma \ref{lemm:sol-form}, we see that
\begin{align}             
    u_3(t) ={} &   U e ^{t \Delta_{\rm h}} (u_{0,3}- S \cdot u_{0,\rm h} ) ,   \label{sol-form-3} \\
    u_{\rm h}(t)={} &    e ^{ t \Delta_{\rm h} }  (u_{0,\rm h}  + S u_{0,3}  )  -S U e ^{t \Delta_{\rm h}} (u_{0,3}- S \cdot u_{0,\rm h} ).     \label{sol-form-h}
\end{align}

With the solution formula at hand, we begin with the boundedness of operator $U$ and the Riesz operator $S$ in the Chemin--Lerner type space. 
\begin{lemm}\label{lemm-U}
Let $1 \leq p,q,\sigma \leq \infty$. Then, there exists a positive constant $C=C(p,q)$ such that 
\[
\| U f \|_{\mathcal{L}_{\sigma}^{p,q} (\mathbb{R}^3_+) } \leq C \| f \|_{\mathcal{L}_{\sigma}^{p,q} (\mathbb{R}^3_+) } \,\,\,\,\text{  for all   }\,\,  f \in \mathcal{L}_{\sigma}^{p,q}(\mathbb{R}^3_+).
\]
{
Besides, there exists a positive constant $C=C(p)$ such that
\[
\|  S f \|_{\mathcal{L}_{\sigma}^{p,q} (\mathbb{R}^3_+) }   \leq C\| f \|_{\mathcal{L}_{\sigma}^{p,q} (\mathbb{R}^3_+) }   \,\,\,\,\text{  for all   }\,\,  f \in \mathcal{L}_{\sigma}^{p,q}(\mathbb{R}^3_+).
\]
}
\end{lemm}
\begin{rem}
It is known that $U$ is bounded on $L^r(\mathbb{R}^3_+)$ for all $1< r < \infty$ by the standard theory of singular integral operators.
In contrast, we reveal that $U$ is bounded in $\mathcal{L}^{p,q}_{\sigma}(\mathbb{R}^3_+)$ in the end point cases $p=q=1$ or $p=q=\infty$.
\end{rem}
\begin{proof}[Proof of Lemma \ref{lemm-U}]
Since it holds 
\begin{align*}
    Uf(x)
    =
    |\nablah|\int_0^{x_3} e^{-|\nablah|(x_3-y_3)}f(\xh,y_3)dy_3,
\end{align*}
we infer that
\begin{align*}
    \| \Deltahh_j Uf \|_{L^p(\mathbb{R}^2_{\xh})}
    \leq
    C2^j
    \int_0^{x_3}
    e^{-c2^j(x_3-y_3)}
    \| \Deltahh_j f(\cdot,y_3)\|_{L^p(\mathbb{R}^2_{\xh})}
    dy_3.
\end{align*}
Taking $L^q(\mathbb{R}_{x_3>0})$ norm of the above relation and using the Hausdorff--Young inequality yields that 
\begin{align*}
    \| \Deltahh_j Uf \|_{\Lv^q\Lh^p(\mathbb{R}^3_+)}
    \leq
    C2^j \int_0^{\infty} e^{-c2^jy_3}dy_3
    \| \Deltahh_j f\|_{\Lv^q\Lh^p(\mathbb{R}^3_+)}
    =
    C
    \| \Deltahh_j f\|_{\Lv^q\Lh^p(\mathbb{R}^3_+)},
\end{align*}
which provides the first estimate upon taking the $\ell^{\sigma}$-norm over the index $j \in \mathbb{Z}$.

The proof of the boundedness for the Riesz operator $S$ is similar as soon as one uses Lemma \ref{app-lemm-2} to see
\[
\|   \Deltahh_j Sf  \|_{L^p(\mathbb{R}^2_{\xh})} \leq C
\|   \Deltahh_j f  \|_{L^p(\mathbb{R}^2_{\xh})} . 
\]
This completes the proof.
\end{proof}

Next, we establish the decay estimates for the anisotropic heat semigroup $\{ e^{t\Deltah}\}_{t>0}$ in the Chemin--Lerner type space.

\begin{lemm}\label{lemm-heat}
Let $1 \leq p\leq q \leq \infty$, $1 \leq r, \sigma \leq \infty$, and $\alpha = (\alphah, \alpha_3) \in (\mathbb{N} \cup \{0 \})^2 \times (\mathbb{N} \cup \{0 \})$. 
Then, there exists a positive constant $C=C(p,q,\alpha)$ such that 
\begin{align*}
    \| \nabla^{\alpha} e^{t\Deltah} f \|_{\mathcal{L}_{\sigma}^{q,r} (\mathbb{R}^3_+) } 
    \leq 
    C t^{-(\frac{1}{p}-\frac{1}{q}) - \frac{|\alphah|}{2}}
    \| \p_3^{\alpha_3} f \|_{\mathcal{L}_{\sigma}^{p,r} (\mathbb{R}^3_+)  }
\end{align*}
for all $t>0$ and $f$ such that $\p_3^{\alpha_3} f \in \mathcal{L}_{\sigma}^{p,r}(\mathbb{R}^3_+)$.

In particular, if $(1/p-1/q)+|\alphah|/2 > 0$, then there exists a positive constant $C=C(p,q,\sigma,\alpha)$ such that 
\begin{align*}
    \| \nabla^{\alpha} e^{t\Deltah} f \|_{\mathcal{L}_1^{q,r} (\mathbb{R}^3_+) } 
    \leq 
    C t^{-(\frac{1}{p}-\frac{1}{q})-\frac{|\alphah|}{2}}
    \| \p_3^{\alpha_3} f \|_{\mathcal{L}_{\infty}^{p,r} (\mathbb{R}^3_+) }
\end{align*}
for all $t>0$ and $f$ such that $\p_3^{\alpha_3} f \in \mathcal{L}_{\infty}^{p,r}(\mathbb{R}^3_+)$. 
\end{lemm}
\begin{proof}
It follows from the standard $L^p-L^q$ estimates for the heat semigroup $\{ e^{t\Deltah}\}_{t>0}$ that 
\begin{align*}
    \| \Deltahh_j \nabla^{\alpha} e^{t\Deltah}f \|_{\Lv^r\Lh^q (\mathbb{R}^3_+) } \leq C t^{-(\frac{1}{p} - \frac{1}{q})-\frac{|\alphah|}{2}} \| \Deltahh_j \partial_3^{\alpha_3} f \|_{\Lv^r \Lh^p (\mathbb{R}^3_+) }.
\end{align*}
Taking $\ell^{\sigma}(\mathbb{Z})$ with respect to $j$, we obtain the first estimate.

Next, we see that
\begin{align*}
    \| \nabla^{\alpha} e^{t\Deltah} f \|_{\mathcal{L}_1^{q,r} (\mathbb{R}^3_+) } 
    ={}&
    \sum_{j \in \mathbb{Z}}
    \| \Deltahh_j \nabla^{\alpha} e^{t\Deltah} f \|_{\Lv^r \Lh^q (\mathbb{R}^3_+) }\\
    \leq{}& 
    C
    \sum_{j \in \mathbb{Z}}
    2^{|\alphah|j}
    e^{-c2^{2j}t}
    \| \Deltahh_j \p_3^{\alpha_3} f \|_{\Lv^r \Lh^q (\mathbb{R}^3_+) }\\
    \leq{}&
    C
    \sum_{j \in \mathbb{Z}}
    2^{2\{(\frac{1}{p}-\frac{1}{q})+\frac{|\alphah|}{2}\}j}
    e^{-c2^{2j}t}
    \| \Deltahh_j \p_3^{\alpha_3} f \|_{\Lv^r \Lh^p (\mathbb{R}^3_+) }\\
    \leq{}&
    C
    \left(
    \sum_{j \in \mathbb{Z}}
    2^{2\{(\frac{1}{p}-\frac{1}{q})+\frac{|\alphah|}{2}\}j}
    e^{-c2^{2j}t}
    \right)
    \sup_{j\in \mathbb{Z}}\| \Deltahh_j \p_3^{\alpha_3} f \|_{\Lv^r \Lh^p (\mathbb{R}^3_+) }\\
    \leq{}&
    C t^{-(\frac{1}{p}-\frac{1}{q})-\frac{|\alphah|}{2}}
    \| \p_3^{\alpha_3} f \|_{\mathcal{L}_{\infty}^{p,r} (\mathbb{R}^3_+) },
\end{align*}
which completes the proof. Here, we used in the last step the simple fact that for any positive $s$
\begin{align*}
\sup_{t>0} \sum_{j \in \mathbb{Z}} t^s 2^{2js} e ^{  -ct 2^{2j} } <\infty .
\end{align*}
\end{proof}

As a corollary of Lemmas \ref{lemm-U} and \ref{lemm-heat}, we obtain the $\mathcal{L}^{p,r}_{\sigma}-\mathcal{L}^{q,r}_{\sigma}$ type estimates for the semigroup $\{ e^{-tA} \}_{t>0} $.
\begin{cor}\label{cor:semigr}
Let $1 \leq p \leq q \leq \infty$, $1 \leq r,\sigma \leq \infty$, and $\alpha = (\alphah, \alpha_3) \in (\mathbb{N} \cup \{0 \})^2 \times (\mathbb{N} \cup \{0 \})$. 
Then, there exists a positive constant $C=C(p,q,\alpha)$ such that 
\begin{align*}
    \| \nabla^{\alpha} e^{-tA}u_0 \|_{\mathcal{L}_{\sigma}^{q,r}(\mathbb{R}^3_+)}
    \leq
    C
    t^{-(\frac{1}{p}-\frac{1}{q})-\frac{|\alphah|}{2}}
    \| \p_3^{\alpha_3} u_0 \|_{\mathcal{L}_{\sigma}^{p,r}(\mathbb{R}^3_+) }
\end{align*}
for all $u_0$ with $\p_3^{\alpha_3} u_0 \in\mathcal{L}_{\sigma}^{p,r}(\mathbb{R}^3_+),\div u_0 = 0$.
\end{cor}

The following lemma is the main achievement of this section, revealing the enhanced dissipation mechanism for the third component of velocity field. 
\begin{lemm}\label{lemm-enh-decay}
    Let $1 \leq p,q \leq \infty$  with $(p,q)\neq (1,1)$ and $1 < r \leq \infty$.
    Then, there exists a positive constant $C=C(p,q,r)$ such that 
    \begin{align}\label{enh-decay-p}
        \begin{split}
        \| (e^{-tA}u_{0})_{\rm h} \|_{\mathcal{L}^{r,q}_1(\mathbb{R}^3_+)}
        &
        \leq
        Ct^{-(1-\frac{1}{r})}\| u_0 \|_{\mathcal{L}^{1,q}_{\infty}(\mathbb{R}^3_+)}
        +
        Ct^{-(1-\frac{1}{r})-\frac{1}{2}(1-\frac{1}{q})}\| u_0 \|_{{\mathcal{L}^{1}_{\infty}(\mathbb{R}^3_+)}}\\
        \| (e^{-tA}u_{0})_3 \|_{\mathcal{L}^{p,q}_1(\mathbb{R}^3_+)}
        &
        \leq
        Ct^{-(1-\frac{1}{p})-\frac{1}{2}(1-\frac{1}{q})}\| u_0 \|_{{\mathcal{L}^{1}_{\infty}(\mathbb{R}^3_+)}}
        \end{split}
    \end{align}
    for all {$u_0 $} with $\div u_0=0$.

    Moreover, it holds for $1 < q \leq \infty$
    \begin{align}\label{enh-decay-1}
        \begin{split}
        \| e^{-tA}u_{0} \|_{\mathcal{L}^{1,q}_{\sigma}(\mathbb{R}^3_+)}
        &
        \leq
        C\| u_0 \|_{\mathcal{L}^{1,q}_{\sigma}(\mathbb{R}^3_+)}
        +
        Ct^{-\frac{1}{2}(1-\frac{1}{q})}
        \| u_0 \|_{{\mathcal{L}^{1}_{\infty}(\mathbb{R}^3_+)}}\\
        \| e^{-tA}u_{0} \|_{\mathcal{L}^1_{\sigma}(\mathbb{R}^3_+)}
        &
        \leq
        C
        \| u_0 \|_{{\mathcal{L}^1_{\sigma}(\mathbb{R}^3_+)}}
        \end{split}
    \end{align}
    for all $1\leq \sigma \leq \infty$ and {$u_0$} with $\div u_0=0$.
\end{lemm}

\begin{proof}
Observe that the estimates for $(e^{-tA}u_{0})_{\rm h}$ immediately follow from estimates on $(e^{-tA}u_{0})_{3}$, Lemma \ref{lemm-heat} {and the second estimate of Lemma \ref{lemm-U}}, we may only focus on $(e^{-tA}u_{0})_3$.  
By the solution formula \eqref{sol-form-3}, we see that
\begin{align*}
    (e^{-tA}u_{0})_3(x)
    ={}&
    \int_0^{x_3} |\nablah|^{\frac{1}{q}}e^{-(x_3 - y_3)|\nablah|}\left[|\nablah|^{1-\frac{1}{q}}e ^{t \Delta_{\rm h}} (u_{0,3}- S \cdot u_{0,\rm h} ) \right](x_{\rm h},y_3) dy_3.
\end{align*}
Hence, we have 
\begin{align*}
   &
   \| \Delta_j^{\rm h} (e^{-tA}u_{0})_3(\cdot,x_3) \|_{L^p(\mathbb{R}^2)}\\
    &\quad \leq{}
    C 
    \int_0^{x_3}
    2^{\frac{1}{q}j}e^{-c2^j(x_3-y_3)}
    \left\| \left[\Delta_j^{\rm h}|\nablah|^{1-\frac{1}{q}}e ^{t \Delta_{\rm h}} (u_{0,3}- S \cdot u_{0,\rm h} ) \right](\cdot,y_3)  \right\|_{L^p(\mathbb{R}^2)}dy_3.
\end{align*}
Thus, by the Hausdorff--Young inequality, it holds
\begin{align*}
    &\| (e^{-tA}u_{0})_3(t) \|_{\mathcal{L}^{p,q}_1(\mathbb{R}^3_+)}\\
    &\quad \leq{}
    C
    \sum_{j \in \mathbb{Z}}
    2^{\frac{1}{q}j}\|e^{-c2^jz_3}\|_{L^q_{z_3}(0,\infty)}
   \left \| \Delta_j^{\rm h}|\nablah|^{1-\frac{1}{q}}e ^{t \Delta_{\rm h}} (u_{0,3}- S \cdot u_{0,\rm h} ) \right\|_{L^1_{z_3}(0,\infty;L^p(\mathbb{R}^2))}\\
    &\quad \leq{}
    C
    \sum_{j \in \mathbb{Z}}
    \left\| \Delta_j^{\rm h}|\nablah|^{1-\frac{1}{q}}e ^{t \Delta_{\rm h}} (u_{0,3}- S \cdot u_{0,\rm h} ) \right\|_{L^1_{x_3}(0,\infty;L^p(\mathbb{R}^2_{x_{\rm h}}))}\\
    &\quad\leq{}
    C 
    \| e ^{t \Delta_{\rm h}} (u_{0,3}- S \cdot u_{0,\rm h} )\|_{{L^1_{x_3}}(0,\infty;\dot{B}_{p,1}^{1-\frac{1}{q}}(\mathbb{R}^2_{x_{\rm h}}))}.
\end{align*}
For the case of $(p,q)\neq (1,1)$, 
we see that
\begin{align*}
    &\| e ^{t \Delta_{\rm h}} (u_{0,3}- S \cdot u_{0,\rm h} )\|_{L^1_{x_3}(0,\infty;\dot{B}_{p,1}^{1-\frac{1}{q}}(\mathbb{R}^2_{x_{\rm h}}))} \\
    &\quad=
    \sum_{j \in \mathbb{Z}}
    2^{(1-\frac{1}{q})j}
    \| 
    \Deltahh_j  e ^{t \Delta_{\rm h}} (u_{0,3}- S \cdot u_{0,\rm h} )
    \|_{L^1_{x_3}(0,\infty;L^p(\mathbb{R}^2_{x_{\rm h}}))}\\ 
    &\quad\leq
    C
    \sum_{j \in \mathbb{Z}}
    2^{(1-\frac{1}{q})j}e^{-c2^{2j}t}
    \| 
    \Deltahh_j (u_{0,3}- S \cdot u_{0,\rm h} )
    \|_{L^1_{x_3}(0,\infty;L^p(\mathbb{R}^2_{x_{\rm h}}))}\\
    &\quad\leq
    C
    \sum_{j \in \mathbb{Z}}
    2^{\{2(1-\frac{1}{p})+(1-\frac{1}{q})\}j}e^{-c2^{2j} t}
    \| 
    \Deltahh_j (u_{0,3}- S \cdot u_{0,\rm h} )
    \|_{L^1(\mathbb{R}^3_+)}\\
    &\quad\leq
    C
    \left( 
    \sum_{j \in \mathbb{Z}}
    2^{\{2(1-\frac{1}{p})+(1-\frac{1}{q})\}j}e^{-c2^{2j}t}
    \right) 
    \sup_{j \in \mathbb{Z}}
    \| 
    \Deltahh_j (u_{0,3}- S \cdot u_{0,\rm h} )
    \|_{L^1(\mathbb{R}^3_+)}\\
    &\quad \leq{}
    C
    t^{-(1-\frac{1}{p})-\frac{1}{2}(1-\frac{1}{q})}
    \sup_{j \in \mathbb{Z}}
    \| 
    \Deltahh_j u_0
    \|_{L^1(\mathbb{R}^3_+)}     \\
    &\quad ={}
    C
    t^{-(1-\frac{1}{p})-\frac{1}{2}(1-\frac{1}{q})}
    \| u_0 \|_{\mathcal{L}^1_{\infty}(\mathbb{R}^3_+)},
\end{align*}
which provides the second estimate of \eqref{enh-decay-p}. Here, we again invoked the simple fact that for any positive $s$
\begin{align*}
\sup_{t>0} \sum_{j \in \mathbb{Z}} t^s 2^{2js} e ^{  -ct 2^{2j} } <\infty .
\end{align*}

Next, we focus on the case $p=q=1$.
From the similar calculations as above, we deduce from Lemma \ref{lemm-U} and Lemma \ref{lemm-heat} that
\begin{align*}
    &
    \| (e^{-tA}u_{0})_3 \|_{\mathcal{L}^1_{\sigma}(\mathbb{R}^3_+)}\\
    &\quad \leq{}
    C 
    \left\|
    \left\{
    \left\|
    \int_0^{x_3}
    2^{j}e^{-c2^j(x_3-y_3)}
    \left\| [\Delta_j^{\rm h}e ^{t \Delta_{\rm h}} (u_{0,3}- S \cdot u_{0,\rm h} )](\cdot,y_3)\right\|_{L^1(\mathbb{R}^2)}dy_3
    \right\|_{L^1(\mathbb{R}_{x_3>0})}
    \right\}_{j \in \mathbb{Z}}
    \right\|_{\ell^{\sigma}}\\
    &\quad \leq 
    C
    \| e ^{t \Delta_{\rm h}} (u_{0,3}- S \cdot u_{0,\rm h} )\|_{\widetilde{L^1_{x_3}}(0,\infty;\dot{B}_{1,\sigma}^0(\mathbb{R}^2_{x_{\rm h}}))}\\
    &\quad \leq{}
    C
    \| u_{0,3}- S \cdot u_{0,\rm h} \|_{\widetilde{L^1_{x_3}}(0,\infty;\dot{B}_{1,\sigma}^0(\mathbb{R}^2_{x_{\rm h}}))}\\
    &\quad \leq{}
    C
    \| u_0 \|_{\widetilde{L^1_{x_3}}(0,\infty;\dot{B}_{1,\sigma}^0(\mathbb{R}^2_{x_{\rm h}}))}\\
    &\quad ={}
    C
    \| u_0 \|_{\mathcal{L}^1_{\sigma}(\mathbb{R}^3_+)}.
\end{align*}
This implies \eqref{enh-decay-1}
and completes the proof.
\end{proof}

From the continuous embedding $\mathcal{L}_{1}^{p,q}(\mathbb{R}^3_+) \hookrightarrow \Lv^q\Lh^p(\mathbb{R}^3_+)$, we obtain the following corollary:
\begin{cor} \label{cor2.6}
    Let $1 \leq p,q \leq \infty$  with $(p,q)\neq (1,1)$ and $1 < r \leq \infty$.
    Then, there exists a positive constant $C=C(p,q,r)$ such that 
    \begin{align*}
        \begin{split}
        \| (e^{-tA}u_{0})_{\rm h}
        \|_{L^q_{x_3}L^r_{\xh}(\mathbb{R}^3_+)}
        &
        \leq
        Ct^{-(1-\frac{1}{r})}\| u_0 \|_{\mathcal{L}^{1,q}_{\infty}}
        +
        Ct^{-(1-\frac{1}{r})-\frac{1}{2}(1-\frac{1}{q})}\| u_0 \|_{{\mathcal{L}^{1}_{\infty}(\mathbb{R}^3_+)}}\\
        \| (e^{-tA}u_{0})_3 \|_{L^q_{x_3}L^p_{\xh}(\mathbb{R}^3_+)}
        &
        \leq
        Ct^{-(1-\frac{1}{p})-\frac{1}{2}(1-\frac{1}{q})}\| u_0 \|_{{\mathcal{L}^{1}_{\infty}(\mathbb{R}^3_+)}}
        \end{split}
    \end{align*}
    for all $u_0 \in \mathcal{L}^{1}_{\infty}(\mathbb{R}^3_+) $ with $\div u_0=0$.
\end{cor}

Finally, we may use the similar calculations as Lemma \ref{lemm-enh-decay} and the continuous embedding $\mathcal{L}_{1}^{p,q}(\mathbb{R}^3_+) \hookrightarrow \Lv^q\Lh^p(\mathbb{R}^3_+)$ again to obtain the decay estimates for first order derivatives: 
\begin{cor}\label{cor2.7}
 Let $1 \leq p,q \leq \infty$  with $(p,q)\neq (1,1)$ and $1 < r \leq \infty$.
    Then, there exists a positive constant $C=C(p,q,r)$ such that 
    \begin{align*}
        \begin{split}
        \|\p_3 (e^{-tA}u_{0})_{\rm h}
        \|_{L^q_{x_3}L^r_{\xh}(\mathbb{R}^3_+)}
        &
        \leq
        Ct^{-(1-\frac{1}{r})}\| \p_3 u_0 \|_{\mathcal{L}^{1,q}_{\infty}}
        +
        Ct^{-(1-\frac{1}{r})-\frac{1}{2}(1-\frac{1}{q})}\| \p_3 u_0 \|_{{\mathcal{L}^{1}_{\infty}(\mathbb{R}^3_+)}}\\
        \end{split}
    \end{align*}
    for all {$u_0 $} with $\div u_0=0$.
    
    In addition, assume that 
    $1 \leq p,q,r \leq \infty$. Then, there exists a positive constant $C=C(p,q,r)$ such that 
     \begin{align*}
        \begin{split}
        \|\nablah (e^{-tA}u_{0})_{\rm h}
        \|_{L^q_{x_3}L^r_{\xh}(\mathbb{R}^3_+)}
        &
        \leq
        Ct^{-(1-\frac{1}{r})  -\f{1}{2} }\|  u_0 \|_{\mathcal{L}^{1,q}_{\infty}}
        +
        Ct^{-(1-\frac{1}{r})-\frac{1}{2}(1-\frac{1}{q})  -\f{1}{2}  }\|  u_0 \|_{{\mathcal{L}^{1}_{\infty}(\mathbb{R}^3_+)}}\\
        \|\nablah (e^{-tA}u_{0})_3 \|_{L^q_{x_3}L^p_{\xh}(\mathbb{R}^3_+)}
        &
        \leq
        Ct^{-(1-\frac{1}{p})-\frac{1}{2}(1-\frac{1}{q}) -\f{1}{2} }\|  u_0 \|_{{\mathcal{L}^{1}_{\infty}(\mathbb{R}^3_+)}}
        \end{split}
    \end{align*}
     for all {$u_0 $} with $\div u_0=0$.
\end{cor}

\begin{rem}
We note that the divergence free condition yields 
\begin{align*}
    \|\p_3 (e^{-tA}u_{0})_3 \|_{L^q_{x_3}L^p_{\xh}(\mathbb{R}^3_+)}
    ={}&
    \|\nablah \cdot (e^{-tA}u_{0})_{\rm h}\|_{L^q_{x_3}L^p_{\xh}(\mathbb{R}^3_+)}\\
    \leq{}&
    Ct^{-(1-\frac{1}{p})  -\f{1}{2} }\|  u_0 \|_{\mathcal{L}^{1,q}_{\infty}}
    +
    Ct^{-(1-\frac{1}{p})-\frac{1}{2}(1-\frac{1}{q})  -\f{1}{2}  }\|  u_0 \|_{{\mathcal{L}^{1}_{\infty}(\mathbb{R}^3_+)}}.
\end{align*}
\end{rem}

Before ending this subsection, we establish the boundedness of operators $V^{(\pm)},W^{(\pm)},T$ in the 
Chemin--Lerner type space.
\begin{lemm}\label{bdd-VWT}
Let $1 \leq p,q,\sigma \leq \infty$. Then, there exists a positive constant $C=C(p,q)$ such that 
\begin{align*}
& \| V^{(\pm)} f \|_{\mathcal{L}_{\sigma}^{p,q} (\mathbb{R}^3_+) } \leq C \| f \|_{\mathcal{L}_{\sigma}^{p,q} (\mathbb{R}^3_+) }    \\  
& \| W^{(\pm)} f \|_{\mathcal{L}_{\sigma}^{p,q} (\mathbb{R}^3_+) } \leq C \| f \|_{\mathcal{L}_{\sigma}^{p,q} (\mathbb{R}^3_+) }  \\
& \| T f \|_{\mathcal{L}_{\sigma}^{p,q} (\mathbb{R}^3_+) } \leq C \| f \|_{\mathcal{L}_{\sigma}^{p,q} (\mathbb{R}^3_+) }
\end{align*}
for all $f \in \mathcal{L}_{\sigma}^{p,q}(\mathbb{R}^3_+)$.
\end{lemm}
\begin{proof}
Observe that the boundedness of operators $V,W$ follows in the same way as $U$ in Lemma \ref{lemm-U}. Next, we handle the operator $T$. By H\"{o}lder inequality, 
\begin{align*}
    \| \Deltahh_j Tf \|_{L^p(\mathbb{R}^2_{\xh})}
    \leq &\
    C 2^j 
    \int_0^{\infty} 
    e^{-c2^j(x_3+y_3)}
    \| \Deltahh_j f(\cdot,y_3)\|_{L^p(\mathbb{R}^2_{\xh})}
    dy_3  \\
     = &\   C 2^j 
      e^{-c2^j x_3 }  
       \int_0^{\infty} 
    e^{-c2^j  y_3 }
    \| \Deltahh_j f(\cdot,y_3)\|_{L^p(\mathbb{R}^2_{\xh})}
    dy_3   \\
     \leq  &\ C 2^j 
      e^{-c2^j x_3 }  
       \| \Deltahh_j f\|_{\Lv^q\Lh^p(\mathbb{R}^3_+)}
       \|  e^{-c2^j  y_3 }   \|_{L^{ \f{q}{q-1}   } (\mathbb{R}_{y_3>0})}  \\
        \leq  &\ C 2^j
        e^{-c2^j x_3 }   \| \Deltahh_j f\|_{\Lv^q\Lh^p(\mathbb{R}^3_+)} 
        2^{  -j (1-\f{1}{q})        }.
\end{align*}
Then we take $L^q(\mathbb{R}_{x_3>0})$-norm of the above relation to obtain
\begin{align*}
\| \Deltahh_j Tf\|_{\Lv^q\Lh^p(\mathbb{R}^3_+)} 
\leq C 2^j 2^{ -j\f{1}{q} }    \| \Deltahh_j f\|_{\Lv^q\Lh^p(\mathbb{R}^3_+)} 2^{  -j (1-\f{1}{q})        }  =C \| \Deltahh_j f\|_{\Lv^q\Lh^p(\mathbb{R}^3_+)} ,
\end{align*}
which gives the boundedness of $T$ by taking the $\ell^{\sigma}$-norm over the index $j \in \mathbb{Z}$. Finally, since $ V^{(\pm)}, W^{(\pm)}$ are linear combinations of $V,W,T$, the boundedness follows readily.
\end{proof}

\section{Decay estimates for the Duhamel terms in case of \texorpdfstring{$2\leq p,q \leq \infty$}{} }\label{decay-geq-2}

The aim of this section is to establish the decay estimates for the Duhamel terms in Corollary \ref{lemm:sol-form-u} in case of $2\leq p,q \leq \infty$. To this end, we suppose that $u(t)$ satisfies the following assumptions for some $s \in \mathbb{N}, 0<T_{\ast}\leq \infty$ and $ A_{\ast} \geq 0$: 
\begin{description}

\item[\textbf{(A1)}] $u\in C((0,\infty);X^s(\R^3_+)  ), \div u=0$ and for any $t>0$
\[
\| u(t)  \|_{H^s(\R^3_+)} \leq A_{\ast}, \quad \quad 
 \sum_{\alpha_3=0}^1 \| \p_3^{\alphav} u(t) \|_{\mathcal{L}_{\infty}^{1}\cap \mathcal{L}_{\infty}^{1,\infty} (\mathbb{R}^3_+)  }        \leq A_{\ast}(1+t). 
\]
\item[\textbf{(A2)}] For any $\alpha=(\alphah,\alpha_3) \in (\mathbb{N}\cup \{ 0 \} )^2 \times (\mathbb{N}\cup \{ 0 \} ) $ with $|\alpha| \leq 1 $,
\begin{equation} \nonumber 
\| \nabla^{\alpha} \uh(t)   \| _{ \Lv^q\Lh^p (\R^3_+)  }  \leq     C  A_{\ast} \left( t^{ -(1-\f{1}{p})  -\f{|\alphah|}{2}         } +    t^{ - (1-\f{1}{p})-\f{1}{2}(1-\f{1}{q})   -\f{|\alphah|}{2}         }  \right)  
\end{equation}
for all $0<t<T_{\ast}$, $2\leq  p,q \leq \infty$ and     
\begin{equation} \nonumber
 \| \nablah^{\alphah}  u_3(t) \|_{ \Lv^q\Lh^p (\R^3_+)  }  \leq      C A_{\ast}  t^{ - (1-\f{1}{p})-\f{1}{2}(1-\f{1}{q})   -\f{|\alphah|}{2}         }           
\end{equation} 
for all $0<t<T_{\ast}$, $2\leq  p,q \leq \infty$. 
\end{description}

Based on the assumptions above and the basic Sobolev embeddings, it follows immediately that 
\begin{lemm}\label{decay-duhamel-1}
Let $u(t)$ be subject to \textbf{(A1)} and \textbf{(A2)} for some $s \in \mathbb{N} $ with $s \geq 3, 0<T_{\ast}\leq \infty$ and $ A_{\ast} \geq 0$. Then there exists a pure constant $C>0$ such that for all $2\leq  p,q \leq \infty$  
\begin{align*}
     &  \| \uh(t) \| _{ \Lv^q\Lh^p (\R^3_+)  } ,\| \p_3 \uh(t) \| _{ \Lv^q\Lh^p (\R^3_+)  }   \leq     C A_{\ast} (1+t)^{- (1-\f{1}{p})  }  ,          \\
     & \| \nablah \uh(t) \| _{ \Lv^q\Lh^p (\R^3_+)  } ,\| \p_3 u_3 (t) \| _{ \Lv^q\Lh^p (\R^3_+)  }   \leq   C A_{\ast} (1+t)^{- (1-\f{1}{p}) -\f{1}{2}  },     \\
   &      \|  u_3 (t) \| _{ \Lv^q\Lh^p (\R^3_+)  }  \leq    C A_{\ast} (1+t)^{- (1-\f{1}{p})-\f{1}{2}(1-\f{1}{q})  } ,             \\ 
 &      \|  \nablah u_3 (t) \| _{ \Lv^q\Lh^p (\R^3_+)  } \leq  C A_{\ast} (1+t)^{- (1-\f{1}{p})-\f{1}{2}(1-\f{1}{q})-\f{1}{2}  } ,                                
\end{align*}
for all $0<t <T_{\ast}$.   
\end{lemm}

We are now in a position to show the decay estimates for the Duhamel terms. It plays crucial role in proving the decay estimates of nonlinear solutions in case of $2\leq p,q \leq \infty$.  
\begin{lemm}\label{decay-duhamel-2} 
Let $u(t)$ be subject to \textbf{(A1)} and \textbf{(A2)} for some $s \in \mathbb{N} $ with $s \geq 5, 0<T_{\ast}\leq \infty$ and $ A_{\ast} \geq 0$. Then there exists a pure constant $C>0$ such that for any $2\leq p,q \leq \infty, 0<t<T_{\ast}$ and any $\alpha=(\alphah,\alpha_3) \in (\mathbb{N}\cup \{ 0 \} )^2 \times (\mathbb{N}\cup \{ 0 \} ) $ with $|\alpha| \leq 1 $ 
\begin{align*}
 \| \Grad^{\alpha} \mathcal{D}_1^{\rm v}[u](t)    \|_{ \Lv^q\Lh^p (\R^3_+)  } \leq    &     C A_{\ast}^2  (1+t)^ {- (1-\f{1}{p})-\f{1}{2}(1-\f{1}{q})-\f{1+|\alphah|}{2}  }  \log(2+t) ,    \\  
   \| \Grad^{\alpha} \mathcal{D}_2^{\rm v}[u](t)    \|_{ \Lv^q\Lh^p (\R^3_+)  } \leq    &     C A_{\ast}^2  (1+t)^{ - (1-\f{1}{p})-\f{1}{2}(1-\f{1}{q})-\f{1+|\alphah|}{2}   }     ,     \\
   \| \Grad^{\alpha}\mathcal{D}_3^{\rm v}[u](t)    \|_{ \Lv^q\Lh^p (\R^3_+)  } \leq    &    C A_{\ast}^2  (1+t) ^{  - (1-\f{1}{p}) -\f{1+|\alphah|}{2}  } ,    
\end{align*}
and  
\begin{align*}
 \| \Grad^{\alpha} \mathcal{D}_1^{\rm h}[u](t)    \|_{ \Lv^q\Lh^p (\R^3_+)  }  \leq    &     C A_{\ast}^2  (1+t) ^{  - (1-\f{1}{p}) -\f{|\alphah|}{2}    }       ,   \\
   \| \Grad^{\alpha} \mathcal{D}_2^{\rm h}[u](t)    \|_{ \Lv^q\Lh^p (\R^3_+)  }  \leq    &     C A_{\ast}^2  (1+t)^{ - (1-\f{1}{p}) -\f{1+|\alphah|}{2}   }  \log(2+t)  ,    \\
   \| \Grad^{\alpha} \mathcal{D}_3^{\rm h}[u](t)    \|_{ \Lv^q\Lh^p (\R^3_+)  } \leq    &    C A_{\ast}^2  (1+t) ^{- (1-\f{1}{p})-\f{1}{2}(1-\f{1}{q})-\f{1+|\alphah|}{2}  }  \log(2+t)   ,   \\
    \| \Grad^{\alpha} \mathcal{D}_4^{\rm h}[u](t)    \|_{ \Lv^q\Lh^p (\R^3_+)  } \leq    &     C A_{\ast}^2  (1+t)^  {- (1-\f{1}{p})-\f{1}{2}(1-\f{1}{q})-\f{1+|\alphah|}{2}  }      ,  \\
     \| \Grad^{\alpha} \mathcal{D}_5^{\rm h}[u](t)    \|_{ \Lv^q\Lh^p (\R^3_+)  } \leq    &   C A_{\ast}^2  (1+t)^{ - (1-\f{1}{p}) -\f{1+|\alphah|}{2}   }    .          
\end{align*}
\end{lemm}
\begin{proof} 
We start with $\mathcal{D}_1^{\rm h}[u](t)$. Using Lemma \ref{decay-duhamel-1}, one finds for all $2\leq p,q \leq \infty$
\begin{align}
 \|u_3 \uh (\tau)  \|_{\Lv^q \Lh^1 (\R^3_+)  }   & \leq \|u_3(\tau)  \|_{\Lv^{\infty}\Lh^2 (\R^3_+) } \| \uh(\tau) \|_{\Lv^q \Lh^2 (\R^3_+) } \leq CA_{\ast}^2 (1+\tau)^{-\f{3}{2}},     \nonumber      \\
  \|\p_3(u_3 \uh) (\tau)  \|_{\Lv^q \Lh^1 (\R^3_+)  }   & \leq   \|\p_3 u_3(\tau)  \|_{\Lv^{\infty}\Lh^2 (\R^3_+) } \| \uh(\tau) \|_{\Lv^q \Lh^2 (\R^3_+) }       \nonumber    \\ 
  & \quad \quad +  \| u_3(\tau)  \|_{\Lv^{\infty}\Lh^2 (\R^3_+) } \| \p_3 \uh(\tau) \|_{\Lv^q \Lh^2 (\R^3_+) }   \nonumber   \\
    &  \leq C A_{\ast}^2 (1+\tau)^{-\f{3}{2} },   \nonumber  \\
 \| (u_3 \uh) (\tau)  \|_{\Lv^q \Lh^p (\R^3_+)  }   & \leq 
      \| u_3(\tau)  \|_{L^{\infty}(\R^3_+)} \| \uh(\tau) \|_{\Lv^q \Lh^p (\R^3_+)  } \leq CA_{\ast}^2  (1+\tau)^{ -\f{3}{2}-(1-\f{1}{p})   } ,    \nonumber     \\
   \|\nablah (u_3 \uh) (\tau)  \|_{\Lv^q \Lh^p (\R^3_+)  }   & \leq 
      \| \nablah u_3(\tau)  \|_{L^{\infty}(\R^3_+)} \| \uh(\tau) \|_{\Lv^q \Lh^p (\R^3_+)  }     \nonumber\\
      & \quad \quad + \|  u_3(\tau)  \|_{L^{\infty}(\R^3_+)} \| \nablah \uh(\tau) \|_{\Lv^q \Lh^p (\R^3_+)  }  \nonumber\\
      &   \leq CA_{\ast}^2  (1+\tau)^{ -\f{3}{2}-(1-\f{1}{p}) -\f{1}{2}  }  ,     \nonumber  \\
    \|\p_3 (u_3 \uh) (\tau)  \|_{\Lv^q \Lh^p (\R^3_+)  }   & \leq 
      \| \p_3 u_3(\tau)  \|_{L^{\infty}(\R^3_+)} \| \uh(\tau) \|_{\Lv^q \Lh^p (\R^3_+)  }       \label{lem5.2-3}  \\
      & \quad \quad + \|  u_3(\tau)  \|_{L^{\infty}(\R^3_+)} \| \p_3 \uh(\tau) \|_{\Lv^q \Lh^p (\R^3_+)  }  \nonumber\\
      &   \leq   CA_{\ast}^2  (1+\tau)^{ -\f{3}{2}-(1-\f{1}{p}) }  .   \nonumber 
\end{align}
Let us further divide it into $\alpha_3=0$ and $\alpha_3=1$ respectively. If $\alpha_3=0$, we use \eqref{lem5.2-3} to control 
\begin{align*}
    &  \left\| \nablah^{\alphah} \int_0^t   e^{(t-\tau) \Deltah }  \p_3 (u_3 \uh)(\tau)     d \tau \right\|_{\Lv^q \Lh^p (\R^3_+)  }            \\
    &   \quad  \leq \int_0^{\f{t}{2}  } \| \nablah^{\alphah} \Gh(t-\tau) \|_{ L^p(\R^2) } \|\p_3(u_3 \uh)(\tau)   \|_{\Lv^q \Lh^1 (\R^3_+)  }   d\tau   \\
    &   \quad \quad +  \int_{\f{t}{2}}^{t} \| \nablah^{\alphah} \Gh(t-\tau) \|_{ L^1(\R^2) }   \|\p_3 (u_3 \uh) (\tau)  \|_{\Lv^q \Lh^p (\R^3_+)  }  d\tau      \\
    & \quad \leq  CA_{\ast}^2 \left( \int_0^{\f{t}{2}  } (t-\tau)^{ -(1-\f{1}{p}) -\f{|\alphah|}{2}  } (1+\tau)^{-\f{3}{2}}d\tau 
    +   \int_{\f{t}{2}}^{t} (t-\tau)^{-\f{ |\alphah| }{2} } (1+\tau)^{ -\f{3}{2}-(1- \f{1}{p} )  }  d\tau   \right) \\
    & \quad \leq CA_{\ast}^2 t^{ -(1-\f{1}{p}) -\f{|\alphah|}{2}  } 
\end{align*}
for $t \geq 1$ and 
\begin{align*}
    &    \left\| \nablah^{\alphah} \int_0^t   e^{(t-\tau) \Deltah }  \p_3 (u_3 \uh)(\tau)     d \tau \right\|_{\Lv^q \Lh^p (\R^3_+)  }            \\
    &   \quad  \leq \int_0^t  \| \nablah^{\alphah} \Gh(t-\tau) \|_{ L^1(\R^2) }   \|\p_3 (u_3 \uh) (\tau)  \|_{\Lv^q \Lh^p (\R^3_+)  }  d\tau       \\
    &   \quad \leq  CA_{\ast}^2 \int_0^t  (t-\tau)^{-\f{ |\alphah| }{2} } (1+\tau)^{ -\f{3}{2}-(1- \f{1}{p} )  }   d\tau       \\
    & \quad \leq  CA_{\ast}^2  
\end{align*}
for $0<t\leq 1$. Therefore,
\begin{equation}\label{lem5.2-5}
   \left\| \nablah^{\alphah} \int_0^t   e^{(t-\tau) \Deltah }  \p_3 (u_3 \uh)(\tau)     d \tau \right\|_{\Lv^q \Lh^p (\R^3_+)  }    
    \leq  CA_{\ast}^2 (1+t)^{ -(1-\f{1}{p}) -\f{|\alphah|}{2}  }  . 
\end{equation}
To consider the case $\alpha_3=1$, we first deduce from Lemma \ref{decay-duhamel-1} and the Gagliardo-Nirenberg interpolation inequality\footnote{In case of half-space, Gagliardo-Nirenberg interpolation inequality was used in Borchers-Miyakawa \cite{BM88}, Kozono \cite{K89}, Han \cite{Han16}, etc. A detailed proof could be found in Brezis-Mironescu \cite{BrM18}.
} that for all $2\leq q \leq \infty$
\begin{align*}
  \|\p_3^2 (u_3 \uh) (\tau)  \|_{\Lv^q \Lh^1 (\R^3_+)  }  &  \leq  \| \p_3^2 u_3(\tau)  \|_{\Lv^{\infty} \Lh^2 (\R^3_+)  }   \| \uh(\tau) \|_{\Lv^q \Lh^2 (\R^3_+)  }      \\
    &   \quad  +2 \| \p_3 u_3(\tau)  \|_{\Lv^{\infty} \Lh^2 (\R^3_+)  }  \| \p_3 \uh(\tau) \|_{\Lv^{q} \Lh^2 (\R^3_+)  }      \\
    &   \quad     + \| u_3(\tau) \|_{\Lv^q \Lh^2 (\R^3_+)  } \| \p_3^2 \uh(\tau)  \|_{\Lv^{\infty} \Lh^2 (\R^3_+)  }   \\
    & \leq C \Big( \| \p_3 u_3(\tau) \|_{L^2(\R^3_+)}^{\f{5}{8}  } \| \p_3^5 u_3(\tau) \|_{L^2(\R^3_+)}^{\f{3}{8}  } \| \uh(\tau) \|_{\Lv^q \Lh^2 (\R^3_+)  }   \\
    &  \quad + \| \p_3 u_3(\tau)  \|_{\Lv^{\infty} \Lh^2 (\R^3_+)  }  \| \p_3 \uh(\tau) \|_{\Lv^{q} \Lh^2 (\R^3_+)  }  \\
    &   \quad +  \| u_3(\tau) \|_{\Lv^q \Lh^2 (\R^3_+)  } \| \p_3 \uh(\tau) \|_{L^2(\R^3_+)}^{\f{5}{8}  } \| \p_3^5 \uh(\tau) \|_{L^2(\R^3_+)}^{\f{3}{8}  } \Big) \\
    & \leq C A_{\ast}^2 (1+\tau)^{-\f{17}{16}};  
\end{align*} 
\begin{align*}
  \|\p_3^2 (u_3 \uh) (\tau)  \|_{\Lv^q \Lh^2 (\R^3_+)  }  &  \leq  \| \p_3^2 u_3(\tau)  \|_{\Lv^{\infty} \Lh^2 (\R^3_+)  }  \| \uh(\tau) \|_{\Lv^q \Lh^{\infty} (\R^3_+)  }     \\
    &   \quad  +2 \| \p_3 u_3(\tau)  \|_{L^{\infty} (\R^3_+)  }  \| \p_3 \uh(\tau) \|_{\Lv^{q} \Lh^2 (\R^3_+)  }      \\
    &   \quad     + \| u_3(\tau) \|_{\Lv^q \Lh^{\infty} (\R^3_+)  }  \| \p_3^2 \uh(\tau)  \|_{\Lv^{\infty} \Lh^2 (\R^3_+)  }   \\  
     & \leq C \Big( \| \p_3 u_3(\tau) \|_{L^2(\R^3_+)}^{\f{5}{8}  } \| \p_3^5 u_3(\tau) \|_{L^2(\R^3_+)}^{\f{3}{8}  } \| \uh(\tau) \|_{\Lv^q \Lh^{\infty} (\R^3_+)  } \\
    &  \quad + \| \p_3 u_3(\tau)  \|_{L^{\infty} (\R^3_+)  }  \| \p_3 \uh(\tau) \|_{\Lv^{q} \Lh^2 (\R^3_+)  }  \\
    &   \quad +  \| u_3(\tau) \|_{\Lv^q \Lh^{\infty} (\R^3_+)  } \| \p_3 \uh(\tau) \|_{L^2(\R^3_+)}^{\f{5}{8}  } \| \p_3^5 \uh(\tau) \|_{L^2(\R^3_+)}^{\f{3}{8}  } \Big) \\
    & \leq C A_{\ast}^2 (1+\tau)^{-\f{25}{16}}. 
\end{align*} 
Based on these two estimates,  
\begin{align*}
    &  \left\| \p_3 \int_0^t   e^{(t-\tau) \Deltah }  \p_3 (u_3 \uh)(\tau)     d \tau \right\|_{\Lv^q \Lh^p (\R^3_+)  }            \\
    &   \quad  \leq \int_0^{\f{t}{2}  } \| \Gh(t-\tau) \|_{ L^p(\R^2) } \|\p_3^2 (u_3 \uh)(\tau)   \|_{\Lv^q \Lh^1 (\R^3_+)  }   d\tau   \\
    &   \quad \quad +  
    \int_{\f{t}{2}}^{t} \|  \Gh(t-\tau) \|_{ L^{\f{2p}{2+p}} (\R^2) }   \|\p_3^2 (u_3 \uh) (\tau)  \|_{\Lv^q \Lh^2 (\R^3_+)  }  d\tau      \\
    & \quad \leq  CA_{\ast}^2 \left( \int_0^{\f{t}{2}  } (t-\tau)^{ -(1-\f{1}{p})  } (1+\tau)^{-\f{17}{16}}d\tau 
    +  
    \int_{\f{t}{2}}^{t} (t-\tau)^{-(\f{ 1 }{2} -\f{1}{p}) } (1+\tau)^{ -\f{25}{16} }  d\tau   \right)   \\  
    & \quad \leq CA_{\ast}^2 t^{ -(1-\f{1}{p})   } 
\end{align*}
for $t \geq 1$ and
\begin{align*}
    &    \left\|\p_3 \int_0^t   e^{(t-\tau) \Deltah }  \p_3 (u_3 \uh)(\tau)     d \tau \right\|_{\Lv^q \Lh^p (\R^3_+)  }            \\
    &   \quad  
    \leq 
    \int_0^t  \| \Gh(t-\tau) \|_{ L^{\f{2p}{2+p}}(\R^2) }   \|\p_3 (u_3 \uh) (\tau)  \|_{\Lv^q \Lh^2 (\R^3_+)  }  d\tau        \\
    &   \quad 
    \leq 
    CA_{\ast}^2 
    \int_0^t  (t-\tau)^{-(\f{ 1 }{2}-\f{1}{p}  )} (1+\tau)^{ -\f{25}{16} }   d\tau      \\
    & \quad \leq  CA_{\ast}^2  
\end{align*}
for $0<t\leq 1$. Thus, 
\begin{equation}\label{lem5.2-6}
   \left\| \p_3 \int_0^t   e^{(t-\tau) \Deltah }  \p_3 (u_3 \uh)(\tau)     d \tau \right\|_{\Lv^q \Lh^p (\R^3_+)  }    
    \leq  CA_{\ast}^2 (1+t)^{ -(1-\f{1}{p}) }.  
\end{equation}

Next, we treat $\mathcal{D}_2^{\rm h}[u](t)$. Using Lemma \ref{decay-duhamel-1} again, one finds for all $2\leq p,q \leq \infty$ and $k,l\in \{1,2 \}$
\begin{align}
 \|u_k u_l (\tau)  \|_{\Lv^q \Lh^1 (\R^3_+)  }   & \leq \|u_k(\tau)  \|_{\Lv^{\infty}\Lh^2 (\R^3_+) } \| u_l(\tau) \|_{\Lv^q \Lh^2 (\R^3_+) } \leq CA_{\ast}^2 (1+\tau)^{-1},     \nonumber      \\
  \|\p_3(u_k u_l) (\tau)  \|_{\Lv^q \Lh^1 (\R^3_+)  }   & \leq C  \|\p_3 u_k(\tau)  \|_{\Lv^{\infty}\Lh^2 (\R^3_+) } \| u_l(\tau) \|_{\Lv^q \Lh^2 (\R^3_+) }         \nonumber    \\
    &  \leq C A_{\ast}^2 (1+\tau)^{-1},   \nonumber  \\
 \|\Grad^{\alpha} (u_k u_l) (\tau)  \|_{\Lv^q \Lh^p (\R^3_+)  }   & \leq 
    C  \| \Grad^{\alpha}u_k(\tau)  \|_{L^{\infty}(\R^3_+)} \| u_l(\tau) \|_{\Lv^q \Lh^p (\R^3_+)  }   \label{lem5.2-1}  \\
    & \leq CA_{\ast}^2  (1+\tau)^{ -1-\f{|\alphah|}{2}   } (1+\tau)^{-(1-\f{1}{p})}   \nonumber  \\
    & = CA_{\ast}^2  (1+\tau)^{ -(2- \f{1}{p} )-\f{|\alphah|}{2}   }.   \nonumber 
\end{align}
It then follows from \eqref{lem5.2-1} that 
\begin{align*}
    &  \left\| \Grad^{\alpha} \int_0^t   e^{(t-\tau) \Deltah }  \div_{\rm h} (\uh \otimes \uh)(\tau)     d \tau \right\|_{\Lv^q \Lh^p (\R^3_+)  }            \\
    &   \quad  \leq \int_0^{\f{t}{2}  } \| \nablah^{\alphah}\nablah \Gh(t-\tau) \|_{ L^p(\R^2) } \|\p_3^{\alpha_3}(u_k u_l)(\tau)   \|_{\Lv^q \Lh^1 (\R^3_+)  }   d\tau   \\
    &   \quad \quad +  \int_{\f{t}{2}}^{t} \| \nablah \Gh(t-\tau) \|_{ L^1(\R^2) }   \|\Grad^{\alpha} (u_k u_l) (\tau)  \|_{\Lv^q \Lh^p (\R^3_+)  }  d\tau      \\
    & \quad \leq  CA_{\ast}^2 \left( \int_0^{\f{t}{2}  } (t-\tau)^{ -(1-\f{1}{p}) -\f{1+|\alphah|}{2}  } (1+\tau)^{-1}d\tau 
    +   \int_{\f{t}{2}}^{t} (t-\tau)^{-\f{1}{2}} (1+\tau)^{ -(2- \f{1}{p} )-\f{|\alphah|}{2}   }  d\tau   \right) \\
    & \quad \leq CA_{\ast}^2 t^{ -(1-\f{1}{p}) -\f{1+|\alphah|}{2}  } \log(2+t)
\end{align*}
for $t \geq 1$ and 
\begin{align*}
    &    \left\| \Grad^{\alpha} \int_0^t   e^{(t-\tau) \Deltah }  \div_{\rm h} (\uh \otimes \uh)(\tau)     d \tau \right\|_{\Lv^q \Lh^p (\R^3_+)  }           \\
    &   \quad  \leq \int_0^t \| \nablah \Gh(t-\tau) \|_{ L^1(\R^2) }   \|\Grad^{\alpha} (u_k u_l) (\tau)  \|_{\Lv^q \Lh^p (\R^3_+)  }  d\tau       \\
    &   \quad \leq  CA_{\ast}^2 \int_0^t   (t-\tau)^{-\f{1}{2}} (1+\tau)^{ -(2- \f{1}{p} )-\f{|\alphah|}{2}   }  d\tau       \\
    & \quad \leq  CA_{\ast}^2
\end{align*}
for $0<t\leq 1$. Thus we see
\begin{equation}\label{lem5.2-2}
    \left\| \Grad^{\alpha} \int_0^t   e^{(t-\tau) \Deltah } \div_{\rm h} (\uh \otimes \uh)(\tau)     d \tau \right\|_{\Lv^q \Lh^p (\R^3_+)  }    
    \leq CA_{\ast}^2 (1+t)^{ -(1-\f{1}{p}) -\f{1+|\alphah|}{2}  } \log(2+t). 
\end{equation}

We next focus on $\mathcal{D}_3^{\rm h}[u](t)$. 
Due to the operator $W^{(+)}$ and Riesz operator therein, we cannot work in the anisotropic Lebesgue space $\Lv^{q}\Lh^p (\R^3_+)$. Instead, we adopt the Chemin--Lerner type space $\mathcal{L}^{p,q}_1(\mathbb{R}^3_+)$.
Using the definition of $W^{(+)}$, 
\begin{align*}
    &  W^{(+)}   e^{(t - \tau)\Deltah}
    \nablah
    \sum_{\ell,m=1}^2
    {S_{x_\ell}S_{x_m}}   
    (u_{\ell}u_{m})(\tau)         \\
    &   \quad  = \f{1}{2}\int_0^{\infty} 
       |\nablah|^{\f{1}{q}} e^{- |\nablah| |x_3-y_3|}      \Big[ |\nablah|^{1-\f{1}{q}}  e^{(t - \tau)\Deltah}  \nablah  \sum_{\ell,m=1}^2
    {S_{x_\ell}S_{x_m}}   
    (u_{\ell}u_{m})       \Big] (\tau,\xh,y_3) d y_3              \\
    &   \quad \quad +  \f{1}{2}\int_0^{\infty} 
     |\nablah|^{\f{1}{q}} e^{- |\nablah| (x_3+y_3)}  
     \Big[ |\nablah|^{1-\f{1}{q}} e^{(t - \tau)\Deltah}  \nablah  \sum_{\ell,m=1}^2
    {S_{x_\ell}S_{x_m}}   
    (u_{\ell}u_{m})       \Big] (\tau,\xh,y_3) d y_3   \\
    & \quad =: I_1 +I_2.
\end{align*}
We first estimate $I_1$. It follows from Lemma \ref{app-lemm-2} that for any $\ell,m\in \{1,2\}$  
\begin{align*}
   &
   \| \Delta_j^{\rm h}    I_1    (\tau,\cdot,x_3)\|_{L^p(\mathbb{R}^2)}\\
    &\quad \leq{}
    C  
    \int_0^{\infty}
    2^{\frac{1}{q}j}e^{-c2^j |x_3-y_3| }
    \left\| \left[\Delta_j^{\rm h}|\nablah|^{1-\frac{1}{q}}e ^{(t-\tau) \Delta_{\rm h}}  \nablah  (u_{\ell}u_{m} )  \right](\tau,\cdot,y_3)  \right\|_{L^p(\mathbb{R}^2)}dy_3.
\end{align*}
With the help of the Hausdorff--Young inequality, 
\begin{align}
    &\| I_1 (\tau)  \|_{\mathcal{L}^{p,q}_1(\mathbb{R}^3_+)}  \nonumber \\
    &\quad \leq{}
    C
    \sum_{j \in \mathbb{Z}}
    2^{\frac{1}{q}j}\|e^{-c2^j  |z_3|  }\|_{L^q_{z_3}(0,\infty)}
   \left \| \Delta_j^{\rm h}|\nablah|^{1-\frac{1}{q}}e ^{(t-\tau) \Delta_{\rm h}}  \nablah (u_{\ell}u_{m} )  (\tau)  \right\|_{L^1_{z_3}(0,\infty;L^p(\mathbb{R}^2))} \nonumber \\
    &\quad \leq{}
    C
    \sum_{j \in \mathbb{Z}}
    \left\| \Delta_j^{\rm h}|\nablah|^{1-\frac{1}{q}}e ^{(t-\tau) \Delta_{\rm h}}  \nablah (u_{\ell}u_{m} )  (\tau)  \right\|_{L^1_{x_3}(0,\infty;L^p(\mathbb{R}^2_{x_{\rm h}}))} \nonumber \\
    &\quad\leq{}
    C 
    \| e ^{(t-\tau) \Delta_{\rm h}}  \nablah (u_{\ell}u_{m} ) 
  (\tau) 
 \|_{{L^1_{x_3}}(0,\infty;\dot{B}_{p,1}^{1-\frac{1}{q}}(\mathbb{R}^2_{x_{\rm h}}))}.  \label{lem5.2-9}
\end{align}
On the one hand, for all $2\leq p,q \leq \infty$, we see 
\begin{align}
    &\| e ^{(t-\tau) \Delta_{\rm h}} \nablah  (u_{\ell}u_{m} )(\tau) \|_{L^1_{x_3}(0,\infty;\dot{B}_{p,1}^{1-\frac{1}{q}}(\mathbb{R}^2_{x_{\rm h}}))}   \nonumber \\
    &\quad=
    \sum_{j \in \mathbb{Z}}
    2^{(1-\frac{1}{q})j}
    \| 
    \Deltahh_j  e ^{(t-\tau) \Delta_{\rm h}}  \nablah  (u_{\ell}u_{m} ) (\tau) 
    \|_{L^1_{x_3}(0,\infty;L^p(\mathbb{R}^2_{x_{\rm h}}))} \nonumber \\ 
    &\quad\leq
    C
    \sum_{j \in \mathbb{Z}}
    2^{(1-\frac{1}{q})j}e^{-c2^{2j}(t-\tau)} 2^{j} 
    \| 
    \Deltahh_j (u_{\ell}u_{m})(\tau) 
    \|_{L^1_{x_3}(0,\infty;L^p(\mathbb{R}^2_{x_{\rm h}}))}  \nonumber \\
    &\quad\leq
    C
    \sum_{j \in \mathbb{Z}}
    2^{\{2(1-\frac{1}{p})+(1-\frac{1}{q})+1  \}j}e^{-c2^{2j} (t-\tau) }
    \| 
    \Deltahh_j (u_{\ell}u_{m})(\tau) 
    \|_{L^1(\mathbb{R}^3_+)}  \label{lem5.2-10}       \\
    &\quad\leq
    C
    \left( 
    \sum_{j \in \mathbb{Z}}
    2^{\{2(1-\frac{1}{p})+(1-\frac{1}{q}) +1\}j}e^{-c2^{2j}(t-\tau) }
    \right) 
    \sup_{j \in \mathbb{Z}}
    \| 
    \Deltahh_j (u_{\ell}u_{m})(\tau) 
    \|_{L^1(\mathbb{R}^3_+)}  \nonumber \\
    &\quad \leq{}
    C
    (t-\tau)^{-(1-\frac{1}{p})-\frac{1}{2}(1-\frac{1}{q})-\f{1}{2}}
    \| (u_{\ell}u_{m})(\tau)  \|_{\mathcal{L}^1_{\infty}(\mathbb{R}^3_+)}   \nonumber \\
    &\quad \leq {}
    C
    (t-\tau)^{-(1-\frac{1}{p})-\frac{1}{2}(1-\frac{1}{q})-\f{1}{2}}
    \|(u_{\ell}u_{m})(\tau) \|_{L^1(\mathbb{R}^3_+)};  \nonumber
\end{align}
on the other hand, for all $2\leq p,q \leq \infty$, it holds 
\begin{align}
    &\| e ^{(t-\tau) \Delta_{\rm h}} \nablah  (u_{\ell}u_{m})(\tau) \|_{L^1_{x_3}(0,\infty;\dot{B}_{p,1}^{1-\frac{1}{q}}(\mathbb{R}^2_{x_{\rm h}}))}  \nonumber \\
    &\quad=
    \sum_{j \in \mathbb{Z}}
    2^{(1-\frac{1}{q})j}
    \| 
    \Deltahh_j  e ^{(t-\tau) \Delta_{\rm h}} \nablah  (u_{\ell}u_{m})(\tau) 
    \|_{L^1_{x_3}(0,\infty;L^p(\mathbb{R}^2_{x_{\rm h}}))}  \nonumber \\ 
    &\quad\leq
    C
    \sum_{j \in \mathbb{Z}}
    2^{(1-\frac{1}{q})j}e^{-c2^{2j}(t-\tau)} 
    \| 
    \Deltahh_j (\nablah  (u_{\ell}u_{m}) ) (\tau)  
    \|_{L^1_{x_3}(0,\infty;L^p(\mathbb{R}^2_{x_{\rm h}}))} \label{lem5.2-11} \\
     &\quad\leq
    C
   \left(  \sum_{j \in \mathbb{Z}}
    2^{(1-\frac{1}{q})j}e^{-c2^{2j}(t-\tau)}  \right) \sup_{j\in \mathbb{Z}} 
    \| 
    \Deltahh_j (\nablah  (u_{\ell}u_{m}) ) (\tau)  
    \|_{L^1_{x_3}(0,\infty;L^p(\mathbb{R}^2_{x_{\rm h}}))}  \nonumber \\
    &\quad\leq
    C
    \left( \sum_{j \in \mathbb{Z}}
     2^{(1-\frac{1}{q})j} e^{-c2^{2j} (t-\tau) } \right)
    \| 
    \nablah  (u_{\ell}u_{m})(\tau) 
    \|_{L^1_{x_3}(0,\infty;L^p(\mathbb{R}^2_{x_{\rm h}}))}  \nonumber  \\
    &\quad \leq {}
    C
    (t-\tau)^{-\frac{1}{2}(1-\frac{1}{q}) }
   \| 
    \nablah  (u_{\ell}u_{m})(\tau) 
    \|_{L^1_{x_3}(0,\infty;L^p(\mathbb{R}^2_{x_{\rm h}}))} .  \nonumber
\end{align}
Next, it follows from Lemma \ref{decay-duhamel-1} that 
\begin{align}
   \|(u_{\ell}u_{m})(\tau) \|_{L^1(\mathbb{R}^3_+)}     & \leq \|\uh(\tau) \|_{L^2(\mathbb{R}^3_+)}  \|\uh(\tau) \|_{L^2(\mathbb{R}^3_+)}    \leq  CA_{\ast}^2 (1+\tau)^{-1},     \label{lem5.2-12} \\    
  \| 
    \nablah  (u_{\ell}u_{m})(\tau) 
    \|_{\Lv^{1} \Lh^{p} (\R^3_+) }    &   \leq 2  \| \nablah \uh \|_{\Lv^{2} \Lh^{\infty} (\R^3_+) }   \|  \uh \|_{\Lv^{2} \Lh^{p} (\R^3_+) }   \leq CA_{\ast}^2 (1+\tau)^{- (\f{5}{2} -\f{1}{p})  } . \nonumber    
\end{align}
Collecting \eqref{lem5.2-9}-\eqref{lem5.2-12}, 
\begin{align*}
    &  \left\|    \int_0^t I_1 (\tau)        d \tau   \right\|_{\mathcal{L}^{p,q}_1(\mathbb{R}^3_+)}          \\
    &  \quad  \leq  \int_0^{\f{t}{2}}  \| I_1  (\tau)  \|_{\mathcal{L}^{p,q}_1(\mathbb{R}^3_+)}      d \tau
    + \int_{\f{t}{2}}^{t} \| I_1  (\tau)  \|_{\mathcal{L}^{p,q}_1(\mathbb{R}^3_+)}    d \tau          \\
    &   \quad \leq  C  \int_0^{\f{t}{2}}  (t-\tau)^{-(1-\frac{1}{p})-\frac{1}{2}(1-\frac{1}{q})-\f{1}{2}}
    \|(u_{\ell}u_{m})(\tau) \|_{L^1(\mathbb{R}^3_+)}  d \tau   \\
    & \quad \quad +C \int_{\f{t}{2}}^{t}    (t-\tau)^{-\frac{1}{2}(1-\frac{1}{q}) }
   \| 
    \nablah  (u_{\ell}u_{m})(\tau) 
    \|_{\Lv^{1} \Lh^{p} (\R^3_+) }    d\tau \\
    & \quad \leq CA_{\ast}^2 \left(  \int_0^{\f{t}{2}}  (t-\tau)^{-(1-\frac{1}{p})-\frac{1}{2}(1-\frac{1}{q})-\f{1}{2}}
   (1+\tau)^{-1}  d \tau  +  \int_{\f{t}{2}}^{t}    (t-\tau)^{-\frac{1}{2}(1-\frac{1}{q}) }   (1+\tau)^{ -(\f{5}{2} -\f{1}{p})   }       d\tau    \right) \\
   & \quad \leq CA_{\ast}^2  t^{-(1-\frac{1}{p})-\frac{1}{2}(1-\frac{1}{q})-\f{1}{2}} \log(2+t)  
\end{align*}
for all $t\geq 1$ and 
\begin{align*}  
    &    \left\|    \int_0^t  I_1 (\tau)        d \tau   \right\|_{\mathcal{L}^{p,q}_1(\mathbb{R}^3_+)}          
     \leq \int_0^t    (t-\tau)^{-\frac{1}{2}(1-\frac{1}{q}) }
   \| 
    \nablah  (u_{\ell} u_m)(\tau) 
    \|_{\Lv^{1} \Lh^{p} (\R^3_+) }    d\tau      \\
    &   \quad  \leq  CA_{\ast}^2 \int_0^t    (t-\tau)^{-\frac{1}{2}(1-\frac{1}{q}) }   (1+\tau)^{ -(\f{5}{2} -\f{1}{p}) }   d\tau      \leq 
     CA_{\ast}^2
\end{align*}
for all $0<t\leq 1$. These two relations together imply
\begin{equation}\label{lem5.2-13} 
     \left\|    \int_0^t I_1(\tau)        d \tau   \right\|_{\mathcal{L}^{p,q}_1(\mathbb{R}^3_+)} 
      \leq CA_{\ast}^2  (1+t)^{-(1-\frac{1}{p})-\frac{1}{2}(1-\frac{1}{q})-\f{1}{2}} \log(2+t) . 
\end{equation}
Observe that $I_2$ is bounded similarly. Indeed, it suffices to notice 
\begin{align*}
   &
   \| \Delta_j^{\rm h}    I_2   (\tau,\cdot,x_3)\|_{L^p(\mathbb{R}^2)}\\
    &\quad \leq{}
    C   2^{\frac{1}{q}j}   e^{-c2^j x_3 }
    \int_0^{\infty}
   e^{-c2^j y_3 }
    \left\| \left[\Delta_j^{\rm h}|\nablah|^{1-\frac{1}{q}}e ^{(t-\tau) \Delta_{\rm h}}  \nablah  (u_{\ell}u_{m} )  \right](\tau,\cdot,y_3)  \right\|_{L^p(\mathbb{R}^2)}dy_3  \nonumber \\  
    & \quad \leq{}
    C   2^{\frac{1}{q}j}   e^{-c2^j x_3 } 
     \left \| \Delta_j^{\rm h}|\nablah|^{1-\frac{1}{q}}e ^{(t-\tau) \Delta_{\rm h}}  \nablah (u_{\ell}u_{m} )  (\tau)  \right\|_{L^1_{y_3}(0,\infty;L^p(\mathbb{R}^2))} ;    \nonumber
\end{align*}
whence
\begin{align}
    &\| I_2 (\tau)  \|_{\mathcal{L}^{p,q}_1(\mathbb{R}^3_+)}  \nonumber \\
    &\quad \leq{}
    C
    \sum_{j \in \mathbb{Z}}
    2^{\frac{1}{q}j}\|e^{-c 2^j  x_3  }\|_{L^q_{x_3}(0,\infty)}
   \left \| \Delta_j^{\rm h}|\nablah|^{1-\frac{1}{q}}e ^{(t-\tau) \Delta_{\rm h}}  \nablah (u_{\ell}u_{m} )  (\tau)  \right\|_{L^1_{y_3}(0,\infty;L^p(\mathbb{R}^2))} \nonumber \\
    &\quad \leq{}
    C
    \sum_{j \in \mathbb{Z}}
    \left\| \Delta_j^{\rm h}|\nablah|^{1-\frac{1}{q}}e ^{(t-\tau) \Delta_{\rm h}}  \nablah (u_{\ell}u_{m} )  (\tau)  \right\|_{L^1_{x_3}(0,\infty;L^p(\mathbb{R}^2_{x_{\rm h}}))} \nonumber \\
    &\quad\leq{}
    C 
    \| e ^{(t-\tau) \Delta_{\rm h}}  \nablah (u_{\ell}u_{m} ) 
  (\tau) 
 \|_{{L^1_{x_3}}(0,\infty;\dot{B}_{p,1}^{1-\frac{1}{q}}(\mathbb{R}^2_{x_{\rm h}}))}. \nonumber
\end{align}
This gives the decay estimate of $\|\Grad^{\alpha} \mathcal{D}_3^{\rm h}[u](t) \|_{\mathcal{L}^{p,q}_1(\mathbb{R}^3_+)} $ in case of $|\alpha|=0$. In order to obtain the case of $|\alpha|=1$, we invoke the following useful property
\begin{align*}
    &    \p_j [ W^{(+)} f  ]= W^{(+)} [\p_j f] , \,\,\,\, j=1,2,       \\
    &  \p_3 [ W^{(+)} f  ]= \f{1}{2}  \int_0^{\infty}
    |\nablah| e^{- |\nablah| |x_3-y_3|  } \p_3 f (\xh,y_3)
    d y_3   \\
    &   \quad \quad \quad \quad  \quad \quad \quad   - |\nablah|  \f{1}{2}  \int_0^{\infty}        |\nablah| e^{- |\nablah| (x_3+y_3)  }    f (\xh,y_3)
    d y_3  .
\end{align*}
The details are omitted here. 

Calculating as above, one gets the decay estimates for the remaining integrals. We shall not reproduce the details for brevity. The proof of Lemma \ref{decay-duhamel-2} is finished.
\end{proof}

\section{Proof of Theorem \ref{main-theorem} in case of \texorpdfstring{$2\leq p,q \leq \infty$}{} } \label{main-thm-1}

To begin with, we state the following global well-posedness for the anisotropic Navier--Stokes equations in the half-space. 
\begin{lemm}\label{strong-sol}
Let $k \geq 2$ be an integer. There exists a positive constant $\delta_0=\delta_0(k)$ such that for any $u_0 \in H^k(\R^3_+)$ with $\div u_0 =0, u_0|_{\p \R^3_+}=0,  \|  u_0 \|_{H^k(\R^3_+)   } \leq \delta_0$, \eqref{eq1} admits a unique solution $u \in C ([0,\infty);  H^k(\R^3_+) )$. Moreover,
\[
\|  u(t) \|_{H^k(\R^3_+)   }^2 + \int_0^t \| \nablah u(\tau) \|_{H^k(\R^3_+)   }^2  d \tau \leq C \|  u_0 \|_{H^k(\R^3_+)   }^2
\]
for any $t\geq 0$. 
\end{lemm}
The proof of Lemma \ref{strong-sol} follows from \cite{JWY21} with slight modifications. The details are thus omitted. In order to obtain the decay estimates of solutions, we additionally assume that the initial velocity belongs to the Chemin--Lerner type space and further obtain the following lemma.

\begin{lemm}\label{strong-sol-2}
Let $s \geq 3$ be an integer. There exists a positive constant $\delta_0=\delta_0(s)$ such that for any $u_0 \in X^s(\R^3_+)$ with $\div u_0 =0,u_0|_{\p \R^3_+}=0, \|  u_0 \|_{H^s(\R^3_+)   } \leq \delta_0$, \eqref{eq1} admits a unique solution { $u \in C ((0,\infty);  X^s(\mathbb{R}^3_+) )$}. 
\end{lemm}

\begin{proof}
On the one hand, we consider the linear part. We see by Corollary \ref{cor:semigr} and the crucial relation $L^2(\mathbb{R}^3_+) = \mathcal{L}_2^{2,2}(\mathbb{R}^3_+)$, which indeed can be proved by a similar manner of the proof for the classical relation $L^2(\mathbb{R}^d)=\dot{B}_{2,2}^0(\mathbb{R}^d)$, that
\begin{align*}
    \| e^{-tA}u_0 \|_{X^s(\R^3_+)} 
    &
    \leq
    C
    \sum_{|\alpha| \leq s}
    \| \nabla^{\alpha} e^{-tA}u_0 \|_{\mathcal{L}_2^{2,2}(\mathbb{R}^3_+)}
    +
    \sum_{\alpha_3=0}^1 \| \p_3^{\alphav} e^{-tA}u_0 \|_{\mathcal{L}_{\infty}^{1}\cap \mathcal{L}_{\infty}^{1,\infty} (\R^3_+)    }\\
    &
    \leq
    C
    \sum_{|\alpha| \leq s}
    \| \nabla^{\alpha}u_0 \|_{\mathcal{L}_2^{2,2}(\mathbb{R}^3_+)}
    +
    C
    \sum_{\alpha_3=0}^1 \| \p_3^{\alphav} u_0 \|_{\mathcal{L}_{\infty}^{1}\cap \mathcal{L}_{\infty}^{1,\infty}  (\R^3_+)   }\\
    &
    = C \| u_0 \|_{X^s   (\R^3_+)  }.   
\end{align*}

On the other hand, we handle the nonlinear part. 
We also see by Lemmas \ref{lemm-U}-\ref{lemm-heat}-\ref{bdd-VWT}, Corollary \ref{lemm:sol-form-u} and the continuous embedding $L^1(\mathbb{R}^3_+) \hookrightarrow \mathcal{L}^1_{\infty}(\mathbb{R}^3_+)$ that 
\begin{align*}
  &  \sum_{\alpha_3=0}^{1}
    \left \|
    \p_3^{\alpha_3}
    \int_0^t 
    e^{-(t-\tau)A}\mathbb{P}\div (u \otimes u)(\tau) d\tau
    \right \|_{\mathcal{L}^1_{\infty}  (\R^3_+) 
 }  \\
    &  \quad \quad  \leq
    C
    \sum_{|\alpha|\leq 2}
    \sum_{k,\ell=1}^3
    \int_0^t
    \| \nabla^{\alpha} (u_ku_{\ell})(\tau)\|_{\mathcal{L}^1_{\infty} (\R^3_+) 
 }d\tau\\
    & \quad \quad      \leq
    C
    \sum_{|\alpha|\leq 2}
    \sum_{k,\ell=1}^3
    \int_0^t
    \| \nabla^{\alpha} (u_ku_{\ell})(\tau)\|_{L^1 (\R^3_+)  }d\tau\\
    & \quad \quad  \leq
    C
    \int_0^t
    \| u(\tau) \|_{H^2(\R^3_+)}  ^2 d\tau\\
    &  \quad \quad   \leq
    C t \| u_0 \|_{H^2 (\R^3_+) }^2      \,\,\,  \text{   (due to Lemma \ref{strong-sol}) }                 \\
     & \quad \quad    \leq  Ct \| u_0\|_{X^s (\R^3_+) } . \,\,\, \text{   (due to $\|  u_0 \|_{H^s(\R^3_+)   } \leq \delta_0$) }   
\end{align*}
By Lemmas \ref{lemm-U}-\ref{lemm-heat}-\ref{bdd-VWT} and the continuous embeddings $\Lh^1\Lv^{\infty}(\mathbb{R}^3_+) \hookrightarrow \Lv^{\infty}\Lh^1(\mathbb{R}^3_+) \hookrightarrow \mathcal{L}^{1,\infty}_{\infty}(\mathbb{R}^3_+)$ and $H^1(\mathbb{R}_{x_3 >0}) \hookrightarrow L^{\infty}(\mathbb{R}_{x_3>0})$ that 
\begin{align*}
  &  \sum_{\alpha_3=0}^{1}
    \left \|
    \p_3^{\alpha_3}
    \int_0^t 
    e^{-(t-\tau)A}\mathbb{P}\div ( u \otimes u)(\tau) d\tau
    \right \|_{\mathcal{L}^{1,\infty}_{\infty}(\R^3_+)} \\
    & \quad \quad  \leq
    C
    \sum_{|\alpha|\leq 2}
    \sum_{k,\ell=1}^3
    \int_0^t
    \| \nabla^{\alpha} (u_ku_{\ell})(\tau)\|_{\mathcal{L}^{1,\infty}_{\infty}(\R^3_+)}d\tau\\
    &\quad \quad   \leq
    C
    \sum_{|\alpha|\leq 2}
    \sum_{k,\ell=1}^3
    \int_0^t
    \| \nabla^{\alpha} (u_k u_{\ell})(\tau)\|_{\Lh^1\Lv^{\infty}(\R^3_+)}d\tau\\
    & \quad \quad  \leq
    C
    \sum_{|\alpha|\leq 2}
    \int_0^t
    \| \nabla^{\alpha}u(\tau) \|_{\Lh^2\Lv^{\infty}(\R^3_+)}^2 d\tau\\
    &\quad \quad \leq
    C
    \sum_{|\alpha|\leq 3}
    \int_0^t
    \| \nabla^{\alpha}u(\tau) \|_{L^2(\R^3_+)}^2 d\tau\\
  & \quad \quad \leq
    Ct \| u_0 \|_{H^3(\R^3_+)}^2   \,\,\,  \text{   (due to Lemma \ref{strong-sol}) }          \\
     &\quad \quad \leq  Ct \| u_0\|_{X^s(\R^3_+)} . \,\,\, \text{   (due to $\|  u_0 \|_{H^s(\R^3_+)   } \leq \delta_0$) }       
\end{align*}
Thus, we complete the proof.
\end{proof}

In order to complete the proof of Theorem \ref{main-theorem} in case of $2\leq p,q\leq \infty$, it remains to show the optimal decay estimates. We adopt the classical bootstrapping argument. To this end, we know from the linear estimates, namely Corollaries \ref{cor2.6}-\ref{cor2.7}, that there exists $C_1>0$ such that for any $\alpha=(\alphah,\alpha_3) \in (\mathbb{N}\cup \{ 0 \} )^2 \times (\mathbb{N}\cup \{ 0 \} ) $ with $|\alpha| \leq 1 $
\begin{align}
    &    \| \nabla^{\alpha} (e^{-tA}u_{0})_{\rm h}   \| _{\Lv^q\Lh^p (\R^3_+) }  \leq      C_1  \left(  t^{ -(1-\f{1}{p})  -\f{|\alphah|}{2}         } +   t^{ - (1-\f{1}{p})-\f{1}{2}(1-\f{1}{q})   -\f{|\alphah|}{2}         }   \right) \| u_0 \|_{X^s(\R^3_+)   }    , \nonumber     \\
    &  \| \nablah^{\alphah}  (e^{-tA}u_{0})_{3} \|_{\Lv^q\Lh^p (\R^3_+)}  \leq       C_1  t^{ - (1-\f{1}{p})-\f{1}{2}(1-\f{1}{q})   -\f{|\alphah|}{2}         }  \| u_0 \|_{X^s(\R^3_+)   } \label{0.1}
\end{align}
for all $0<t<T_{\ast}$, $2\leq  p,q \leq \infty$.

Suppose that for any $\alpha=(\alphah,\alpha_3) \in (\mathbb{N}\cup \{ 0 \} )^2 \times (\mathbb{N}\cup \{ 0 \} ) $ with $|\alpha| \leq 1 $
\begin{align}
    &    \| \nabla^{\alpha} \uh(t)   \| _{\Lv^q\Lh^p (\R^3_+) }  \leq     2 C_1  \left(  t^{ -(1-\f{1}{p})  -\f{|\alphah|}{2}         } +   t^{ - (1-\f{1}{p})-\f{1}{2}(1-\f{1}{q})   -\f{|\alphah|}{2}         }   \right) \| u_0 \|_{X^s(\R^3_+)   }    ,   \nonumber   \\
    &  \| \nablah^{\alphah}  u_3(t) \|_{\Lv^q\Lh^p (\R^3_+)}  \leq      2 C_1  t^{ - (1-\f{1}{p})-\f{1}{2}(1-\f{1}{q})   -\f{|\alphah|}{2}         }  \| u_0 \|_{X^s(\R^3_+)   }  \label{0.2} 
\end{align}
for all $0<t<T_{\ast}$, $2\leq  p,q \leq \infty$. Based on \eqref{0.1}, it then follows from Corollary \ref{lemm:sol-form-u} and Lemma \ref{decay-duhamel-2}, upon setting $A_{\ast}=2 C_1 \| u_0 \|_{X^s(\R^3_+)   } $, that 
for any $\alpha=(\alphah,\alpha_3) \in (\mathbb{N}\cup \{ 0 \} )^2 \times (\mathbb{N}\cup \{ 0 \} ) $ with $|\alpha| \leq 1 $
\begin{align*}
   \| \nabla^{\alpha} \uh(t)   \| _{\Lv^q\Lh^p (\R^3_+) }  &     \leq      C_1  \left(  t^{ -(1-\f{1}{p})  -\f{|\alphah|}{2}         } +   t^{ - (1-\f{1}{p})-\f{1}{2}(1-\f{1}{q})   -\f{|\alphah|}{2}         }   \right) \| u_0 \|_{X^s(\R^3_+)   }          \\
    &  \quad + C_2 (1+t)^{-(1-\f{1}{p})  -\f{|\alphah|}{2}   }     \| u_0 \|_{X^s(\R^3_+)   } ^2      \\
    &  \leq \f{3}{2} C_1         \left(  t^{ -(1-\f{1}{p})  -\f{|\alphah|}{2}         } +   t^{ - (1-\f{1}{p})-\f{1}{2}(1-\f{1}{q})   -\f{|\alphah|}{2}         }   \right) \| u_0 \|_{X^s(\R^3_+)   }   ,                  \\
  \| \nablah^{\alphah}  u_3(t) \|_{\Lv^q\Lh^p (\R^3_+)}  &    \leq       C_1  t^{ - (1-\f{1}{p})-\f{1}{2}(1-\f{1}{q})   -\f{|\alphah|}{2}         }  \| u_0 \|_{X^s(\R^3_+)   }   \\
  & \quad + C_2 (1+t)^{  - (1-\f{1}{p})-\f{1}{2}(1-\f{1}{q})   -\f{|\alphah|}{2}      }   \| u_0 \|_{X^s(\R^3_+)   } ^2  \\
  &   \leq \f{3}{2} C_1 t^{ - (1-\f{1}{p})-\f{1}{2}(1-\f{1}{q})   -\f{|\alphah|}{2}         }  \| u_0 \|_{X^s(\R^3_+)   } 
\end{align*}
for all $0<t<T_{\ast}$, $2\leq  p,q \leq \infty$. Here, we have chosen the initial data so small that 
\begin{equation}\label{0.3}
\| u_0 \|_{X^s(\R^3_+)   } \leq \min\left\{ \delta_0,\f{C_1}{2C_2} \right\}. 
\end{equation}

Therefore, we have obtained a sharper bound than our hypothesis. The bootstrapping argument then assesses that the optimal decay estimates \eqref{0.2} actually hold for all $t>0$ under the smallness condition \eqref{0.3}. Furthermore, we infer from Lemma \ref{strong-sol} and Sobolev embeddings that 
\begin{align}
    \| \nabla^{\alpha} \uh(t)   \| _{\Lv^q\Lh^p (\R^3_+) } ,
     \| \nablah^{\alphah}  u_3(t) \|_{\Lv^q\Lh^p (\R^3_+)}    \leq   C \| u_0 \|_{H^s(\R^3_+)   } \label{0.4} 
\end{align}
for all $0<t<T_{\ast}$, $2\leq  p,q \leq \infty$. 
Combining \eqref{0.2} and \eqref{0.4}, we obtain \eqref{decay-h-1} and \eqref{decay-3-1}.


\section{Decay estimates for the Duhamel terms in case of \texorpdfstring{$1 \leq  p,q < 2$}{} }\label{decay-less-2}

Based on Corollary \ref{lemm:sol-form-u} and the decay estimates of nonlinear solutions obtained in Section \ref{main-thm-1}, we derive the decay estimates of Duhamel terms in case of $1 \leq  p,q < 2$. 

\begin{lemm}\label{decay-duhamel-leq-2} 
There exists a pure constant $C>0$ such that for any $1\leq p,q <2, 0<t<\infty$ and any $\alpha=(\alphah,\alpha_3) \in (\mathbb{N}\cup \{ 0 \} )^2 \times (\mathbb{N}\cup \{ 0 \} ) $ with $|\alpha| \leq 1 $ 
\begin{align*}
 \| \Grad^{\alpha} \mathcal{D}_1^{\rm v}[u](t)    \|_{ \Lv^q\Lh^p (\R^3_+)  } \leq    &     C \| u_0 \|_{X^s(\R^3_+)   } ^2  (1+t)^ {- (1-\f{1}{p})-\f{1}{2}(1-\f{1}{q})-\f{1+|\alphah|}{2}  }  \log(2+t) ,    \\  
   \| \Grad^{\alpha} \mathcal{D}_2^{\rm v}[u](t)    \|_{ \Lv^q\Lh^p (\R^3_+)  } \leq    &     C \| u_0 \|_{X^s(\R^3_+)   } ^2  (1+t)^{ - (1-\f{1}{p})-\f{1}{2}(1-\f{1}{q})-\f{1+|\alphah|}{2}   }     ,     \\
   \| \Grad^{\alpha}\mathcal{D}_3^{\rm v}[u](t)    \|_{ \Lv^q\Lh^p (\R^3_+)  } \leq    &    C \| u_0 \|_{X^s(\R^3_+)   } ^2  (1+t) ^{  - (1-\f{1}{p}) -\f{1+|\alphah|}{2}  } ,    
\end{align*}
and  
\begin{align*}
 \| \Grad^{\alpha} \mathcal{D}_1^{\rm h}[u](t)    \|_{ \Lv^q\Lh^p (\R^3_+)  }  \leq    &     C \| u_0 \|_{X^s(\R^3_+)   } ^2  (1+t) ^{  - (1-\f{1}{p}) -\f{|\alphah|}{2}    }       ,   \\
   \| \Grad^{\alpha} \mathcal{D}_2^{\rm h}[u](t)    \|_{ \Lv^q\Lh^p (\R^3_+)  }  \leq    &     C \| u_0 \|_{X^s(\R^3_+)   } ^2  (1+t)^{ - (1-\f{1}{p}) -\f{1+|\alphah|}{2}   }  \log(2+t)  ,    \\
   \| \Grad^{\alpha} \mathcal{D}_3^{\rm h}[u](t)    \|_{ \Lv^q\Lh^p (\R^3_+)  } \leq    &    C \| u_0 \|_{X^s(\R^3_+)   } ^2  (1+t) ^{- (1-\f{1}{p})-\f{1}{2}(1-\f{1}{q})-\f{1+|\alphah|}{2}  }  \log(2+t)   ,   \\
    \| \Grad^{\alpha} \mathcal{D}_4^{\rm h}[u](t)    \|_{ \Lv^q\Lh^p (\R^3_+)  } \leq    &     C   \| u_0 \|_{X^s(\R^3_+)   } ^2  (1+t)^  {- (1-\f{1}{p})-\f{1}{2}(1-\f{1}{q})-\f{1+|\alphah|}{2}  }      ,  \\
     \| \Grad^{\alpha} \mathcal{D}_5^{\rm h}[u](t)    \|_{ \Lv^q\Lh^p (\R^3_+)  } \leq    &   C \| u_0 \|_{X^s(\R^3_+)   } ^2  (1+t)^{ - (1-\f{1}{p}) -\f{1+|\alphah|}{2}   }    .    
\end{align*}
\end{lemm}
\begin{proof} 
The proof of Lemma \ref{decay-duhamel-leq-2} follows the similar line as that of Lemma \ref{decay-duhamel-2}. However, due to the range of $p,q$ under consideration, decay estimates for solutions and their derivatives have to be modified in a suitable manner.

We begin with $\mathcal{D}_1^{\rm h}[u](t)$. Using \eqref{decay-h-1} and \eqref{decay-3-1}, one sees for all $1\leq p,q <2 $
\begin{align}
 \|u_3 \uh (\tau)  \|_{\Lv^q \Lh^1 (\R^3_+)  }   & \leq \|u_3(\tau)  \|_{\Lv^{ \f{2q}{2-q}  }\Lh^2 (\R^3_+) }
 \| \uh(\tau) \|_{\Lv^2 \Lh^2 (\R^3_+) }  
 \nonumber      \\    
 &  \leq 
 C \| u_0 \|_{X^s(\R^3_+)   } ^2  
 (1+\tau)^{-(\f{7}{4}-\f{1}{2q} ) },  \nonumber      \\    
  \|\p_3(u_3 \uh) (\tau)  \|_{\Lv^q \Lh^1 (\R^3_+)  }   & \leq   \|\p_3 u_3(\tau)  \|_{\Lv^{ \f{2q}{2-q}  }  \Lh^2 (\R^3_+) }
  \| \uh(\tau) \|_{\Lv^2 \Lh^2 (\R^3_+) }       \nonumber    \\ 
  & \quad \quad + 
  \| u_3(\tau)  \|_{\Lv^{ \f{2q}{2-q}  }\Lh^2 (\R^3_+) } 
  \| \p_3 \uh(\tau) \|_{\Lv^2 \Lh^2 (\R^3_+) }   \nonumber   \\
    &  \leq C \| u_0 \|_{X^s(\R^3_+)   } ^2 
    (1+\tau)^{-(\f{7}{4}-\f{1}{2q} ) },   \nonumber  \\   
 \| (u_3 \uh) (\tau)  \|_{\Lv^q \Lh^p (\R^3_+)  }   & \leq 
      \| u_3(\tau)  \|_{  \Lv^{ \f{2q}{2-q}   } \Lh^{ \f{2p}{2-p}  }(\R^3_+)         }       
      \| \uh(\tau) \|_{\Lv^2 \Lh^2 (\R^3_+)  }   \nonumber     \\
    &   \leq    C \| u_0 \|_{X^s(\R^3_+)   } ^2     (1+\tau)^{ -(\f{11}{4}-\f{1}{2q} -\f{1}{p}   ) } ,    \nonumber     \\ 
   \|\nablah (u_3 \uh) (\tau)  \|_{\Lv^q \Lh^p (\R^3_+)  }   & \leq 
      \| \nablah u_3(\tau)  \|_{  \Lv^{ \f{2q}{2-q}   } \Lh^{ \f{2p}{2-p}  }(\R^3_+)         }  
      \| \uh(\tau) \|_{\Lv^2 \Lh^2 (\R^3_+)  }     \nonumber\\
      & \quad \quad + \|  u_3(\tau)  \|_{  \Lv^{ \f{2q}{2-q}   } \Lh^{ \f{2p}{2-p}  }(\R^3_+)         }   \| \nablah \uh(\tau) \|_{\Lv^2 \Lh^2 (\R^3_+)  }  \nonumber\\
      &   \leq 
      C \| u_0 \|_{X^s(\R^3_+)   } ^2
      (1+\tau)^{ -(\f{13}{4}-\f{1}{2q} -\f{1}{p}   ) }  ,     \nonumber  \\   
    \|\p_3 (u_3 \uh) (\tau)  \|_{\Lv^q \Lh^p (\R^3_+)  }   & 
      \leq 
      \| \p_3 u_3(\tau)  \|_{  \Lv^{ \f{2q}{2-q}   } \Lh^{ \f{2p}{2-p}  }(\R^3_+)         } 
      \| \uh(\tau) \|_{\Lv^2 \Lh^2 (\R^3_+)  }    
      \label{lem5.1-1}   \\
      & \quad \quad +
      \|  u_3(\tau)  \|_{  \Lv^{ \f{2q}{2-q}   } \Lh^{ \f{2p}{2-p}  }(\R^3_+)         } 
      \| \p_3 \uh(\tau) \|_{\Lv^2 \Lh^2 (\R^3_+)  }  \nonumber\\
      &   \leq C \| u_0 \|_{X^s(\R^3_+)   } ^2  
      (1+\tau)^{ -(\f{11}{4}-\f{1}{2q} -\f{1}{p}   ) }  .   \nonumber 
\end{align}
Let us further divide it into $\alpha_3=0$ and $\alpha_3=1$ respectively. If $\alpha_3=0$, we use \eqref{lem5.1-1} to control 
\begin{align*}
    &  \left\| \nablah^{\alphah} \int_0^t   e^{(t-\tau) \Deltah }  \p_3 (u_3 \uh)(\tau)     d \tau \right\|_{\Lv^q \Lh^p (\R^3_+)  }            \\
    &   \quad  \leq \int_0^{\f{t}{2}  } \| \nablah^{\alphah} \Gh(t-\tau) \|_{ L^p(\R^2) } \|\p_3(u_3 \uh)(\tau)   \|_{\Lv^q \Lh^1 (\R^3_+)  }   d\tau   \\
    &   \quad \quad +  \int_{\f{t}{2}}^{t} \| \nablah^{\alphah} \Gh(t-\tau) \|_{ L^1(\R^2) }   \|\p_3 (u_3 \uh) (\tau)  \|_{\Lv^q \Lh^p (\R^3_+)  }  d\tau      \\
    & \quad \leq  C \| u_0 \|_{X^s(\R^3_+)   } ^2 \Big( \int_0^{\f{t}{2}  } (t-\tau)^{ -(1-\f{1}{p}) -\f{|\alphah|}{2}  } (1+\tau)^{-(\f{7}{4}-\f{1}{2q} ) }   d\tau   \\
    &   \quad \quad  \quad \quad \quad \quad \quad \quad\quad \quad \quad \quad +   \int_{\f{t}{2}}^{t} (t-\tau)^{-\f{ |\alphah| }{2} } (1+\tau)^{ -(\f{11}{4}-\f{1}{2q} -\f{1}{p}   ) }  d\tau   \Big)    \\
    & \quad \leq C \| u_0 \|_{X^s(\R^3_+)   } ^2 t^{ -(1-\f{1}{p}) -\f{|\alphah|}{2}  }   
\end{align*}
for $t \geq 1$ and 
\begin{align*}
    &    \left\| \nablah^{\alphah} \int_0^t   e^{(t-\tau) \Deltah }  \p_3 (u_3 \uh)(\tau)     d \tau \right\|_{\Lv^q \Lh^p (\R^3_+)  }            \\
    &   \quad  \leq \int_0^t  \| \nablah^{\alphah} \Gh(t-\tau) \|_{ L^1(\R^2) }   \|\p_3 (u_3 \uh) (\tau)  \|_{\Lv^q \Lh^p (\R^3_+)  }  d\tau       \\
    &   \quad \leq  C \| u_0 \|_{X^s(\R^3_+)   } ^2 \int_0^t  (t-\tau)^{-\f{ |\alphah| }{2} } (1+\tau)^{ -(\f{11}{4}-\f{1}{2q} -\f{1}{p}   ) }    d\tau       \\
    & \quad \leq  C  \| u_0 \|_{X^s(\R^3_+)   } ^2
\end{align*}
for $0<t\leq 1$. Therefore,
\begin{equation}\label{lem5.1-2}
   \left\| \nablah^{\alphah} \int_0^t   e^{(t-\tau) \Deltah }  \p_3 (u_3 \uh)(\tau)     d \tau \right\|_{\Lv^q \Lh^p (\R^3_+)  }    
    \leq  C  \| u_0 \|_{X^s(\R^3_+)   } ^2
    (1+t)^{ -(1-\f{1}{p}) -\f{|\alphah|}{2}  }  . 
\end{equation}

To consider the case of $\alpha_3=1$, we infer from the Gagliardo-Nirenberg interpolation inequality, \eqref{decay-h-1} and \eqref{decay-3-1} that for any $1\leq p,q <2$ 
\begin{align*}
  \|\p_3^2 (u_3 \uh) (\tau)  \|_{\Lv^q \Lh^1 (\R^3_+)  }  & 
  \leq  \| \p_3^2 u_3(\tau)  \|_{\Lv^{ \f{2q}{2-q}  } \Lh^2 (\R^3_+)  } 
  \| \uh(\tau) \|_{\Lv^2 \Lh^2 (\R^3_+)  }      \\
    &   \quad  
    +
    2 \| \p_3 u_3(\tau)  \|_{\Lv^{\f{2q}{2-q}} \Lh^2 (\R^3_+)  }  
    \| \p_3 \uh(\tau) \|_{\Lv^{2} \Lh^2 (\R^3_+)  }      \\
    &   \quad    
    +
    \| u_3(\tau) \|_{\Lv^{\f{2q}{2-q}} \Lh^2 (\R^3_+)  } 
    \| \p_3^2 \uh(\tau)  \|_{\Lv^{2} \Lh^2 (\R^3_+)  }   \\
    & \leq C 
    \Big(
    \| \p_3 u_3(\tau) \|_{L^2(\R^3_+)}^{  1-\f{2q-1}{4q}  } 
    \| \p_3^5 u_3(\tau) \|_{L^2(\R^3_+)}^{  \f{2q-1}{4q}  } 
    \| \uh(\tau) \|_{\Lv^2 \Lh^2 (\R^3_+)  }   \\
    &  \quad +
    \| \p_3 u_3(\tau)  \|_{\Lv^{\f{2q}{2-q}} \Lh^2 (\R^3_+)  } 
    \| \p_3 \uh(\tau) \|_{\Lv^{2} \Lh^2 (\R^3_+)  }  \\
    &   \quad + 
    \| u_3(\tau) \|_{\Lv^{\f{2q}{2-q}} \Lh^2 (\R^3_+)  } 
    \| \p_3 \uh(\tau) \|_{L^2(\R^3_+)}^{\f{3}{4}  }
    \| \p_3^5 \uh(\tau) \|_{L^2(\R^3_+)}^{\f{1}{4}  }
    \Big)     \\
    & \leq C  \| u_0 \|_{X^s(\R^3_+)   } ^2
    (1+\tau)^{- \min\left\{  \f{4q+1}{4q},\f{13q-4}{8q}     \right\}    };  
\end{align*} 
\begin{align*}
  \|\p_3^2 (u_3 \uh) (\tau)  \|_{\Lv^q \Lh^2 (\R^3_+)  }  &  
  \leq 
  \| \p_3^2 u_3(\tau)  \|_{\Lv^{\f{2q}{2-q}} \Lh^2 (\R^3_+)  }
  \| \uh(\tau) \|_{\Lv^{2} \Lh^{\infty} (\R^3_+)  }     \\
    &   \quad 
    +2 \| \p_3 u_3(\tau)  \|_{ \Lv^{\f{2q}{2-q}} \Lh^{\infty} (\R^3_+)  } 
    \| \p_3 \uh(\tau) \|_{\Lv^{2} \Lh^2 (\R^3_+)  }      \\
    &   \quad     
    + \| u_3(\tau) \|_{\Lv^{2} \Lh^{\infty} (\R^3_+)  }  
    \| \p_3^2 \uh(\tau)  \|_{\Lv^{\f{2q}{2-q}} \Lh^2 (\R^3_+)  } \\ 
     & \leq C
     \Big( 
     \| \p_3 u_3(\tau) \|_{L^2(\R^3_+)}^{  1-\f{2q-1}{4q}  } 
    \| \p_3^5 u_3(\tau) \|_{L^2(\R^3_+)}^{  \f{2q-1}{4q}  } 
     \| \uh(\tau) \|_{\Lv^{2} \Lh^{\infty} (\R^3_+)  } \\
    &  \quad + 
  \| \p_3 u_3(\tau)  \|_{ \Lv^{\f{2q}{2-q}} \Lh^{\infty} (\R^3_+)  } 
    \| \p_3 \uh(\tau) \|_{\Lv^{2} \Lh^2 (\R^3_+)  }  \\
    &   \quad + 
    \| u_3(\tau) \|_{\Lv^{2} \Lh^{\infty} (\R^3_+)  } 
     \| \p_3 \uh(\tau) \|_{L^2(\R^3_+)}^{  1-\f{2q-1}{4q}  } 
    \| \p_3^5 \uh(\tau) \|_{L^2(\R^3_+)}^{  \f{2q-1}{4q}  } 
    \Big) \\
    & \leq
    C  \| u_0 \|_{X^s(\R^3_+)   } ^2
    (1+\tau)^{   -\f{12q+1}{8q}  }. 
\end{align*} 
It is easy to check that for any $1\leq p,q <2$
\begin{align*}
&  \min\left\{  \f{4q+1}{4q},\f{13q-4}{8q}     \right\} >1 , 
\quad \quad 
 -\f{12q+1}{8q} +\f{1}{2}>1.
\end{align*} 
Then we can repeat step by step as in Lemma \ref{decay-duhamel-2}.

To proceed, we consider $\mathcal{D}_2^{\rm h}[u](t)$. Using \eqref{decay-h-1} and \eqref{decay-3-1} again, one sees for all $1\leq p,q <2 $ and $k,l\in \{1,2 \}$
\begin{align}
 \|u_k u_l (\tau)  \|_{\Lv^q \Lh^1 (\R^3_+)  }   & 
 \leq \|u_k(\tau)  \|_{\Lv^{  \f{2q}{2-q}  }\Lh^2 (\R^3_+) }
 \| u_l(\tau) \|_{\Lv^2 \Lh^2 (\R^3_+) }  \nonumber \\
 &
 \leq 
 C   \| u_0 \|_{X^s(\R^3_+)   } ^2 (1+\tau)^{-1},     \nonumber      \\
  \|\p_3(u_k u_l) (\tau)  \|_{\Lv^q \Lh^1 (\R^3_+)  }   & \leq C 
  \|\p_3 u_k(\tau)  \|_{\Lv^{\f{2q}{2-q}}\Lh^2 (\R^3_+) } 
  \| u_l(\tau) \|_{\Lv^q \Lh^2 (\R^3_+) }         \nonumber    \\
    &  \leq
    C  
    \| u_0 \|_{X^s(\R^3_+)   } ^2
    (1+\tau)^{-1},   \nonumber  \\
 \|\Grad^{\alpha} (u_k u_l) (\tau)  \|_{\Lv^q \Lh^p (\R^3_+)  }   & \leq 
    C  \| \Grad^{\alpha}u_k(\tau)  \|_{ \Lv^{\f{2q}{2-q}}\Lh^{\f{2p}{2-p}} (\R^3_+)}
    \| u_l(\tau) \|_{\Lv^2 \Lh^2 (\R^3_+)  }   \label{lem5.1-3}  \\
    & \leq
    C \| u_0 \|_{X^s(\R^3_+)   } ^2
    (1+\tau)^{ -(1-\f{2-p}{2p})-\f{|\alphah|}{2}   } (1+\tau)^{-\f{1}{2}}   \nonumber  \\
    & = 
    C   \| u_0 \|_{X^s(\R^3_+)   } ^2 
    (1+\tau)^{ -(2- \f{1}{p} )-\f{|\alphah|}{2}   }.   \nonumber 
\end{align}
It follows readily from \eqref{lem5.1-3} that 
\begin{align*}
    &  \left\| \Grad^{\alpha} \int_0^t   e^{(t-\tau) \Deltah }  \div_{\rm h} (\uh \otimes \uh)(\tau)     d \tau \right\|_{\Lv^q \Lh^p (\R^3_+)  }            \\
    &   \quad  \leq \int_0^{\f{t}{2}  } \| \nablah^{\alphah}\nablah \Gh(t-\tau) \|_{ L^p(\R^2) } \|\p_3^{\alpha_3}(u_k u_l)(\tau)   \|_{\Lv^q \Lh^1 (\R^3_+)  }   d\tau   \\
    &   \quad \quad +  \int_{\f{t}{2}}^{t} \| \nablah \Gh(t-\tau) \|_{ L^1(\R^2) }   \|\Grad^{\alpha} (u_k u_l) (\tau)  \|_{\Lv^q \Lh^p (\R^3_+)  }  d\tau      \\
    & \quad \leq  
    C \| u_0 \|_{X^s(\R^3_+)   } ^2 
    \Big( \int_0^{\f{t}{2}  } (t-\tau)^{ -(1-\f{1}{p}) -\f{1+|\alphah|}{2}  } (1+\tau)^{-1}d\tau \\
    & \quad \quad \quad \quad \quad \quad \quad \quad \quad \quad \quad \quad 
    +   \int_{\f{t}{2}}^{t} (t-\tau)^{-\f{1}{2}} (1+\tau)^{ -(2- \f{1}{p} )-\f{|\alphah|}{2}   }  d\tau   \Big) \\
    & \quad \leq 
    C \| u_0 \|_{X^s(\R^3_+)   } ^2 
    t^{ -(1-\f{1}{p}) -\f{1+|\alphah|}{2}  } \log(2+t)
\end{align*}
for $t \geq 1$ and 
\begin{align*}
    &    \left\| \Grad^{\alpha} \int_0^t   e^{(t-\tau) \Deltah }  \div_{\rm h} (\uh \otimes \uh)(\tau)     d \tau \right\|_{\Lv^q \Lh^p (\R^3_+)  }           \\
    &   \quad  \leq \int_0^t \| \nablah \Gh(t-\tau) \|_{ L^1(\R^2) }   \|\Grad^{\alpha} (u_k u_l) (\tau)  \|_{\Lv^q \Lh^p (\R^3_+)  }  d\tau       \\
    &   \quad \leq  C\| u_0 \|_{X^s(\R^3_+)   } ^2 \int_0^t   (t-\tau)^{-\f{1}{2}} (1+\tau)^{ -(2- \f{1}{p} )-\f{|\alphah|}{2}   }  d\tau       \\
    & \quad \leq   C  \| u_0 \|_{X^s(\R^3_+)   } ^2
\end{align*}
for $0<t\leq 1$. Thus
\begin{align*}
  &  \left\| \Grad^{\alpha} \int_0^t   e^{(t-\tau) \Deltah } \div_{\rm h} (\uh \otimes \uh)(\tau)     d \tau \right\|_{\Lv^q \Lh^p (\R^3_+)  }    \\
  & \quad \quad   \quad \quad   
    \leq 
    C
    \| u_0 \|_{X^s(\R^3_+)   } ^2
    (1+t)^{ -(1-\f{1}{p}) -\f{1+|\alphah|}{2}  } \log(2+t). 
\end{align*}

The estimate of $\mathcal{D}_3^{\rm h}[u](t)$ follows exactly as in Lemma \ref{decay-duhamel-2} with slight modifications on the estimate towards $ \| 
    \nablah  (\uh \uh)(\tau) 
    \|_{\Lv^{1} \Lh^{p} (\R^3_+) }$
. The detailed proof for the remaining integrals is again omitted for brevity.   \end{proof}

\section{Proof of Theorem \ref{main-theorem} }\label{main-thm-2}

\subsection{Proof of Theorem \ref{main-theorem} in case of \texorpdfstring{$1 \leq  p,q < 2$}{}} \label{sec-6.1}

It follows from the linear estimates Corollaries \ref{cor2.6}-\ref{cor2.7} that there exists generic $C_1>0$ such that for any $\alpha=(\alphah,\alpha_3) \in (\mathbb{N}\cup \{ 0 \} )^2 \times (\mathbb{N}\cup \{ 0 \} ) $ with $|\alpha| \leq 1 $
\begin{align}
    &    \| \nabla^{\alpha} (e^{-tA}u_{0})_{\rm h}   \| _{\Lv^q\Lh^p (\R^3_+) }  \leq      C_1  \left(  t^{ -(1-\f{1}{p})  -\f{|\alphah|}{2}         } +   t^{ - (1-\f{1}{p})-\f{1}{2}(1-\f{1}{q})   -\f{|\alphah|}{2}         }   \right) \| u_0 \|_{X^s(\R^3_+)   }     \nonumber
\end{align}
for all $t>0, 1 \leq  p,q <2$ satisfying $(1-\f{1}{p}) + \f{|\alphah|}{2}    >0$;
\begin{align}
    &  \| \nablah^{\alphah}  (e^{-tA}u_{0})_{3} \|_{\Lv^q\Lh^p (\R^3_+)}  \leq       C_1  t^{ - (1-\f{1}{p})-\f{1}{2}(1-\f{1}{q})   -\f{|\alphah|}{2}         }  \| u_0 \|_{X^s(\R^3_+)   } \nonumber
\end{align}
for all $t>0, 1 \leq  p,q <2$ satisfying $ (1-\f{1}{p})+\f{1}{2}(1-\f{1}{q})   +\f{|\alphah|}{2}  >0$. Based on Corollary \ref{lemm:sol-form-u}, the decay estimates of Duhamel terms obtained in Lemma \ref{decay-duhamel-leq-2} and the linear estimates above, we infer for any $\alpha=(\alphah,\alpha_3) \in (\mathbb{N}\cup \{ 0 \} )^2 \times (\mathbb{N}\cup \{ 0 \} ) $ with $|\alpha| \leq 1 $
\begin{align*}
   \| \nabla^{\alpha} \uh(t)   \| _{\Lv^q\Lh^p (\R^3_+) }  &     \leq      C_1  \left(  t^{ -(1-\f{1}{p})  -\f{|\alphah|}{2}         } +   t^{ - (1-\f{1}{p})-\f{1}{2}(1-\f{1}{q})   -\f{|\alphah|}{2}         }   \right) \| u_0 \|_{X^s(\R^3_+)   }          \\
    &  \quad + C_2 (1+t)^{-(1-\f{1}{p})  -\f{|\alphah|}{2}   }     \| u_0 \|_{X^s(\R^3_+)   } ^2      \\
    &  \leq  C_3         \left(  t^{ -(1-\f{1}{p})  -\f{|\alphah|}{2}         } +   t^{ - (1-\f{1}{p})-\f{1}{2}(1-\f{1}{q})   -\f{|\alphah|}{2}         }   \right) \| u_0 \|_{X^s(\R^3_+)   }   ,                  \\
  \| \nablah^{\alphah}  u_3(t) \|_{\Lv^q\Lh^p (\R^3_+)}  &    \leq       C_1  t^{ - (1-\f{1}{p})-\f{1}{2}(1-\f{1}{q})   -\f{|\alphah|}{2}         }  \| u_0 \|_{X^s(\R^3_+)   }   \\
  & \quad + C_2 (1+t)^{  - (1-\f{1}{p})-\f{1}{2}(1-\f{1}{q})   -\f{|\alphah|}{2}      }   \| u_0 \|_{X^s(\R^3_+)   } ^2  \\
  &   \leq  C_3 t^{ - (1-\f{1}{p})-\f{1}{2}(1-\f{1}{q})   -\f{|\alphah|}{2}         }  \| u_0 \|_{X^s(\R^3_+)   } 
\end{align*}
for all $t>0$, $1\leq  p,q <2$ satisfying $(1-\f{1}{p}) + \f{|\alphah|}{2}    >0$ and $ (1-\f{1}{p})+\f{1}{2}(1-\f{1}{q})   +\f{|\alphah|}{2}  >0$ respectively. Here, we invoked the smallness of initial data. We have verified \eqref{decay-h-2} and \eqref{decay-3-2}, thus finishing the proof of Theorem \ref{main-theorem} in case of $1\leq p,q <2$.

\subsection{Proof of Theorem \ref{main-theorem}: the remaining cases} 

Observe that the similar argument of Section \ref{decay-less-2} still works well for Case $II$ and Case $III$ in \eqref{pq-range-h} and \eqref{pq-range-3} respectively. More precisely, it holds that
\begin{lemm}\label{decay-duhamel-6.2-lem} 
There exists a pure constant $C>0$ such that for any $p,q $ satisfying
\begin{align*}
 & 1\leq p < 2, 2\leq q \leq \infty \quad
 \text{or} \quad  1\leq q < 2, 2\leq p \leq \infty
\end{align*}
and any $\alpha=(\alphah,\alpha_3) \in (\mathbb{N}\cup \{ 0 \} )^2 \times (\mathbb{N}\cup \{ 0 \} ) $ with $|\alpha| \leq 1 $ 
\begin{align*}
 \| \Grad^{\alpha} \mathcal{D}_1^{\rm v}[u](t)    \|_{ \Lv^q\Lh^p (\R^3_+)  } \leq    &     C \| u_0 \|_{X^s(\R^3_+)   } ^2  (1+t)^ {- (1-\f{1}{p})-\f{1}{2}(1-\f{1}{q})-\f{1+|\alphah|}{2}  }  \log(2+t) ,    \\  
   \| \Grad^{\alpha} \mathcal{D}_2^{\rm v}[u](t)    \|_{ \Lv^q\Lh^p (\R^3_+)  } \leq    &     C \| u_0 \|_{X^s(\R^3_+)   } ^2  (1+t)^{ - (1-\f{1}{p})-\f{1}{2}(1-\f{1}{q})-\f{1+|\alphah|}{2}   }     ,     \\
   \| \Grad^{\alpha}\mathcal{D}_3^{\rm v}[u](t)    \|_{ \Lv^q\Lh^p (\R^3_+)  } \leq    &    C \| u_0 \|_{X^s(\R^3_+)   } ^2  (1+t) ^{  - (1-\f{1}{p}) -\f{1+|\alphah|}{2}  } ,    
\end{align*}
and  
\begin{align*}
 \| \Grad^{\alpha} \mathcal{D}_1^{\rm h}[u](t)    \|_{ \Lv^q\Lh^p (\R^3_+)  }  \leq    &     C \| u_0 \|_{X^s(\R^3_+)   } ^2  (1+t) ^{  - (1-\f{1}{p}) -\f{|\alphah|}{2}    }       ,   \\
   \| \Grad^{\alpha} \mathcal{D}_2^{\rm h}[u](t)    \|_{ \Lv^q\Lh^p (\R^3_+)  }  \leq    &     C \| u_0 \|_{X^s(\R^3_+)   } ^2  (1+t)^{ - (1-\f{1}{p}) -\f{1+|\alphah|}{2}   }  \log(2+t)  ,    \\
   \| \Grad^{\alpha} \mathcal{D}_3^{\rm h}[u](t)    \|_{ \Lv^q\Lh^p (\R^3_+)  } \leq    &    C \| u_0 \|_{X^s(\R^3_+)   } ^2  (1+t) ^{- (1-\f{1}{p})-\f{1}{2}(1-\f{1}{q})-\f{1+|\alphah|}{2}  }  \log(2+t)   ,   \\
    \| \Grad^{\alpha} \mathcal{D}_4^{\rm h}[u](t)    \|_{ \Lv^q\Lh^p (\R^3_+)  } \leq    &     C   \| u_0 \|_{X^s(\R^3_+)   } ^2  (1+t)^  {- (1-\f{1}{p})-\f{1}{2}(1-\f{1}{q})-\f{1+|\alphah|}{2}  }      ,  \\
     \| \Grad^{\alpha} \mathcal{D}_5^{\rm h}[u](t)    \|_{ \Lv^q\Lh^p (\R^3_+)  } \leq    &   C \| u_0 \|_{X^s(\R^3_+)   } ^2  (1+t)^{ - (1-\f{1}{p}) -\f{1+|\alphah|}{2}   }    .    
\end{align*}
\end{lemm}
\begin{proof} 
Due to the symmetry of $p,q$ under consideration, it suffices to handle the case $1\leq p < 2, 2\leq q \leq \infty$. 
The proof follows the similar line as that of Lemma \ref{decay-duhamel-2} with suitable modifications. 

We begin with $\mathcal{D}_1^{\rm h}[u](t)$. Based on \eqref{decay-h-1} and \eqref{decay-3-1}, one sees for all $1\leq p < 2, 2\leq q \leq \infty $
\begin{align}
 \|u_3 \uh (\tau)  \|_{\Lv^q \Lh^1 (\R^3_+)  }   & \leq \|u_3(\tau)  \|_{\Lv^{ \infty  }\Lh^2 (\R^3_+) }
 \| \uh(\tau) \|_{\Lv^q \Lh^2 (\R^3_+) }  
 \nonumber      \\    
 & \leq 
 C \| u_0 \|_{X^s(\R^3_+)   } ^2  
 (1+\tau)^{    -\f{3}{2}   },  \nonumber      \\    
  \|\p_3(u_3 \uh) (\tau)  \|_{\Lv^q \Lh^1 (\R^3_+)  }   & \leq   \|\p_3 u_3(\tau)  \|_{\Lv^{ \infty  }  \Lh^2 (\R^3_+) }
  \| \uh(\tau) \|_{\Lv^q \Lh^2 (\R^3_+) }       \nonumber    \\ 
  & \quad \quad + 
  \| u_3(\tau)  \|_{\Lv^{ \infty  }\Lh^2 (\R^3_+) } 
  \| \p_3 \uh(\tau) \|_{\Lv^q \Lh^2 (\R^3_+) }   \nonumber   \\
    &  \leq C \| u_0 \|_{X^s(\R^3_+)   } ^2 
    (1+\tau)^{   -  \f{3}{2}    },   \nonumber  \\   
 \| (u_3 \uh) (\tau)  \|_{\Lv^q \Lh^p (\R^3_+)  }   & \leq 
      \| u_3(\tau)  \|_{  \Lv^{ \infty  } \Lh^{ \f{2p}{2-p}  }(\R^3_+)         }       
      \| \uh(\tau) \|_{\Lv^q \Lh^2 (\R^3_+)  }   \nonumber     \\ 
    &   \leq  
    C \| u_0 \|_{X^s(\R^3_+)   } ^2   
    (1+\tau)^
    { 
    -\f{5}{2} +\f{1}{p}   
    } , 
    \nonumber     \\ 
   \|\nablah (u_3 \uh) (\tau)  \|_{\Lv^q \Lh^p (\R^3_+)  }   & \leq 
      \| \nablah u_3(\tau)  \|_{  \Lv^{ \infty   } \Lh^{ \f{2p}{2-p}  }(\R^3_+)         }  
      \| \uh(\tau) \|_{\Lv^q \Lh^2 (\R^3_+)  }     \nonumber\\
      & \quad \quad + \|  u_3(\tau)  \|_{  \Lv^{ \infty  } \Lh^{ \f{2p}{2-p}  }(\R^3_+)         }   \| \nablah \uh(\tau) \|_{\Lv^q \Lh^2 (\R^3_+)  }  \nonumber\\
      &   \leq 
      C \| u_0 \|_{X^s(\R^3_+)   } ^2
      (1+\tau)
      ^{
      -3+ \f{1}{p}   
      }  ,     \nonumber  \\   
    \|\p_3 (u_3 \uh) (\tau)  \|_{\Lv^q \Lh^p (\R^3_+)  }   & 
      \leq 
      \| \p_3 u_3(\tau)  \|_{  \Lv^{ \infty   } \Lh^{ \f{2p}{2-p}  }(\R^3_+)         } 
      \| \uh(\tau) \|_{\Lv^q \Lh^2 (\R^3_+)  }    
      \label{lem-6.1-1}   \\
      & \quad \quad +
      \|  u_3(\tau)  \|_{  \Lv^{ \infty   } \Lh^{ \f{2p}{2-p}  }(\R^3_+)         } 
      \| \p_3 \uh(\tau) \|_{\Lv^q \Lh^2 (\R^3_+)  }  \nonumber\\
      &   \leq C \| u_0 \|_{X^s(\R^3_+)   } ^2  
      (1+\tau)
      ^{
      -\f{5}{2} +\f{1}{p}   
      }  .   \nonumber 
\end{align}
Let us further divide it into $\alpha_3=0$ and $\alpha_3=1$ respectively. If $\alpha_3=0$, we use \eqref{lem-6.1-1} to control 
\begin{align*}
    &  \left\| \nablah^{\alphah} \int_0^t   e^{(t-\tau) \Deltah }  \p_3 (u_3 \uh)(\tau)     d \tau \right\|_{\Lv^q \Lh^p (\R^3_+)  }            \\
    &   \quad  \leq \int_0^{\f{t}{2}  } \| \nablah^{\alphah} \Gh(t-\tau) \|_{ L^p(\R^2) } \|\p_3(u_3 \uh)(\tau)   \|_{\Lv^q \Lh^1 (\R^3_+)  }   d\tau   \\
    &   \quad \quad +  \int_{\f{t}{2}}^{t} \| \nablah^{\alphah} \Gh(t-\tau) \|_{ L^1(\R^2) }   \|\p_3 (u_3 \uh) (\tau)  \|_{\Lv^q \Lh^p (\R^3_+)  }  d\tau      \\
    & \quad \leq  C \| u_0 \|_{X^s(\R^3_+)   } ^2 \Big( \int_0^{\f{t}{2}  } (t-\tau)^{ -(1-\f{1}{p}) -\f{|\alphah|}{2}  }
    (1+\tau)^{-\f{3}{2} }   d\tau   \\
    &   \quad \quad  \quad \quad \quad \quad \quad \quad\quad \quad \quad \quad +   \int_{\f{t}{2}}^{t} (t-\tau)^{-\f{ |\alphah| }{2} } (1+\tau)^{ -\f{5}{2} +\f{1}{p}    }  d\tau   \Big)    \\
    & \quad \leq C \| u_0 \|_{X^s(\R^3_+)   } ^2 t^{ -(1-\f{1}{p}) -\f{|\alphah|}{2}  }   
\end{align*}
for $t \geq 1$ and 
\begin{align*}
    &    \left\| \nablah^{\alphah} \int_0^t   e^{(t-\tau) \Deltah }  \p_3 (u_3 \uh)(\tau)     d \tau \right\|_{\Lv^q \Lh^p (\R^3_+)  }            \\
    &   \quad  \leq \int_0^t  \| \nablah^{\alphah} \Gh(t-\tau) \|_{ L^1(\R^2) }   \|\p_3 (u_3 \uh) (\tau)  \|_{\Lv^q \Lh^p (\R^3_+)  }  d\tau       \\
    &   \quad \leq  C \| u_0 \|_{X^s(\R^3_+)   } ^2 \int_0^t  (t-\tau)^{-\f{ |\alphah| }{2} } (1+\tau)^{  -\f{5}{2} +\f{1}{p}   }    d\tau       \\
    & \quad \leq  C  \| u_0 \|_{X^s(\R^3_+)   } ^2
\end{align*}
for $0<t\leq 1$. Thus for any $1\leq p < 2, 2\leq q \leq \infty$ it holds 
\begin{align*}
 &  \left\| \nablah^{\alphah} \int_0^t   e^{(t-\tau) \Deltah }  \p_3 (u_3 \uh)(\tau)     d \tau \right\|_{\Lv^q \Lh^p (\R^3_+)  }    
    \leq  C  \| u_0 \|_{X^s(\R^3_+)   } ^2
    (1+t)^{ -(1-\f{1}{p}) -\f{|\alphah|}{2}  }  . 
\end{align*}

To treat the case of $\alpha_3=1$, we again deduce from the Gagliardo-Nirenberg interpolation inequality, \eqref{decay-h-1} and \eqref{decay-3-1} that for any $1\leq p < 2, 2\leq q \leq \infty$ 
\begin{align*}
  \|\p_3^2 (u_3 \uh) (\tau)  \|_{\Lv^q \Lh^1 (\R^3_+)  }  & 
  \leq  \| \p_3^2 u_3(\tau)  \|_{\Lv^{ \infty  } \Lh^2 (\R^3_+)  } 
  \| \uh(\tau) \|_{\Lv^q  \Lh^2 (\R^3_+)  }      \\
    &   \quad  
    +
    2 \| \p_3 u_3(\tau)  \|_{\Lv^{\infty } \Lh^2 (\R^3_+)  }  
    \| \p_3 \uh(\tau) \|_{\Lv^{q} \Lh^2 (\R^3_+)  }      \\
    &   \quad    
    +
    \| u_3(\tau) \|_{\Lv^{q} \Lh^2 (\R^3_+)  } 
    \| \p_3^2 \uh(\tau)  \|_{\Lv^{ \infty } \Lh^2 (\R^3_+)  }   \\ 
    & \leq C  
    \Big(
    \| \p_3 u_3(\tau) \|_{L^2(\R^3_+)}^{  \f{5}{8}  } 
    \| \p_3^5 u_3(\tau) \|_{L^2(\R^3_+)}^{  \f{3}{8}  } 
    \| \uh(\tau) \|_{\Lv^q \Lh^2 (\R^3_+)  }   \\
    &  \quad +
    \| \p_3 u_3(\tau)  \|_{\Lv^{\infty } \Lh^2 (\R^3_+)  } 
    \| \p_3 \uh(\tau) \|_{\Lv^{q} \Lh^2 (\R^3_+)  }  \\
    &   \quad + 
    \| u_3(\tau) \|_{\Lv^{ q  } \Lh^2 (\R^3_+)  } 
    \| \p_3 \uh(\tau) \|_{L^2(\R^3_+)}^{ \f{5}{8}  }
    \| \p_3^5 \uh(\tau) \|_{L^2(\R^3_+)}^{ \f{3}{8}    }
    \Big)     \\
    & \leq C  \| u_0 \|_{X^s(\R^3_+)   } ^2
    (1+\tau)
    ^{  -\f{17}{16}
    };  
\end{align*} 
\begin{align*}
  \|\p_3^2 (u_3 \uh) (\tau)  \|_{\Lv^q \Lh^2 (\R^3_+)  }  &  
  \leq 
  \| \p_3^2 u_3(\tau)  \|_{\Lv^{ \infty } \Lh^2 (\R^3_+)  }
  \| \uh(\tau) \|_{\Lv^{q} \Lh^{\infty} (\R^3_+)  }     \\
    &   \quad 
    +2 \| \p_3 u_3(\tau)  \|_{ \Lv^{\infty } \Lh^{\infty} (\R^3_+)  } 
    \| \p_3 \uh(\tau) \|_{\Lv^{q} \Lh^2 (\R^3_+)  }      \\
    &   \quad     
    + \| u_3(\tau) \|_{\Lv^{q} \Lh^{\infty} (\R^3_+)  }  
    \| \p_3^2 \uh(\tau)  \|_{\Lv^{\infty  } \Lh^2 (\R^3_+)  }   \\   
     & \leq C
     \Big( 
     \| \p_3 u_3(\tau) \|_{L^2(\R^3_+)}^{   \f{5}{8}   } 
    \| \p_3^5 u_3(\tau) \|_{L^2(\R^3_+)}^{  \f{3}{8}  } 
     \| \uh(\tau) \|_{\Lv^{q} \Lh^{\infty} (\R^3_+)  } \\
    &  \quad + 
  \| \p_3 u_3(\tau)  \|_{ \Lv^{\infty } \Lh^{\infty} (\R^3_+)  } 
    \| \p_3 \uh(\tau) \|_{\Lv^{q} \Lh^2 (\R^3_+)  }  \\
    &   \quad + 
    \| u_3(\tau) \|_{\Lv^{q} \Lh^{\infty} (\R^3_+)  } 
     \| \p_3 \uh(\tau) \|_{L^2(\R^3_+)}^{  \f{5}{8}   } 
    \| \p_3^5 \uh(\tau) \|_{L^2(\R^3_+)}^{  \f{3}{8}   } 
    \Big) \\
    & \leq
    C  \| u_0 \|_{X^s(\R^3_+)   } ^2
    (1+\tau)^{ -\f{25}{16}      }. 
\end{align*} 
Similar to Lemma \ref{decay-duhamel-2}, the two estimates above enable us to derive 
\begin{align*}
 &  \left\| \p_3 \int_0^t   e^{(t-\tau) \Deltah }  \p_3 (u_3 \uh)(\tau)     d \tau \right\|_{\Lv^q \Lh^p (\R^3_+)  }    
    \leq  C 
    \| u_0 \|_{X^s(\R^3_+)   } ^2
    (1+t)^{ -(1-\f{1}{p}) } 
\end{align*} 
for any $1\leq p < 2, 2\leq q \leq \infty$.

We next consider $\mathcal{D}_2^{\rm h}[u](t)$. Making use of \eqref{decay-h-1} and \eqref{decay-3-1} once again, one sees for all $1\leq p < 2, 2\leq q \leq \infty$ and $k,l\in \{1,2 \}$ 
\begin{align}
 \|u_k u_l (\tau)  \|_{\Lv^q \Lh^1 (\R^3_+)  }   & 
 \leq \|u_k(\tau)  \|_{\Lv^{  \infty  }\Lh^2 (\R^3_+) }
 \| u_l(\tau) \|_{\Lv^q \Lh^2 (\R^3_+) }  \nonumber \\
 &
 \leq 
 C   \| u_0 \|_{X^s(\R^3_+)   } ^2 (1+\tau)^{-1},     \nonumber      \\
  \|\p_3(u_k u_l) (\tau)  \|_{\Lv^q \Lh^1 (\R^3_+)  }   & \leq C 
  \|\p_3 u_k(\tau)  \|_{\Lv^{\infty }\Lh^2 (\R^3_+) } 
  \| u_l(\tau) \|_{\Lv^q \Lh^2 (\R^3_+) }         \nonumber    \\
    &  \leq
    C  
    \| u_0 \|_{X^s(\R^3_+)   } ^2
    (1+\tau)^{-1},   \nonumber  \\
 \|\Grad^{\alpha} (u_k u_l) (\tau)  \|_{\Lv^q \Lh^p (\R^3_+)  }   & \leq 
    C  \| \Grad^{\alpha}u_k(\tau)  \|_{ \Lv^{
    \infty  }\Lh^{\f{2p}{2-p}} (\R^3_+)}
    \| u_l(\tau) \|_{\Lv^q \Lh^2 (\R^3_+)  }    \nonumber \\
    & \leq  
    C \| u_0 \|_{X^s(\R^3_+)   } ^2
    (1+\tau)^{ -(1-\f{2-p}{2p})-\f{|\alphah|}{2}   } (1+\tau)^{-\f{1}{2}}   \nonumber  \\
    & = 
    C   \| u_0 \|_{X^s(\R^3_+)   } ^2 
    (1+\tau)^{ -(2- \f{1}{p} )-\f{|\alphah|}{2}   }.   \nonumber 
\end{align}
By the above estimates and employing the same calculations as in Lemma \ref{decay-duhamel-leq-2}, it holds for any $1\leq p < 2, 2\leq q \leq \infty$ that 
\begin{align*}
  &  \left\| \Grad^{\alpha} \int_0^t   e^{(t-\tau) \Deltah } \div_{\rm h} (\uh \otimes \uh)(\tau)     d \tau \right\|_{\Lv^q \Lh^p (\R^3_+)  }    \\
  & \quad \quad   \quad \quad   
    \leq 
    C
    \| u_0 \|_{X^s(\R^3_+)   } ^2
    (1+t)^{ -(1-\f{1}{p}) -\f{1+|\alphah|}{2}  } \log(2+t). 
\end{align*}

As in Lemma \ref{decay-duhamel-leq-2}, we will not give a detailed proof for the remaining integrals, since they share similar lines as Lemma \ref{decay-duhamel-2} together with the modifications indicated above.
\end{proof}

With Lemma \ref{decay-duhamel-6.2-lem} and the linear analysis at hand, one may repeat step by step the same arguments from Subsection \ref{sec-6.1}. The proof of Theorem \ref{main-theorem} is finished completely.

\section{Proof of Theorem \ref{main-theorem-2} }\label{proof-thm-1.3}
This section is devoted to the proof of Theorem \ref{main-theorem-2}. 
We first consider the estimate for $u_3$
since it is slightly easier than that of $\uh$.
We make use of the solution formula of $u_3(t)$ obtained in Corollary \ref{lemm:sol-form-u}:
\begin{align}             
    u_3(t,x) ={}&  
    U e ^{t \Delta_{\rm h}} (u_{0,3}- S \cdot u_{0,\rm h} )(x)\\
    &
    +
    V^{(+)}
    \int_0^t 
    e^{(t - \tau)\Deltah}
    \sum_{\ell,m=1}^2
    {\partial_{x_\ell}{S_{x_m}}}
    (u_{\ell}u_m)(\tau) d\tau \\
    &{}
    + 
   \left(V^{(-)} +  T \right)
    \int_0^t 
    e^{(t - \tau)\Deltah}
    |\nablah|
    (u_3(\tau)^2) d\tau  \\
    &
    -
    \left( 
    W -1-T  \right)     
    \int_0^t 
    e^{(t - \tau)\Deltah}
   \div_{\rm h}
    (\uh u_3)(\tau)   d\tau. 
\end{align}
In view of the linear analysis \eqref{enh-decay-1} in Lemma \ref{lemm-enh-decay}, we see
\begin{align} \label{pro-1.3-1}
 \|  U e ^{t \Delta_{\rm h}} (u_{0,3}- S \cdot u_{0,\rm h} )   \|_{ \mathcal{L}^1_1(\R^3_+)  } 
  &  \leq C \| u_0 \|_{ \mathcal{L}^1_1(\R^3_+)  } 
\end{align}
for all $t>0$.
Next, we use Lemma \ref{bdd-VWT}, Lemma \ref{lemm-heat}, Theorem \ref{main-theorem} and the embedding $L^1(\R^3_+) \hookrightarrow \mathcal{L}^1_{\infty}(\R^3_+)  $ to infer 
\begin{align*}
& \left\|   V^{(+)}
    \int_0^t 
    e^{(t - \tau)\Deltah}
    \sum_{\ell,m=1}^2
    {\partial_{x_\ell}{S_{x_m}}}
    (u_{\ell}u_m)(\tau) d\tau  \right\|_{ \mathcal{L}^1_1(\R^3_+)  }   \\
    &  \quad \quad 
 \leq C \int_0^t 
 \left\| 
    e^{(t - \tau)\Deltah}
    \sum_{\ell,m=1}^2
    {\partial_{x_\ell}{S_{x_m}}}
    (u_{\ell}u_m)(\tau)  \right\| _{ \mathcal{L}^1_1(\R^3_+)  }    d\tau         \\
    &   \quad  \quad
        \leq C  \sum_{\ell,m=1}^2  \int_0^t 
        (t-\tau)^{-\f{1}{2}}
        \|  (u_{\ell}u_m)(\tau)  \| _{ \mathcal{L}^1_{\infty}(\R^3_+)  }   
        d\tau
        \\
    &   \quad    \quad 
    \leq 
    C   \sum_{\ell,m=1}^2    \int_0^t 
      (t-\tau)^{-\f{1}{2}}  
      \|  (u_{\ell}u_m)(\tau)  \|_{ L^1(\R^3_+) }
    d\tau
    \\
    & \quad  \quad 
    \leq 
    C  
  \int_0^t 
      (t-\tau)^{-\f{1}{2}}   \| \uh(\tau) \|^2_{ L^2(\R^3_+) } 
      d\tau  \\
       & \quad  \quad 
    \leq 
    C  \| u_0 \|_{X^s(\R^3_+)   } ^2 
  \int_0^t 
      (t-\tau)^{-\f{1}{2}}  (1+\tau)^{-1}
      d\tau  \\
    &    \quad  \quad 
    \leq 
    C  \| u_0 \|_{X^s(\R^3_+)   },
\end{align*}
where in the last step we used the smallness of $\| u_0 \|_{X^s(\R^3_+)   }$ and the basic fact that
\begin{align*}
    &  \int_0^t 
      (t-\tau)^{-\f{1}{2}}  (1+\tau)^{-1}
      d\tau  \leq C \quad  \text{  for all }    t>0. 
\end{align*}
Indeed,
\begin{align*}
   \int_0^t 
      (t-\tau)^{-\f{1}{2}}  (1+\tau)^{-1}
      d\tau    & \leq 
      \int_0^{ \f{t}{2} } 
      (t-\tau)^{-\f{1}{2}}  (1+\tau)^{-1}
      d\tau  
      + 
      \int_{ \f{t}{2} }^t
      (t-\tau)^{-\f{1}{2}}  (1+\tau)^{-1}
      d\tau 
      \\
    &   \leq 
    C \left( 
    t^{-\f{1}{2}  } \log(2+t)
    +(1+t)^{-1}  t^{\f{1}{2}  }
    \right)
    \\
    &   \leq     C 
\end{align*}
if $t \geq 1$;
\begin{align*}
 \int_0^t 
      (t-\tau)^{-\f{1}{2}}  (1+\tau)^{-1}
      d\tau 
      \leq 
       \int_0^t   (t-\tau)^{-\f{1}{2}} 
        d\tau 
        \leq C  t^{\f{1}{2}  } 
        \leq C 
\end{align*}
if $0<t < 1$. 

In the same spirit, we see
\begin{align*}
& \left\|   \left(V^{(-)} +  T \right)
    \int_0^t 
    e^{(t - \tau)\Deltah}
    |\nablah|
    (u_3(\tau)^2) d\tau \right\|_{ \mathcal{L}^1_1(\R^3_+)  }   \\
    &  \quad \quad 
 \leq C \int_0^t 
 \left\| 
    e^{(t - \tau)\Deltah}
   |\nablah|
    (u_3(\tau)^2)  \right\| _{ \mathcal{L}^1_1(\R^3_+)  }    d\tau         \\
    &   \quad  \quad
        \leq C  \int_0^t 
        (t-\tau)^{-\f{1}{2}}
        \|  (u_{3} u_3)(\tau)  \| _{ \mathcal{L}^1_{\infty}(\R^3_+)  }   
        d\tau
        \\
    &   \quad    \quad 
    \leq 
    C      \int_0^t 
      (t-\tau)^{-\f{1}{2}}  
      \|  (u_{3}  u_3)(\tau)  \|_{ L^1(\R^3_+) }
    d\tau
    \\
    & \quad  \quad 
    \leq 
    C  
  \int_0^t 
      (t-\tau)^{-\f{1}{2}}   \| u_3(\tau) \|^2_{ L^2(\R^3_+) } 
      d\tau  \\
       & \quad  \quad 
    \leq 
    C  \| u_0 \|_{X^s(\R^3_+)   } ^2 
  \int_0^t 
      (t-\tau)^{-\f{1}{2}}  (1+\tau)^{-\f{3}{2}}
      d\tau  \\
    &    \quad  \quad 
    \leq 
    C  \| u_0 \|_{X^s(\R^3_+)   }
\end{align*}
and 
\begin{align*}
& \left\|   \left( 
    W -1-T  \right)     
    \int_0^t 
    e^{(t - \tau)\Deltah}
   \div_{\rm h} 
    (\uh u_3)(\tau)   d\tau \right\|_{ \mathcal{L}^1_1(\R^3_+)  }   \\
    &  \quad \quad 
 \leq C \int_0^t 
 \left\| 
    e^{(t - \tau)\Deltah}
   \nablah
    (\uh u_3)(\tau)   \right\| _{ \mathcal{L}^1_1(\R^3_+)  }    d\tau         \\
    &   \quad  \quad
        \leq C  \int_0^t 
        (t-\tau)^{-\f{1}{2}}
        \| (\uh u_3)(\tau)   \| _{ \mathcal{L}^1_{\infty}(\R^3_+)  }   
        d\tau
        \\
    &   \quad    \quad 
    \leq 
    C      \int_0^t 
      (t-\tau)^{-\f{1}{2}}  
      \|  (\uh u_3)(\tau)   \|_{ L^1(\R^3_+) }
    d\tau
    \\
    & \quad  \quad 
    \leq 
    C  
  \int_0^t 
      (t-\tau)^{-\f{1}{2}}   \| u_3(\tau) \|_{ L^2(\R^3_+) } 
       \| \uh (\tau) \|_{ L^2(\R^3_+) } 
      d\tau  \\
       & \quad  \quad 
    \leq 
    C  \| u_0 \|_{X^s(\R^3_+)   } ^2 
  \int_0^t 
      (t-\tau)^{-\f{1}{2}}  (1+\tau)^{-\f{5}{4}}
      d\tau  \\
    &    \quad  \quad 
    \leq 
    C  \| u_0 \|_{X^s(\R^3_+)   } . 
\end{align*}
Combining the above estimates and \eqref{pro-1.3-1}, one gets \eqref{decay-u3-L1L1} upon invoking the embedding 
$\mathcal{L}^1_1(\R^3_+)   \hookrightarrow   L^1(\R^3_+)  $.

Next, we consider the estimate for $\uh$.
On this case, it suffices to consider $\mathcal{D}_1^{\rm h}[u](t)$ because the other terms $\mathcal{D}_m^{\rm h}[u](t)$ ($m=2,3,4,5$) can be estimated similarly to the above argument in the space $\mathcal{L}^1_1(\R^3_+)$. Using the divergence-free condition, it holds 
\begin{align}
 &   \| \mathcal{D}_1^{\rm h}[u](t)\|
    _{L^1(\R^3_+ ) }  \\
&   \quad  \leq
    C
    \int_0^t
    \| \partial_{x_3}(u_3\uh)(\tau) \|_{L^1 (\R^3_+ )}
    d\tau\\ 
 &  \quad \leq
    C
    \int_0^t
    \Big(\| \div_{\rm h} \uh (\tau)\|_{L^2(\R^3_+ )}\| \uh(\tau)\|_{L^2(\R^3_+ )} + \| u_3(\tau) \|_{L^2 (\R^3_+ )}\| \partial_{x_3}\uh(\tau)\|_{L^2 (\R^3_+ ) }  \Big)
    d\tau\\ 
  &  \quad   \leq
    C
    \| u_0 \|_{X^s(\R^3_+)}^2
    \int_0^t
    (1+\tau)^{-\frac{5}{4}}
    d\tau\\ 
  &  \quad   \leq 
    C\| u_0 \|_{X^s(\R^3_+)}.
\end{align}
Thus, we complete the proof of Theorem \ref{main-theorem-2}.

{\bf{Acknowledgements}} 
M. Fujii was supported by Grant-in-Aid for Research Activity Start-up, Grant Number JP23K19011. 
Y. Li was supported by Natural Science Foundation of Anhui Province under grant number 2408085MA018, Natural Science Research Project in Universities of Anhui Province under grant number 2024AH050055, National Natural Science Foundation of China under grant number 12001003; he sincerely thanks Professor Yongzhong Sun for patient guidance and encouragement.

{\bf{Data Availability}} Data sharing is not applicable to this article as no datasets were generated or analyzed
during the current study.

{\bf{Conflicts of interest}} All authors certify that there are no conflicts of interest for this work.

\appendix

\def\thesection{\Alph{section}}
\section{Tool box for the Littlewood--Paley theory}
To make our paper self-contained and for the convenience of the reader, we recall the definition of Littlewood--Paley decomposition and some basic lemmas from Fourier analysis. They are extensively used in the paper without pointing out each time. The details of proof could be found in the nice monograph of Bahouri et al. \cite{Bah-Che-Dan-11}. 

To begin with, we introduce a homogeneous Littlewood–-Paley decomposition in $\R^d$. To this end, we fix two radial non-increasing functions $(\varphi,\chi)$ such that 
\begin{align*}
    &  \supp \, \chi \subset B(0,4/3),\,\,   \supp \, \varphi \subset \{ \xi \in \R^d: 3/4 \leq |\xi| \leq 8/3 \},         \\
    &   0\leq \chi,\varphi \leq 1,\,\, \varphi(\xi)=\chi(\xi/2)-\chi(\xi), \,\, \chi=1\,\,\text{on }\, B(0,3/4),  \\
    &   \sum_{j\in \mathbb{Z}} \varphi(2^{-j}\xi)=1 \,\,\text{  for any  } \xi \in \R^d \backslash\{0\}.
\end{align*}
The homogeneous dyadic operators are defined by
\[
\Delta_j f= \varphi (2^{-j}D)f =\mathscr{F}^{-1} \left(\varphi(2^{-j} \cdot) \mathscr{F}f   \right)
=2^{jd}h(2^{j}\cdot )\ast f
\]
with $h=\mathscr{F}^{-1}(\varphi)$. It then holds the homogeneous decomposition
\[
f= \sum_{j\in \mathbb{Z}}   \Delta_j f
\]
for any $f$ in tempered distribution $\mathscr{S}'(\R^d)$ modulo polynomials. In order to introduce the homogeneous Besov space, we consider a subspace of $\mathscr{S}'(\R^d)$, namely
\[
\mathscr{S}'_h(\R^d):= \{ f\in \mathscr{S}'(\R^d): \lim_{j\rightarrow -\infty}S_j f=0 \}
\]
with $S_j$ being the low frequency cut-off operators given by $S_j f:= \chi(2^{-j}D) f$. 

With the preliminary above, we state
\begin{df}
Let $s\in \R$ and $1\leq p,q \leq \infty$. The homogeneous Besov space $\dot{B}^{s}_{p,q}(\R^d)$ is defined via
\[
\dot{B}^{s}_{p,q}(\R^d):=\left\{ f\in \mathscr{S}'_h(\R^d): 
 \| f \|_{\dot{B}^{s}_{p,q}(\R^d)}=
 \left(  \sum_{j\in \mathbb{Z}}  2^{jsq} \|\Delta_j f \|^{q}_{L^p(\R^d)}  \right)^{\f{1}{q} }<\infty   \right \} .
\]
\end{df}

To proceed, we recall three useful lemmas as follows. The first one is the Bernstein-type lemma.
\begin{lemm}
Let $\mathcal{C}$ be an annulus and $B$ a ball. A positive constant $C$ exists such that for any nonnegative integer $k$, any couple $1\leq p\leq q \leq \infty$, and any function $u \in L^p(\R^d)$, we have
 \[
\sup_{|\alpha|=k}\| \p ^{\alpha}u \|_{L^q(\R^d)} \leq C^{k+1} \lambda^{ k+d(\f{1}{p} -\f{1}{q} ) } \| u \|_{L^p(\R^d)}
 \]
if $ \supp\, \hat{u} \subset \lambda B $;
 \[
C^{-k-1}  \lambda^{ k } \| u \|_{L^p(\R^d)}
 \leq
\sup_{|\alpha|=k}\| \p ^{\alpha}u \|_{L^p(\R^d)} \leq C^{k+1} \lambda^{ k } \| u \|_{L^p(\R^d)}
 \]
if $ \supp\, \hat{u} \subset \lambda \mathcal{C} $.
\end{lemm}

We report in the following lemma on the action of Fourier multipliers behaving like homogeneous functions of degree $m$.
\begin{lemm}\label{app-lemm-2}
Let $\mathcal{C}$ be an annulus, $m\in \R$ and $k=2[1+d/2]$. Let $\sigma$ be a $k$ times differentiable function on $\R^d \backslash \{0\}$ such that for any $\alpha \in (\mathbb{N}\cup \{ 0 \} )^d$ with $|\alpha|\leq k$, there exists a constant $C_{\alpha}$ such that 
\[
|\p^{\alpha} \sigma(\xi) | \leq C_{\alpha} |\xi|^{m-|\alpha|}, \,\,\text{  for any  }\xi \in \R^d.
\]
Then there exists a positive constant $C$ depending only on the constants $C_{\alpha}$ such that for any $1\leq p \leq \infty$ and any $\lambda>0$ we have 
\[
\| \sigma(D) u \|_{L^p(\R^d)} \leq C \lambda^{m} \| u \|_{L^p(\R^d)} 
\]
for any $u \in L^p(\R^d)$ with $ \supp\, \hat{u} \subset \lambda \mathcal{C} $.
\end{lemm}

The next lemma concerns the smoothing effect of heat flow.
\begin{lemm}
Let $\mathcal{C}$ be an annulus. Then there exist two positive constants $c,C$ such that for any $1\leq p \leq \infty$ and any couple $(t,\lambda)$ of positive real numbers, we have 
\[
\| e^{t \Delta}u \|_{L^p(\R^d)} \leq C e^{-ct\lambda^2} \| u \|_{L^p(\R^d)}
\]
if $ \supp\, \hat{u} \subset \lambda \mathcal{C} $.
\end{lemm}


\begin{thebibliography}{100}

\bib{Bah-Che-Dan-11}{book}{
   author={Bahouri, Hajer},
   author={Chemin, Jean-Yves},
   author={Danchin, Rapha\"{e}l},
   title={Fourier analysis and nonlinear partial differential equations},
   volume={343},
   publisher={Springer, Heidelberg},
   pages={xvi+523},
}

\bib{BM88}{article}{
   author={Borchers, Wolfgang},
   author={Miyakawa, Tetsuro}, 
   title={$L^2$ decay for the Navier--Stokes flow in halfspaces},
   journal={Math. Ann.},
   volume={282},
   date={1988},
  pages={139-155},
}



\bib{BrM18}{article}{
   author={Brezis, Ham},
   author={Mironescu, Petru}, 
   title={Gagliardo–Nirenberg inequalities and non-inequalities: The full story},
   journal={Ann. I.H.Poincar\'{e}--AN.},
   volume={35},
   date={2018},
  pages={1355-1376},
}




\bib{CDGG01}{article}{
   author={Chemin, Jean-Yves},
   author={Desjardins, B.},
   author={Gallagher, Isabelle},
   author={Grenier, Emmanuel},
   title={Fluids with anisotropic viscosity},
   journal={M2AN Math. Model. Numer. Anal.},
   volume={34},
   date={2000},
  pages={315-335},
}


\bib{CZ07}{article}{
   author={Chemin, Jean-Yves},
   author={Zhang, Ping},
   title={On the global wellposedness to the $3$-D incompressible anisotropic Navier--Stokes equations},
   journal={Comm. Math. Phys.},
   volume={272},
   date={2007},
  pages={529-566},
}







\bib{DWXZ21}{article}{
   author={Dong, Boqing},
   author={Wu, Jiahong},
   author={Xu, Xiaojing},
   author={Zhu, Ning},
   title={Stability and exponential decay for the $2$D anisotropic Navier--Stokes equations with horizontal dissipation},
   journal={J. Math. Fluid Mech.},
   volume={23},
   date={2021},
   pages={Paper No. 100, 11 pp},
}




\bib{FM01}{article}{
   author={Fujigaki, Yoshiko},
   author={Miyakawa, Tetsuro}, 
   title={Asymptotic profiles of nonstationary incompressible Navier--Stokes flows in the half-space},
   journal={Methods Appl. Anal.},
   volume={8},
   date={2001},
  pages={121-157},
}


\bib{F21}{article}{
   author={Fujii, Mikihiro}, 
   title={Large time behavior of solutions to the $3$D anisotropic Navier--Stokes equation},
   journal={Nonlinear Anal. Real World Appl.},
   volume={71},
   date={2023},
  pages={Paper No. 103821, 33 pp.},
}


\bib{GMS99}{article}{
   author={Giga, Yoshikazu}, 
   author={Matsui, Shin'ya},
   author={Shimizu, Yasuyuki},
   title={On estimates in Hardy spaces for the Stokes flow in a half space},
   journal={Math. Z.},
   volume={231},
   date={1999},
  pages={383-396},
}







\bib{Han10}{article}{
   author={Han, Pigong}, 
   title={Asymptotic behavior for the Stokes flow and Navier--Stokes equations in half spaces},
   journal={J. Differential Equations.},
   volume={249},
   date={2010},
  pages={1817-1852},
}

\bib{Han11}{article}{
   author={Han, Pigong}, 
   title={Decay results of solutions to the incompressible Navier--Stokes flows in a half space},
   journal={J. Differential Equations.},
   volume={250},
   date={2011},
  pages={3937-3959},
}

\bib{Han16}{article}{
   author={Han, Pigong}, 
   title={Large time behavior for the nonstationary Navier--Stokes flows in the half-space},
   journal={Adv. Math.},
   volume={288},
   date={2016},
  pages={1-58},
}





\bib{HH12}{article}{
   author={Han, Pigong},
   author={He, Cheng},
   title={Decay properties of solutions to the incompressible magnetohydrodynamics equations in a half space},
   journal={Math. Methods Appl. Sci. },
   volume={35},
   date={2012},
  pages={1472-1488},
}




\bib{If02}{article}{
author={Iftimie, D.},
title={A uniqueness result for the Navier--Stokes equations with vanishing vertical viscosity},
journal={SIAM J. Math. Anal.},
volume={33},
date={2002},
pages={1483-1493},
}


\bib{IfPl06}{article}{
author={Iftimie, D.},
author={ Planas, G.},
title={Inviscid limits for the Navier-Stokes equations with Navier friction boundary conditions},
journal={Nonlinearity.},
volume={19},
date={2006},
pages={899-918},
}








\bib{JWY21}{article}{
   author={Ji, Ruihong},
   author={Wu, Jiahong}, 
   author={Yang, Wanrong},
   title={Stability and optimal decay for the $3$D Navier--Stokes equations with horizontal dissipation},
   journal={J. Differential Equations.},
   volume={290},
   date={2021},
  pages={57-77},
}


\bib{JTW23}{article}{
   author={Ji, Ruihong},
   author={Tian, Ling},
   author={Wu, Jiahong}, 
   title={$3$D anisotropic Navier--Stokes equations in
$\mathbb{T}^2 \times \R$: stability and large-time behaviour},
   journal={Nonlinearity.},
   volume={36},
   date={2023},
  pages={3219-3237},
}







\bib{K89}{article}{
   author={Kozono, Hideo}, 
   title={Global $L^n$-solution and its decay property for the Navier-Stokes equations in half-space $\R^n_+$},
   journal={J. Differential Equations.},
   volume={79},
   date={1989},
  pages={79-88},
}






\bib{L22}{article}{
   author={Li, Yang}, 
   title={Large time behavior of the solutions to $3$D incompressible MHD system with horizontal dissipation or horizontal magnetic diffusion},
   journal={Calc. Var. Partial Differential Equations.},
   volume={63},
   date={2024},
  pages={Paper No. 43, 65 pp},
}

 



\bib{LPZ20}{article}{
   author={Liu, Yanlin},
   author={Paicu, Marius},
   author={Zhang, Ping},
   title={Global well-posedness of $3$-D anisotropic Navier--Stokes system with small unidirectional derivative},
   journal={Arch. Ration. Mech. Anal.},
   volume={238},
   date={2020},
  pages={805-843},
}


\bib{LiuZh20}{article}{
   author={Liu, Yanlin},
   author={Zhang, Ping},
   title={Global well-posedness of $3$-D anisotropic Navier--Stokes system with large vertical viscous coefficient},
   journal={J. Funct. Anal.},
   volume={279},
   date={2020},
  pages={108736, 33 pp.},
}




\bib{Pai05}{article}{
   author={Paicu, Marius},
   title={\'{E}quation anisotrope de Navier--Stokes dans des espaces critiques},
   journal={Rev. Mat. Iber.},
   volume={21},
   date={2005},
  pages={179-235},
}


\bib{Pai0502}{article}{
   author={Paicu, Marius},
   title={Periodic Navier-Stokes equation with zero viscosity in one direction},
   journal={Comm. Partial Differential Equations.},
   volume={30},
   date={2005},
  pages={1107-1140},
}


\bib{PaiRo09}{article}{
   author={Paicu, Marius},
   author={Raugel, G.},
   title={Anisotropic Navier-Stokes equations in a bounded cylindrical domain},
   journal={London Math. Soc. Lecture Note Ser.},
   volume={364},
   date={2009},
  pages={146-184},
}





\bib{PZ11}{article}{
   author={Paicu, Marius},
   author={Zhang, Ping},
   title={Global solutions to the $3$-D incompressible anisotropic Navier--Stokes system in the critical spaces},
   journal={Comm. Math. Phys.},
   volume={307},
   date={2011},
  pages={713-759},
}



\bib{Ped}{article}{
   author={Pedlosky, Joseph},
   title={Geophysical Fluid Dynamics},
   journal={Berlin-Heidelberg-NewYork: Springer},
   date={1979},
}





\bib{U87}{article}{
   author={Ukai, Seiji},
   title={A solution formula for the Stokes equation in $\R^n_{+}$},
   journal={Comm. Pure Appl. Math.},
   volume={40},
   date={1987},
  pages={611-621},
}














\bib{XZ22}{article}{
   author={Xu, Li},
   author={Zhang, Ping}, 
   title={Enhanced dissipation for the third component of $3$D anisotropic Navier--Stokes equations},
   journal={J. Differential Equations.},
   volume={335},
   date={2022},
  pages={464-496},
}


\bib{Z09}{article}{
   author={Zhang, Ting},
   title={Global wellposed problem for the $3$-D incompressible anisotropic Navier--Stokes equations in an anisotropic space},
   journal={Comm. Math. Phys.},
   volume={287},
   date={2009},
  pages={211-224},
}


\bib{ZF08}{article}{
   author={Zhang, Ting},
   author={Fang, Daoyuan},
   title={Global wellposed problem for the $3$-D incompressible anisotropic Navier--Stokes equations},
   journal={J. Math. Pures Appl.},
   volume={90},
   date={2008},
  pages={413-449},
}







\end{thebibliography}
\end{document}